\documentclass[11pt]{article}

\usepackage{fullpage}
\usepackage{authblk}
\usepackage{natbib}
\usepackage{algorithm}
\usepackage{algpseudocode}
\usepackage{amsmath,amsthm,amsfonts}
\usepackage{amsmath}
\usepackage[subnum]{cases}
\usepackage[most]{tcolorbox}
\usepackage{mathrsfs}
\usepackage{amssymb}
\usepackage{color}
\usepackage{mathrsfs}
\usepackage{empheq}
\usepackage{enumitem}
\usepackage{bm}
\usepackage{multirow}
\usepackage{booktabs}
\usepackage{makecell}
\usepackage{graphicx}
\usepackage{subcaption}
\usepackage{comment}
\usepackage{cases}
\usepackage{appendix}
\usepackage{tikz}
\usetikzlibrary{arrows,shapes}
\usepackage[colorlinks= true, linkcolor=brown, citecolor=blue, urlcolor=black]{hyperref}
\usepackage{cleveref}
\usepackage{marginnote}
\usepackage{diagbox} 
\usepackage{bbm}

\numberwithin{equation}{section}

\theoremstyle{definition}

\newtheorem{theorem}{Theorem}[section]
\newtheorem{lemma}[theorem]{Lemma}

\newtheorem{condition}{Condition}
\newtheorem{prop}[theorem]{Proposition}
\newtheorem{defn}[theorem]{Definition}

\newcommand{\defi}{\overset{\triangle}{=}}

\usepackage{booktabs}
\usepackage{tikz}
\usepackage{pgfplots}
\usetikzlibrary{calc,intersections}
\usetikzlibrary{patterns}

\usetikzlibrary{shapes.geometric, arrows, positioning}

\tikzset{
	mybox/.style  = {draw, rectangle, minimum width=4cm, minimum height=0.8cm, text centered, text width=4.4cm,   
		font=\normalsize},
	box/.style  = {draw, rectangle, minimum width=2.0cm, minimum height=0.6cm, text centered, text width=3.0cm,   
		font=\normalsize},
	myarrow/.style = {line width=0.2pt, draw=black, -triangle 60, postaction={draw, line width=0.2pt, shorten >=10pt,-}}
}

\tikzstyle{arrow} = [->, >=stealth, -triangle 60]

\allowdisplaybreaks

\makeatletter
\newcommand{\leqnomode}{\tagsleft@true}
\newcommand{\reqnomode}{\tagsleft@false}
\makeatother

\begin{document}

\title{Numerical Construction of Quasi-Periodic Solutions for Nonlinear PDEs: II. Unbounded Perturbation}
\author[1,2]{Mingwei Fu}
\author[1,2]{Bin Shi\thanks{Corresponding author: \url{binshi@fudan.edu.cn} } }
\affil[1]{Center for Mathematics and Interdisciplinary Sciences, Fudan University, Shanghai 200433, China}
\affil[2]{Shanghai Institute for Mathematics and Interdisciplinary Sciences, Shanghai 200433, China}

\date\today

\maketitle

\begin{abstract}
Many nonlinear partial differential equations (PDEs) arising in physical models 
involve unbounded perturbations, in which nonlinear terms contain spatial 
derivatives. 
In operator terms, such perturbations can be expressed as compositions of 
unbounded operators with bounded nonlinear maps. 
In this paper, we propose alternating numerical schemes for constructing 
quasi-periodic solutions to this class of PDEs, including the nonlinear acoustic 
wave (NLAW) equation and the Korteweg–de Vries (KdV) equation. 
The resonant $Q$-equations allow us to establish a local diffeomorphism directly 
between the amplitudes and the shifted frequencies, bypassing the Birkhoff 
normal form and facilitating numerical implementation. 
To render the nonlinear perturbations bounded, we apply inverse differential 
operators as a preconditioning; this, however, also rescales the linear part by 
polynomial factors in the spatial mode $n$. 
Two main difficulties arise. 
The first is to control the measure of the resulting ``bad'' frequencies, since 
the bounds on the measures of the excluded frequency sets are enlarged by 
polynomial factors in $n$; we show that this measure nevertheless remains small. 
The second concerns resonances involving the perturbed eigenvalues of the 
linearized operator at large scales. 
Our key observation is that, when the dispersive frequencies grow polynomially 
in $n$ with degree greater than one, such resonances cannot occur if the 
magnitude of the temporal index lies below a threshold proportional to $|n|$ 
raised to one less than that degree. 
It therefore suffices to consider the complementary regime, in which the growth 
induced by the unbounded perturbation can be controlled in terms of the temporal 
index. 
At large scales, an additional complication arises: for a given spatial index 
$n$, all temporal indices in a box may be ``bad.'' 
We show, however, that at each scale there is at most one such spatial index. 
Consequently, the resulting counting loss is only polynomial and can be absorbed 
into the exponential-logarithmic threshold, allowing us once again to retain 
control over the measures of the excluded frequency sets. 
\end{abstract}

\section{Introduction}
\label{sec: intro}

Many nonlinear partial differential equations (PDEs) can be viewed as 
infinite-dimensional Hamiltonian systems. 
A particularly challenging class consists of PDEs 
with~\emph{unbounded perturbations}, that is, whose nonlinear vector fields 
involve spatial derivatives. 
A classical example is the Korteweg–de Vries (KdV) equation
\begin{equation}
    \label{eqn: kdv} 
    u_t + u_{xxx} + \varepsilon (u^2)_x = 0, 
\end{equation}
where $0<\varepsilon \ll 1$ is a perturbation parameter. 
The KdV equation~\eqref{eqn: kdv} models weakly nonlinear, long waves in shallow 
water~\citep{whitham1974linear}, and its dynamics include the formation and 
interaction of solitary waves, termed  ``solitons'' 
by~\citet{zabusky1965interaction}. 
The KdV equation is completely integrable and, under periodic boundary 
conditions, admits a convergent Birkhoff normal form near the origin, as 
established by~\citet{battig1996kdv}. 
This structure enables the direct construction of invariant tori carrying 
quasi-periodic solutions. 
Another physically important example is the nonlinear acoustic wave (NLAW) 
equation, whose three-dimensional version was introduced in the study of 
acoustic turbulence~\citep{zakharov1965weak, zakharov1970specturm}. 
The associated turbulence spectra were further investigated numerically 
by~\citet{kochurin2024acoustic}. 
In one spatial dimension, the NLAW equation reads
\begin{equation}
    \label{eqn: nlaw} 
    u_{tt} - u_{xx} + u_{xxxx} - \varepsilon (u^2)_{xx} = 0.
\end{equation}
The positive sign of the fourth-order term ensures that the linearization 
of~\eqref{eqn: nlaw} is well posed. 
Thus, the NLAW equation~\eqref{eqn: nlaw} is of the ``good'' Boussinesq type, in 
contrast to the classical ``bad'' Boussinesq equation~\citep{boussinesq1872theorie}, 
in which the fourth-order term has the opposite sign. 
The latter originated as a model for weakly nonlinear, long waves in shallow 
water and is closely related to the KdV equation~\eqref{eqn: kdv}, which arises 
as its unidirectional approximation.

Originating in the pioneering work 
of~\citet{kolmogorov1954preservation, arnold1963proof} 
and~\citet{moser1962invariant}, KAM theory provides a fundamental framework for 
studying the persistence of quasi-periodic motions under small perturbations. 
Since the late 1980s, it has been extended to infinite-dimensional systems, 
including Hamiltonian PDEs. 
For PDEs with unbounded perturbations, foundational results were obtained 
by~\citet{kuksin2000analysis} and~\citet{kappeler2003kdv} in their studies of 
perturbed KdV equations. 
A key technical ingredient in this approach is the so-called Kuksin's lemma, 
which provides estimates for the small-divisor problems arising from homological 
equations with large variable coefficients. 
Suppose that the linear frequencies grow like $n^{d}$ as $n \rightarrow \infty$ 
and that the perturbation vector field maps $H^{p}$ into $H^{q}$. 
The loss of regularity is then measured by $\delta = p - q$. 
The classical form of Kuksin’s lemma applies in the range 
$\delta < d - 1$.~\citet{liu2011kam} extended this approach to the critical case 
$\delta = d - 1$ and applied it to derivative nonlinear Schr\"odinger equations 
with Dirichlet boundary conditions. 
Under periodic boundary conditions, where the double frequencies give rise to 
additional resonances,~\citet{liu2014kam} resolved these resonances through a 
finer analysis of the nonlinear perturbation by means of a Birkhoff normal 
form.~\citet{zhang2011kam} developed a related KAM theory for infinite-dimensional 
reversible systems with unbounded perturbations. 
It appears, however, that this method has not yet been applied to investigate the 
NLAW equation~\eqref{eqn: nlaw}.

\subsection{Fourier expansion of the NLAW equation}
\label{subsec: formulation}

As a representative case, we first consider the NLAW equation~\eqref{eqn: nlaw} 
subject to the periodic boundary condition
\begin{equation}
    \label{eqn: per-bdy}
    u(0, t) = u(2\pi, t). 
\end{equation}
Since $u=u(x, t)$ is $2\pi$-periodic in $x$, we expand it in a spatial Fourier 
series:
\begin{equation}
    \label{eqn: fourier-expan}
    u(x, t) = \sum_{n \in \mathbb{Z}} \hat{u}_n(t) e^{inx},
\end{equation}
where the Fourier coefficients are given by
\begin{equation}
    \label{eqn: fourier-coeff}
    \hat{u}_n(t) = \frac{1}{2\pi} \int_{0}^{2\pi} u(x, t) e^{-inx} dx. 
\end{equation}
Since the solutions of~\eqref{eqn: nlaw} are real-valued, it follows directly 
from~\eqref{eqn: fourier-coeff} that $\overline{\hat{u}_n} = \hat{u}_{-n}$ for 
any $n \in \mathbb{Z}$. 
Hence the coefficients with negative index are determined by those with positive 
index, and~\eqref{eqn: fourier-expan} can be rewritten as
\[
    u(x, t) = \hat{u}_{0}(t) + 2\mathrm{Re} \left( \sum_{n=1}^{\infty} \hat{u}_n(t) e^{inx} \right),
\]
so that it suffices to consider $n \in \mathbb{N}$. 
The associated linear equation, obtained by setting $\varepsilon = 0$, reads
\begin{equation}
    \label{eqn: lin-nlaw}
    u_{tt} - u_{xx} + u_{xxxx} = 0.
\end{equation}
Substituting~\eqref{eqn: fourier-expan} into~\eqref{eqn: lin-nlaw}, we obtain the 
dispersion relation
\begin{equation}
    \label{eqn: disper-nlaw}
    \varpi_{n}^2 = n^2 + n^4, 
\end{equation}
for any $n \in \mathbb{Z}$. 
We select a finite set of $m$ distinct wavenumbers in the low-frequency range, 
referred to as the tangential modes,\footnote{The zero mode is not selected as a 
tangential site, since its linear frequency vanishes by~\eqref{eqn: disper-nlaw} 
and it therefore provides no oscillatory degree of freedom. 
As will be seen in the construction of~\Cref{subsec: numerical-scheme}, this mode 
remains identically zero throughout the iteration, so that the resulting solutions 
have zero spatial mean.}
\begin{equation}
    \label{eqn: tan-nlaw}
    \mathcal{N} = \left\{ 0 < n_1 < n_2 < \ldots < n_m \leq \frac{M_0^{\frac13}}{2} - 1 \right\} \subset \mathbb{N}^+,
\end{equation}
where $M_0 = M_0(\varepsilon)$ is a truncation parameter depending on the 
perturbation parameter $\varepsilon$, to be specified in~\eqref{eqn: M0}. 
By~\eqref{eqn: disper-nlaw}, the linear frequencies of the tangential modes are 
$\omega_j := \varpi_{n_j}$ for $j=1, \ldots, m$, and we collect them in the 
frequency vector $\omega = (\omega_1, \ldots, \omega_m) \in \mathbb{R}^m$. 
A solution of~\eqref{eqn: lin-nlaw} supported on these modes is the linear 
superposition
\begin{equation}
    \label{eqn: lin-soln-nlaw}
    u(x, t) = 2 \sum_{j=1}^{m} \sqrt{I_{0, j}} \cos\left( \omega_j t + \theta_{0, j} + n_j x \right),  
\end{equation}
where $I_0 = (I_{0,1}, \ldots, I_{0,m}) \in \mathbb{R}_+^m$ and 
$\theta_0 = (\theta_{0,1}, \ldots, \theta_{0, m}) \in \mathbb{T}^m$ denote the 
initial action and angle variables, respectively.


For convenience, we set $a_j = \sqrt{I_{0,j}}$ for $j=1,\ldots, m$, and assume 
that the amplitudes of the selected tangential modes are bounded away from zero, 
namely 
\[
    \min\limits_{1 \leq j \leq m} a_j \geq \delta_0 > 0.
\] 
Let $e_j$ denote the $j$-th standard basis vector of $\mathbb{R}^m$. 
Then,~\eqref{eqn: lin-soln-nlaw} can be rewritten as
\begin{equation}
    \label{eqn: lin-soln-nlaw-1}
    u(x, t) = 2 \sum_{j=1}^{m} a_j \cos \left( \langle e_j, \; \omega t + \theta_{0} \rangle + n_j x \right). 
\end{equation}
We define the resonant set as
\begin{equation}
    \label{eqn: resonant-nlaw}
    \mathcal{S} := \left\{ (n_j, e_j) \in \mathbb{N}^+ \times \mathbb{Z}^{m} ~\big|~ j =1, \ldots, m, ~ n_j \in \mathcal{N} \right\}. 
\end{equation}
Our aim is to construct quasi-periodic solutions of the NLAW 
equation~\eqref{eqn: nlaw} of the form
\begin{equation}
    \label{eqn: nonlin-soln-nlaw}
    u(x, t) = \sum_{k \in \mathbb{Z}^m} \hat{u}(0, k) \cos \left( \langle k, \; \omega't + \theta_0 \rangle \right) + 2 \sum_{n=1}^{\infty} \sum_{k \in \mathbb{Z}^m} \hat{u}(n, k) \cos \left( \langle k, \; \omega't + \theta_0 \rangle + nx \right),
\end{equation}
with real coefficients $\hat{u}(n, k)$ and a drifted frequency vector 
$\omega' \in \mathbb{R}^m$ close to $\omega$. 
The Fourier coefficients on the resonant set $\mathcal{S}$ are prescribed by 
$\hat{u}(n_j, e_j) = a_j$ for $j=1, \ldots, m$, whereas the frequency vector 
$\omega'$ and the remaining Fourier coefficients $\hat{u}(n,k)$ for any 
$(n,k) \in \mathcal{Z} \setminus \mathcal{S}$ are determined successively 
through an iteration.

\subsection{Numerical scheme}
\label{subsec: numerical-scheme}

In this section, we propose an alternating scheme that progressively deforms the 
linear solution~\eqref{eqn: lin-soln-nlaw-1} into the nonlinear quasi-periodic 
solution~\eqref{eqn: nonlin-soln-nlaw}. 
Besides serving as a computational procedure, the scheme provides the framework 
for the rigorous existence proof. 
For computational convenience, we extend the Fourier coefficients 
$\hat{u}(n, k)$ from $\mathbb{N} \times \mathbb{Z}^m$ to the whole lattice 
$\mathbb{Z}^{1+m}$, and require them to be real and to satisfy the symmetry 
relation $\hat{u}(n,k) = \hat{u}(-n,-k)$ for any $(n, k) \in \mathbb{Z}^{1+m}$. 
Together, these conditions ensure that $u$ is real-valued. 
We then use the full Fourier expansion
\begin{equation}
    \label{eqn: nonlin-soln-nlaw-full}
    u(x, t) = \sum_{(n,k) \in \mathbb{Z}^{1+m}} \hat{u}(n, k) e^{i \left( \langle k, \; \omega't + \theta_0 \rangle + nx \right)}.
\end{equation}
Under the imposed symmetry, this representation agrees with the cosine 
expansion~\eqref{eqn: nonlin-soln-nlaw}. 
In this formulation, the resonant set~\eqref{eqn: resonant-nlaw} is replaced by 
its symmetric extension, consisting of the $2m$ lattice points
\begin{equation}
    \label{eqn: resonant-double-nlaw}
    \mathcal{S} := \left\{ (n_j, e_j), (-n_j, -e_j) \in \mathbb{Z}^{1+m} ~\big|~ j =1, \ldots, m, ~ n_j \in \mathcal{N} \right\}, 
\end{equation}
where $n_1, \ldots, n_m \in \mathbb{N}^+$ are the spatial modes chosen to be fixed 
in~\eqref{eqn: tan-nlaw}. 
On the resonant set $\mathcal{S}$, the Fourier coefficients are prescribed by 
the amplitudes, $\hat{u}(n_j, e_j) = \hat{u}(-n_j, - e_j) = a_j$ for 
$j=1,\ldots, m$, in accordance with~\eqref{eqn: resonant-nlaw}. 
Accordingly, the amplitude vector becomes 
$a = (a_1, \ldots, a_m; a_1, \ldots, a_m) \in \mathbb{R}^{2m}$, which, owing to 
the symmetry, carries only $m$ independent parameters. 
Substituting the ansatz~\eqref{eqn: nonlin-soln-nlaw-full} into the NLAW 
equation~\eqref{eqn: nlaw} and matching the coefficients of 
$e^{i \left( \langle k, \; \omega't + \theta_0 \rangle + nx \right)}$, we obtain 
the lattice equation
\begin{equation}
    \label{eqn: lattice-nlaw-old}
    \left( - \langle k, \omega' \rangle^2 + \varpi_n^2  \right) \hat{u}(n, k) + \varepsilon n^2 \widehat{u^2}(n, k) = 0,
\end{equation}
for any multi-index $(n, k) \in \mathbb{Z}^{1+m}$, where $\widehat{u^2}$ denotes 
the Fourier coefficient of the nonlinear term and is given by the convolution
\begin{equation}
    \label{eqn: pert-nlaw}
    \widehat{u^2}(n, k) = \sum_{n_1 + n_2 = n} \sum_{k_1 + k_2 = k} \hat{u}(n_1, k_1) \hat{u}(n_2, k_2). 
\end{equation}
This convolution preserves both the reality and the symmetry of the Fourier 
coefficients. 
Because the nonlinear term carries a factor of $n^2$, the lattice 
equation~\eqref{eqn: lattice-nlaw-old}, restricted to $\{0\} \times \mathbb{Z}^m$, 
reduces to the linear equations
\begin{equation}
    \label{eqn: 0-k-mode}
    -\langle k, \omega' \rangle^2 \hat{u}(0, k) = 0.
\end{equation}
After excluding the resonant frequencies specified in~\Cref{subsec: main-state}, 
we have $\langle k, \omega' \rangle \neq 0$ for any $k \neq 0$. 
Hence,~\eqref{eqn: 0-k-mode} implies that $\hat{u}(0, k) = 0$ for all 
$k \neq 0$, leaving only the constant mode $\hat{u}(0, 0)$ undetermined. 
Let $b_0 := \hat{u}(0,0)$ denote the constant mode.
Writing $u=b_0+v$, where $v$ has zero spatial mean, we obtain
\begin{equation}
    \label{eqn: nlaw-non-zero}
    v_{tt} - (1+2b_0 \varepsilon)v_{xx} + v_{xxxx} - \varepsilon (v^2)_{xx} = 0. 
\end{equation}
Thus, a nonzero constant mode modifies only the coefficient of the second-order 
spatial derivative, while leaving the form of the equation unchanged. 
In what follows, we restrict attention to the zero-mean case $b_0 = 0$; the case 
$b_0 \neq 0$ is discussed at the end of the paper.

Let $\mathcal{Z} := \left( \mathbb{Z} \setminus \{0\} \right) \times \mathbb{Z}^{m}$. 
Unlike the bounded perturbation setting considered 
in~\citet{Fu2026numerical-pde-1}, the present problem requires us to account for 
the factor $n^2$ in the nonlinear term. 
Since $n \neq 0$, we divide~\eqref{eqn: lattice-nlaw-old} by $n^2$ and work with 
the equivalent lattice equation
\begin{equation}
    \label{eqn: lattice-nlaw}
    \frac{- \langle k, \omega' \rangle^2 + \varpi_n^2 } {n^2} \cdot \hat{u}(n, k) + \varepsilon \widehat{u^2}(n, k) = 0.
\end{equation}
Setting $\hat{X} := \widehat{u^2}$, we rewrite~\eqref{eqn: lattice-nlaw} in 
vector form as
\begin{equation}
    \label{eqn: lattice-vector-nlaw}
    \mathcal{D}\hat{u} + \varepsilon \hat{X} = 0,
\end{equation}
where $\mathcal{D}$ is a diagonal operator indexed by $\mathcal{Z}$, with 
diagonal entries 
\[
    \mathcal{D}(n, k) = \frac{- \langle k, \omega' \rangle^2 + \varpi_n^2}{n^2}. 
\]   
We split~\eqref{eqn: lattice-vector-nlaw} according to the decomposition 
$\mathcal{Z} = \mathcal{S} \cup \left( \mathcal{Z} \setminus \mathcal{S} \right)$, 
where $\mathcal{S}$ is the resonant set given 
in~\eqref{eqn: resonant-double-nlaw}. 
On the resonant set $\mathcal{S}$, we prescribe $\hat{u} |_{\mathcal{S}} = a$ 
and set $\hat{X}_q := \hat{X} |_{\mathcal{S}}$. 
Denoting the restrictions of $\mathcal{D}$ to the resonant set $\mathcal{S}$ by 
$\mathcal{D}_q$, we obtain the $Q$-equations
\begin{equation}
    \label{eqn: q-eqn-nlaw}
    \mathcal{D}_q a + \varepsilon \hat{X}_q = 0, 
\end{equation}
where $\mathcal{D}_q(n_j, e_j) = \mathcal{D}_q(-n_j, -e_j) = (- \omega_j'^2 + \omega_j^2)/n_j^2$ 
for $j=1, \ldots, m$, serving to determine the frequency vector $\omega'$. 
With a slight abuse of notation, we write $- \omega'^2 + \omega^2$ for the 
vector indexed by $\mathcal{S}$ whose entries at $(n_j, e_j)$ and $(-n_j, -e_j)$ 
equal $ - \omega_j'^2 + \omega_j^2 $, and similarly $W_q$ for the diagonal 
operator on $\mathcal{S}$ whose entries at $(n_j, e_j)$ and $(-n_j, -e_j)$ 
equal $n_j^2$. 
Let $A := \mathrm{diag}(a)$, then we can rewrite the 
$Q$-equations~\eqref{eqn: q-eqn-nlaw} as
\begin{equation}
    \label{eqn: q-eqn-nlaw-1}
    (- \omega'^2 + \omega^2) + \varepsilon A^{-1}W_q \hat{X}_q = 0.
\end{equation}
Similarly, on the non-resonant set $\mathcal{Z} \setminus \mathcal{S}$, we 
define $\hat{u}_p = \hat{u} |_{ \mathcal{Z} \setminus \mathcal{S}}$ and 
$\hat{X}_p = \hat{X} |_{ \mathcal{Z} \setminus \mathcal{S}}$ and denote the 
corresponding restrictions of $\mathcal{D}$ by $\mathcal{D}_p$. 
Restricting~\eqref{eqn: lattice-vector-nlaw} to the non-resonant set 
$\mathcal{Z} \setminus \mathcal{S}$ yields the $P$-equations
\begin{equation}
    \label{eqn: p-eqn-nlaw}
    \mathcal{D}_p\hat{u}_p + \varepsilon \hat{X}_p = 0,
\end{equation}
which determine the unknown nonresonant Fourier coefficients $\hat{u}_p$. 
Since both $\hat{X}_q$ and $\hat{X}_p$ depend on $\hat{u}_p$ while $\mathcal{D}_p$ 
depends on $\omega'$, the two systems are coupled and must be solved 
simultaneously for the unknowns $(\omega', \hat{u}_p)$. 
The $Q/P$ decomposition forms the basis of the alternating numerical scheme 
described below.

Since the $P$-equations~\eqref{eqn: p-eqn-nlaw} constitute an 
infinite-dimensional nonlinear algebraic system, we solve them numerically by a 
dimension-enlarged Newton scheme. 
To this end, we define the nonlinear operator $F$, acting on the nonresonant 
Fourier coefficients $\hat{u}_p$, as
\begin{equation}
    \label{eqn: total-vector}
    F(\hat{u}_p, \omega'; a) = \mathcal{D}_p\hat{u}_p + \varepsilon \hat{X}_p. 
\end{equation}
Writing $D:=\mathcal{D}_p$, the linearization of $F$ with respect to 
$\hat{u}_p$ at fixed $\omega'$ is $T = D + \varepsilon S$, where $S$ is a 
Laurent operator (see~\citet{trefethen2005spectra}) with entries 
$S\left( (n,k), (n', k') \right) = 2\hat{u}(n - n', k - k')$. 
Let $M=M(\varepsilon)$, to be specified in~\eqref{eqn: M}, and set the 
truncation scale $N_r := M^{r+1}$. 
Based on the $Q$-equations~\eqref{eqn: q-eqn-nlaw-1} and the 
$P$-equations~\eqref{eqn: p-eqn-nlaw}, we propose the following alternating 
numerical scheme: at each iteration, we first update the frequency vector and 
then the non-resonant coefficients via a truncated Newton step,
\begin{subequations}
    \label{eqn: alternating-numerical-nlaw}
    \begin{empheq}[left=\empheqlbrace]{align} 
       &\omega^{(r+1)} =  \sqrt{\omega^2 + \varepsilon A^{-1} W_q \hat{X}_q\left(\hat{u}_p^{(r)}; a\right)},                                                                                                                                                                   \label{eqn: alternating-frequency-nlaw}                \\
       &\hat{u}_p^{(r+1)} = \hat{u}_p^{(r)} - \left[(T + \varepsilon B)_{N_{r+1}}\left( \hat{u}_p^{(r)}, \omega^{(r+1)};a \right) \right]^{-1} F\left( \hat{u}_p^{(r)}, \omega^{(r+1)};a \right).         \label{eqn: alternating-coefficient-nlaw} 
    \end{empheq}        
\end{subequations} 
where the correction operator $B$ accounts for the dependence of the frequency 
vector on $\hat{u}_p$ through~\eqref{eqn: alternating-frequency-nlaw} and is 
essential for the convergence of the iteration; it is given by
\begin{equation}
    \label{eqn: operator-b-nlaw}
    B = - \langle k, \omega' \rangle \cdot \frac{\partial \langle k, \omega'^{-1} \odot A^{-1}W_q\hat{X}_q \rangle}{\partial \hat{u}_p} W_p^{-1} \hat{u}_p^{\top}, 
\end{equation}
where the symbol $\odot$ denotes componentwise multiplication. 
The restricted operator is defined by 
$(T + \varepsilon B)_N = P_N (T + \varepsilon B) P_N$, regarded as an operator 
on the range of $P_N$, where $P_N$ is the projection onto the lattice box 
$\Lambda^{\star}_{N} \times \Lambda^{m}_{N}$, 
\[
    (P_N \hat{u})(n, k) := 
    \left\{ 
    \begin{aligned}
        & \hat{u}(n, k),  && \mathrm{if}\; (n, k) \in \Lambda^{\star}_{N} \times \Lambda^{m}_{N}, \\
        & 0,              && \mathrm{if}\; (n, k) \notin \Lambda^{\star}_{N} \times \Lambda^{m}_{N},   
    \end{aligned} 
    \right.
\]
with $\Lambda_N := \left\{ n \in \mathbb{Z} \big| |n| \leq N \right\}$ and 
$\Lambda_N^{\star} := \left\{ n \in \mathbb{Z} \big|0< |n| \leq N \right\}$. 
Throughout this paper, unless otherwise specified, $|\cdot|$ and 
$|\cdot|_{\infty}$ denote the $\ell_1$-norm and the $\ell_{\infty}$-norm on 
$\mathbb{Z}^m$ or $\mathbb{R}^m$, respectively.

\subsection{Numerical demonstration}
\label{subsec: numerical-demonstration}

In this section, we demonstrate the performance of the alternating numerical 
scheme~\eqref{eqn: alternating-numerical-nlaw} for the NLAW 
equation~\eqref{eqn: nlaw} subject to the periodic boundary 
condition~\eqref{eqn: per-bdy}.

\paragraph{Periodic solutions ($m=1$)}
For the NLAW equation~\eqref{eqn: nlaw}, the linear solution is a real-valued 
traveling wave given by
\begin{equation}
    \label{eqn: lin-nlaw-1}
    u(x, t) = 2a_1 \cos(\omega_1 t + n_1x),
\end{equation}
whose only two nonzero space-time Fourier coefficients satisfy 
$\hat{u}(n_1, 1) = \hat{u}(-n_1, -1)=a_1$. 
The corresponding nonlinear solution has spatial Fourier coefficients of the form
\begin{subequations}
    \label{eqn: nlaw-periodic-1}
    \begin{empheq}[left=\empheqlbrace]{align} 
       & \hat{u}_{n_1}(t) = \overline{\hat{u}_{-n_1}(t)} = a_1 e^{i\omega_1't} + \sum_{k \in \mathbb{Z}\setminus \{1\}} \hat{u}(n_1, k) e^{ik\omega_1' t},    \label{eqn: nlaw-periodic-1-1}                \\
       & \hat{u}_{n}(t)   = \sum_{k \in \mathbb{Z}} \hat{u}(n, k) e^{ik\omega_1' t}, \quad n \in \mathbb{Z} \setminus \{-n_1, n_1\}.                          \label{eqn: nlaw-periodic-1-2}                
    \end{empheq}        
\end{subequations} 
For each nonzero spatial mode, we plot the trajectory in the normalized 
phase-plane coordinates
\[
    \left(\mathrm{Re} \left(\hat{u}_n(t)\right),~ -\frac{\mathrm{Re} \left(\dot{\hat{u}}_{n}(t) \right)}{\sqrt{n^4 + n^2}} \right), 
\]
in which the linear solution traces a circle in each excited mode. 
For the numerical test, we take $n_1=1$ and $a_1=1$, so that $\omega_1 = \sqrt{2}$, 
and set the perturbation parameter $\varepsilon = 0.5$. 
The results are shown in~\Cref{fig: nlaw-p-1torus-traj}. 
The tangential mode $\hat{u}_1(t)$ retains an amplitude close to unity and 
exhibits a small but discernible frequency shift relative to the linear solution. 
\begin{figure}[htb!]
    \centering
    \begin{subfigure}[t]{0.325\linewidth}
    \centering
    \includegraphics[scale=0.19]{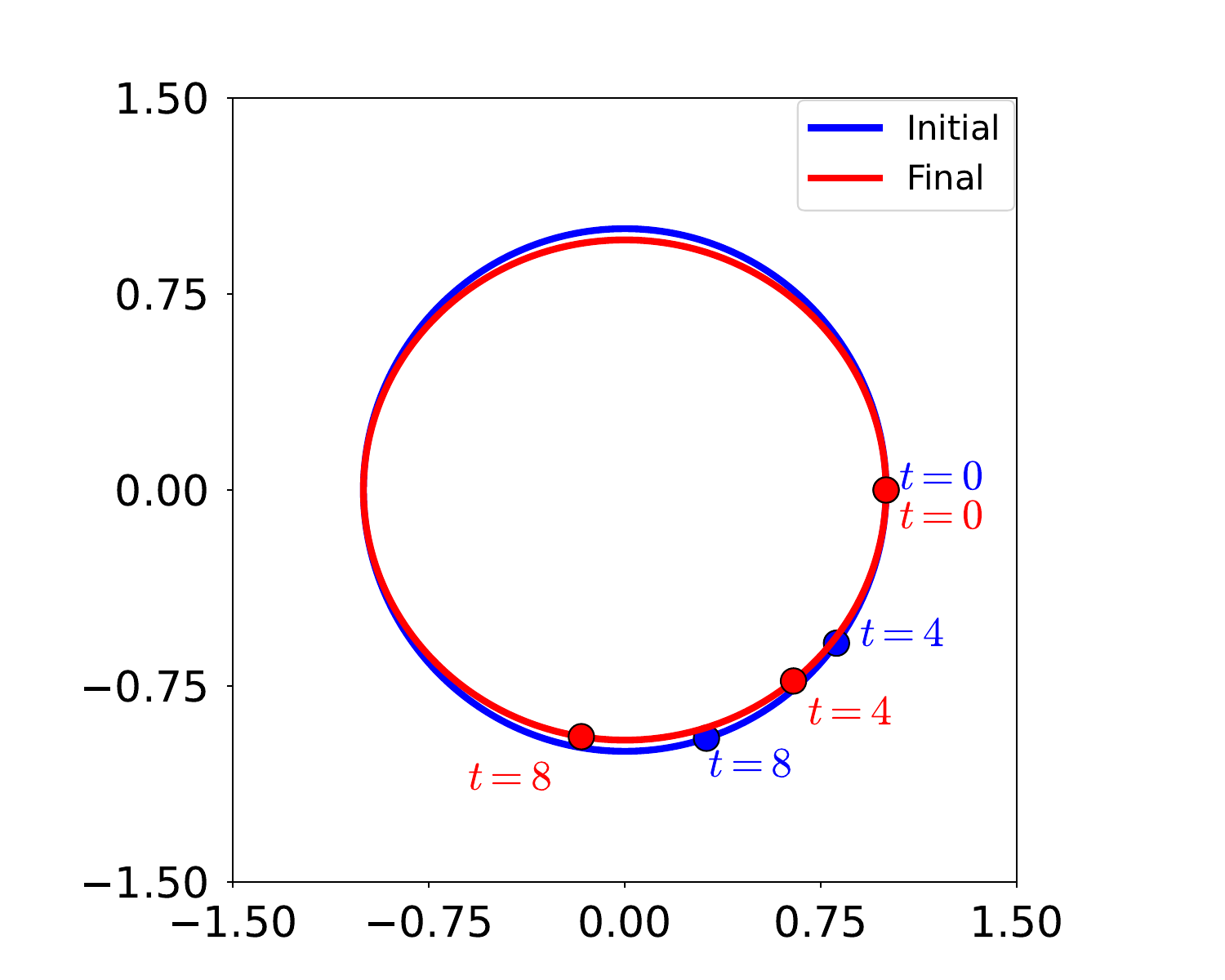}
    \caption{Tangential mode: $\hat{u}_1(t)$}
    \end{subfigure}
    \begin{subfigure}[t]{0.325\linewidth}
    \centering
    \includegraphics[scale=0.19]{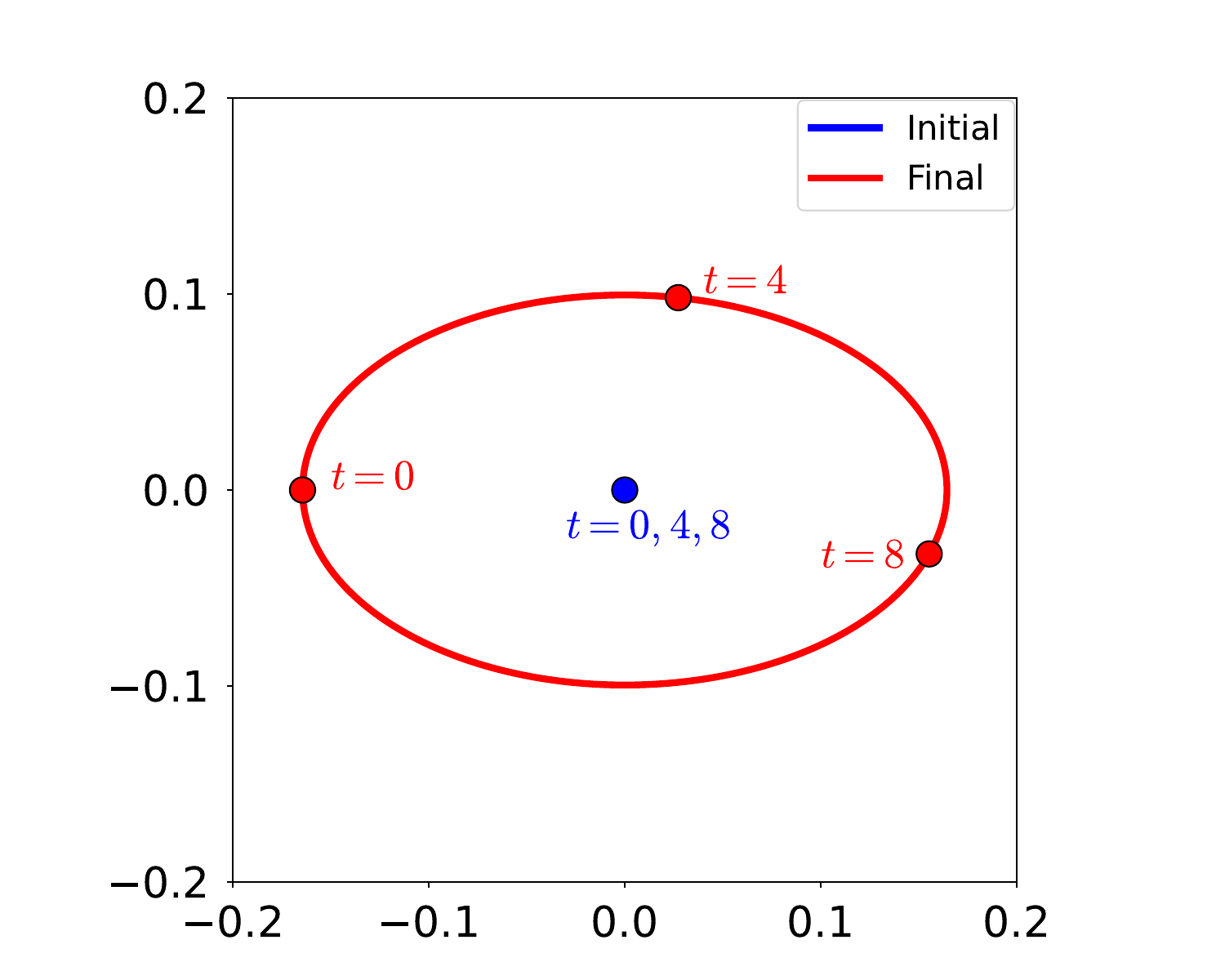}
    \caption{Normal mode: $\hat{u}_2(t)$}
    \end{subfigure}
    \begin{subfigure}[t]{0.325\linewidth}
    \centering
    \includegraphics[scale=0.19]{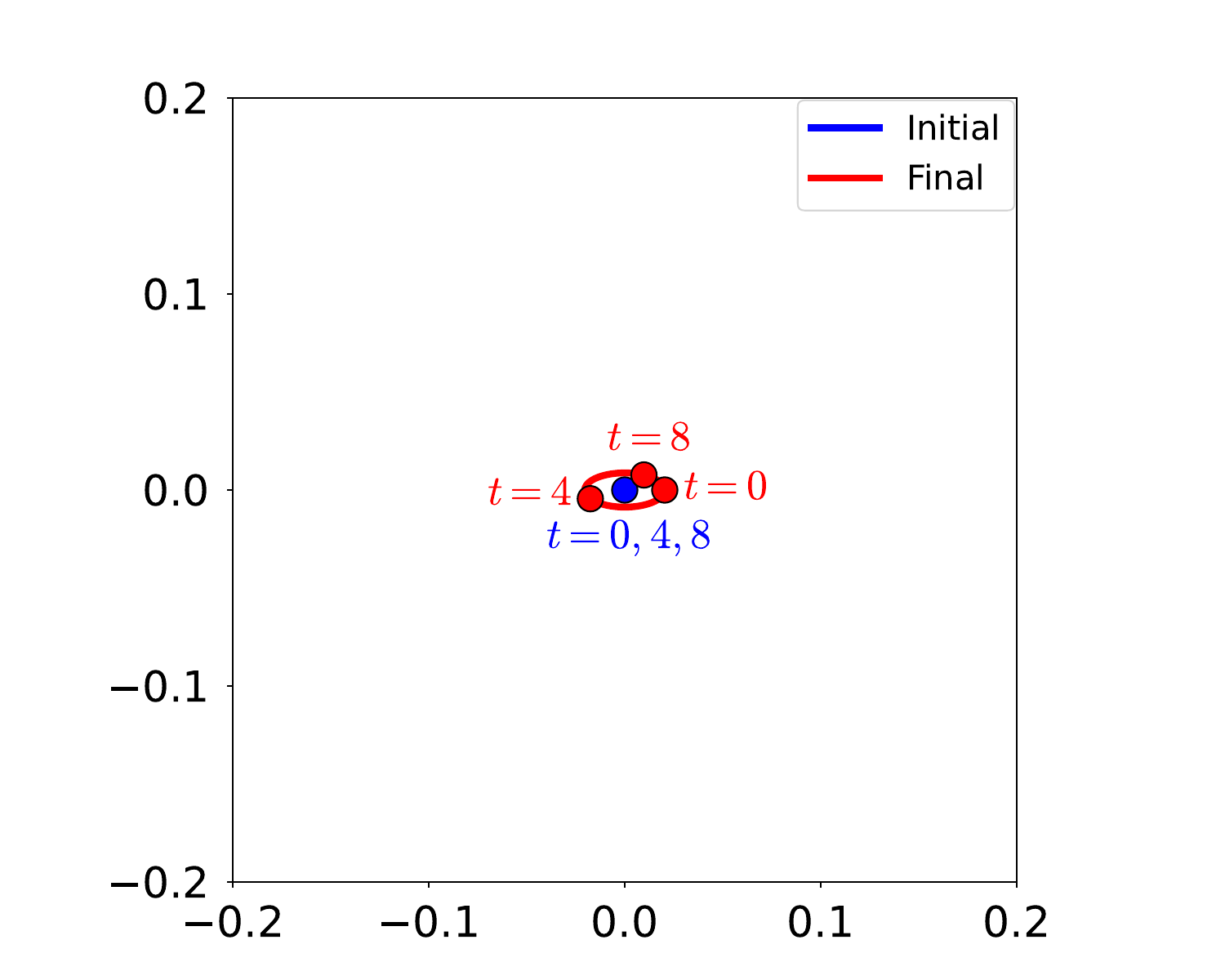}
    \caption{Normal mode: $\hat{u}_3(t)$}
    \end{subfigure}
    \caption{Phase space trajectories of the NLAW equation~\eqref{eqn: nlaw} with 
             the periodic boundary condition~\eqref{eqn: per-bdy}.  
             ``Initial'' and ``Final'' denote the initial linear solution and the 
             computed nonlinear solution, respectively. 
             Markers indicate specific time points at $t = 0$, $4$, and $8$.} 
    \label{fig: nlaw-p-1torus-traj}
\end{figure}
In contrast to the cubic bounded perturbation considered 
in~\citet{Fu2026numerical-pde-1}, the quadratic unbounded perturbation excites 
both normal modes shown in~\Cref{fig: nlaw-p-1torus-traj}: although absent 
from the linear solution, the representative modes, $\hat{u}_2(t)$ and 
$\hat{u}_3(t)$ develop small-amplitude oscillations through nonlinear 
interactions. 
Moreover, $\hat{u}_2(t)$ is excited more strongly than $\hat{u}_3(t)$, reversing 
the relative excitation strengths observed for the cubic bounded perturbation 
in~\citet{Fu2026numerical-pde-1}.~\Cref{fig: nlaw-p-1torus-con} shows the 
convergence of the alternating numerical 
scheme~\eqref{eqn: alternating-numerical-nlaw}, measured by successive 
differences in the Fourier coefficients, the temporal frequency, and the 
reconstructed solution at $t=4$. 
All three quantities decrease rapidly over the iterations. 

\par $\mathrel{\phantom{=}}$
\begin{figure}[htb!]
    \centering
    \begin{subfigure}[t]{0.325\linewidth}
    \centering
    \includegraphics[scale=0.19]{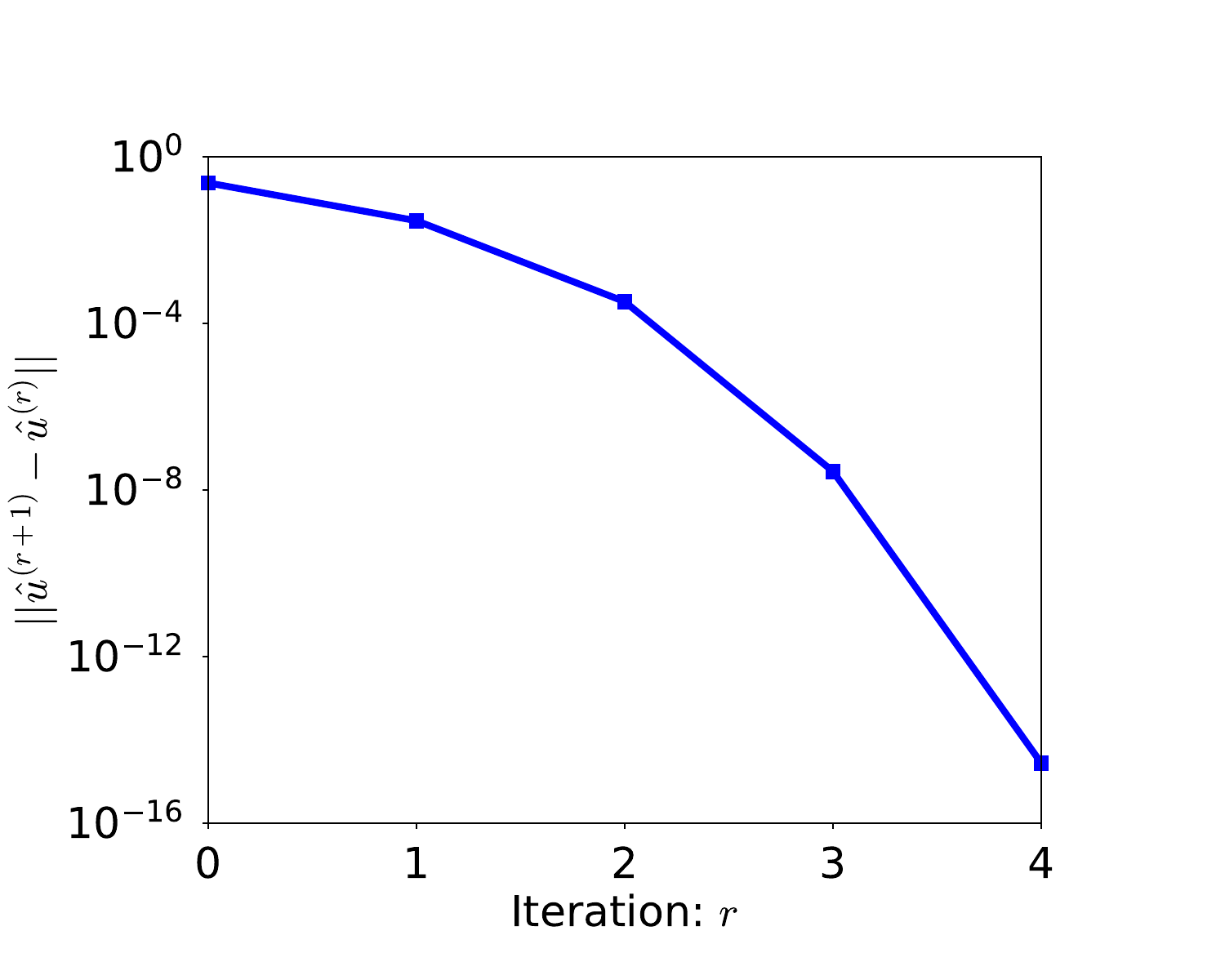}
    \caption{$\| \hat{u}_p^{(r+1)} - \hat{u}_p^{(r)} \|$}
    \end{subfigure}
    \begin{subfigure}[t]{0.325\linewidth}
    \centering
    \includegraphics[scale=0.19]{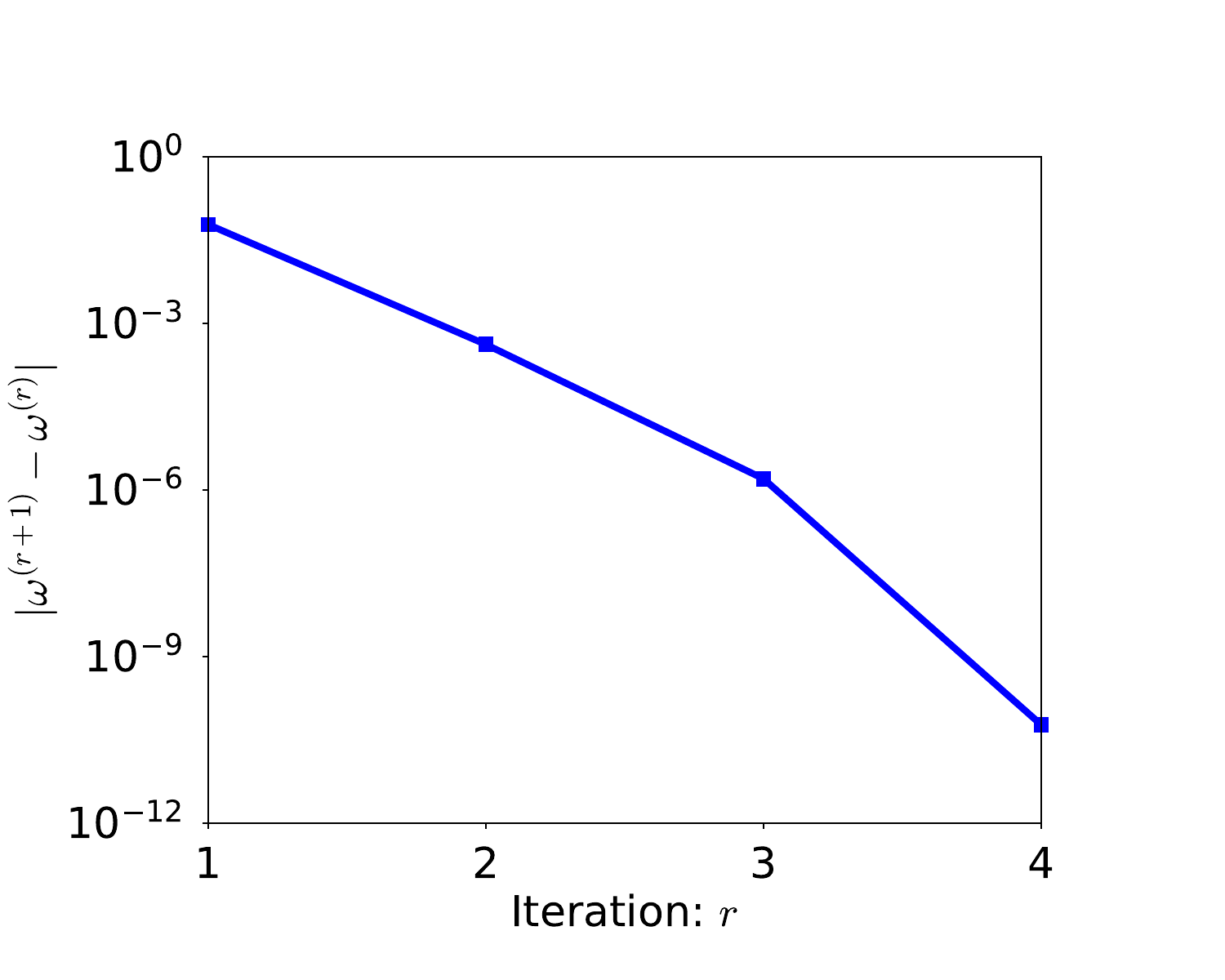}
    \caption{$| \omega^{(r+1)} - \omega^{(r)} |$}
    \end{subfigure}
    \begin{subfigure}[t]{0.325\linewidth}
    \centering
    \includegraphics[scale=0.19]{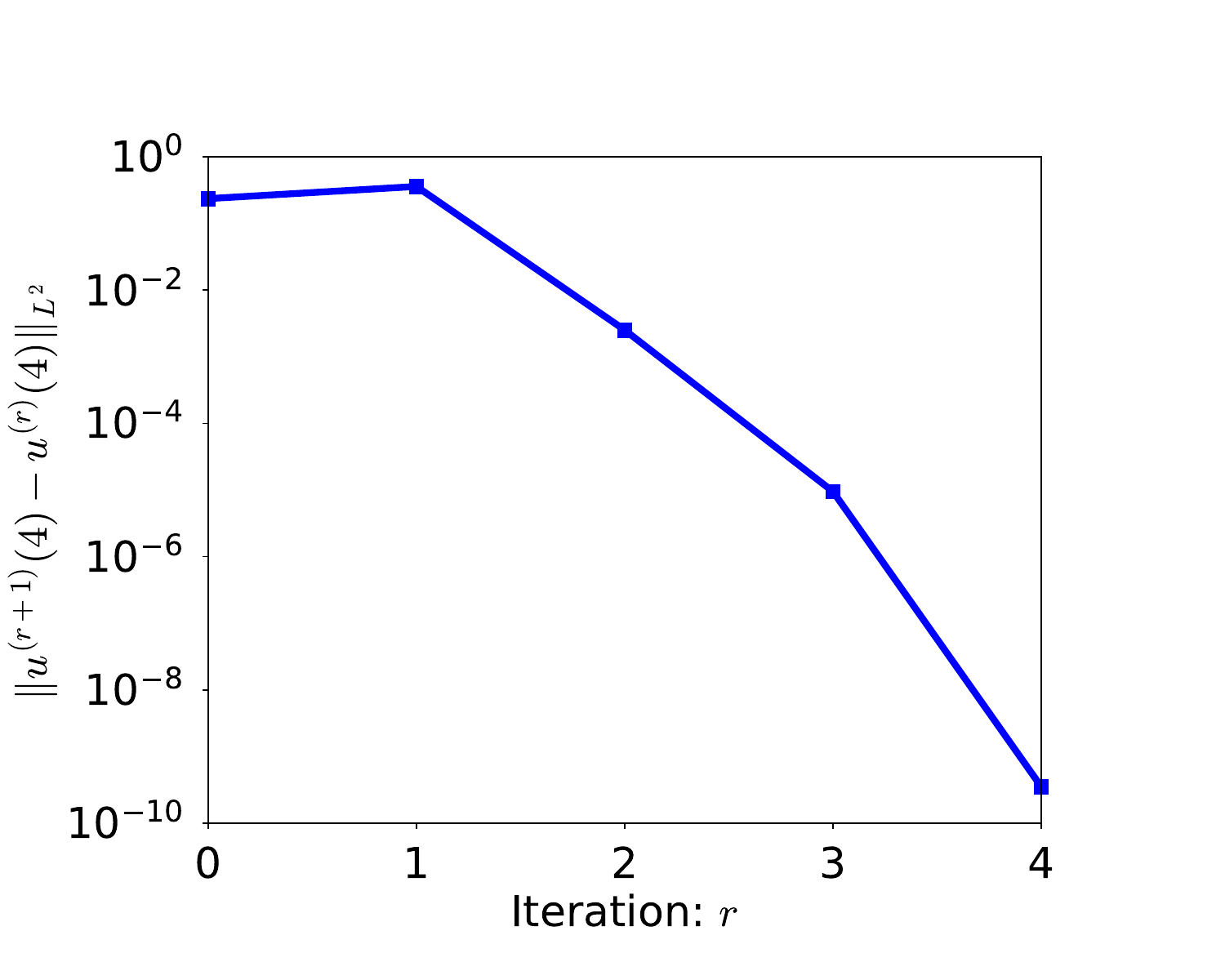}
    \caption{$\| u^{(r+1)}(x, 4) - u^{(r)}(x, 4) \|_{2}$}
    \end{subfigure}
    \caption{Convergence profiles across iterations for the NLAW 
             equation~\eqref{eqn: nlaw} with the periodic boundary 
             condition~\eqref{eqn: per-bdy}.} 
    \label{fig: nlaw-p-1torus-con}
\end{figure}

We next examine the spatiotemporal distribution of $|u(x,t)|^2$. 
For the linear solution~\eqref{eqn: lin-nlaw-1}, the squared modulus reads 
\begin{equation}
    \label{eqn: lin-square-nlaw-1} 
    |u|^2 = 4\cos^2\big(x + \sqrt{2}t\big), 
\end{equation} 
which describes a traveling wave propagating in the negative $x$-direction, with 
spatial period $\pi$ and uniform crest height $4$, as shown 
in~\Cref{fig: nlaw-p-1torus-heat-slice-lin}. 
Thus, each $2\pi$-period contains two identical crests. 
For the nonlinear solution, shown 
in~\Cref{fig: nlaw-p-1torus-heat-slice-nonlin}, the traveling-wave 
structure and propagation direction persist, but the two crests within each 
period are no longer equivalent: one rises to a height of about $6$, exceeding 
the linear maximum, whereas the other falls to about $3$, resulting in 
alternating high and low crests. 
The two crests also differ slightly in width, and the spacing between 
consecutive crests deviates slightly from $\pi$.
\begin{figure}[htb!]
    \centering
    \begin{subfigure}[t]{0.46\linewidth}
    \centering
    \includegraphics[scale=0.28]{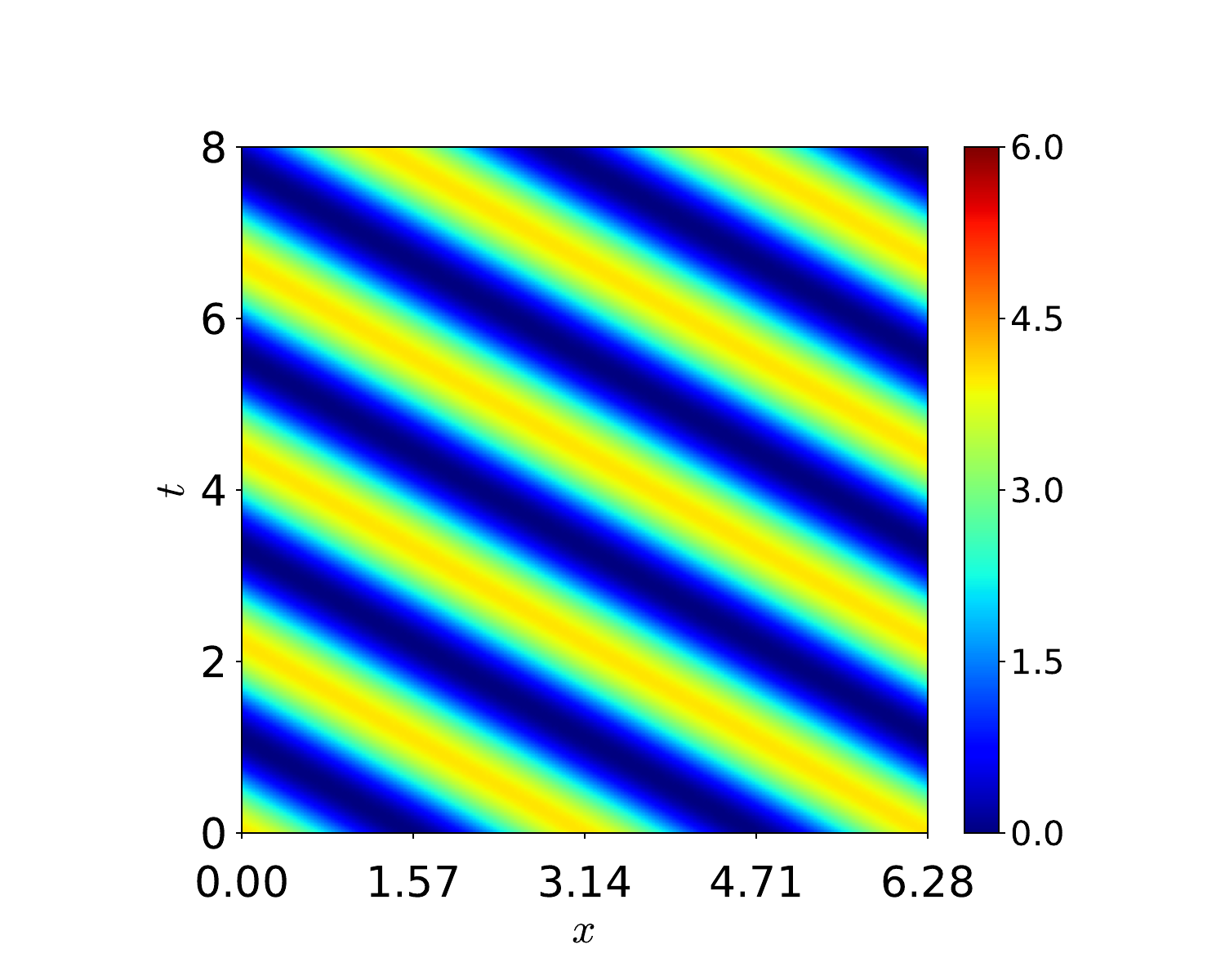}
    \caption{Linear solution}
    \label{fig: nlaw-p-1torus-heat-slice-lin}
    \end{subfigure}
    \begin{subfigure}[t]{0.46\linewidth}
    \centering
    \includegraphics[scale=0.28]{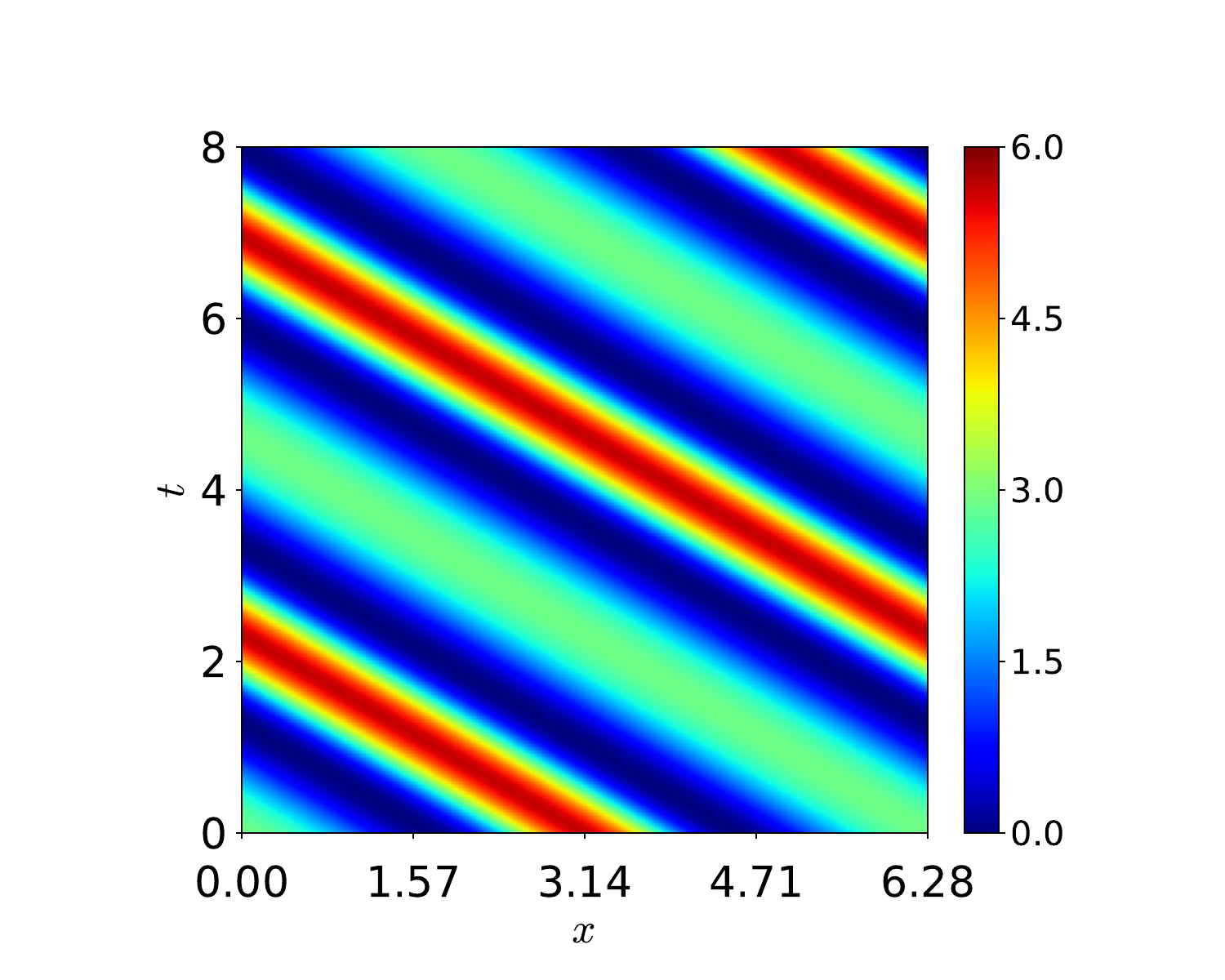}
    \caption{Nonlinear solution}
    \label{fig: nlaw-p-1torus-heat-slice-nonlin}
    \end{subfigure}
    \caption{Spatiotemporal distributions of $|u(x,t)|^2$ for the NLAW 
             equation~\eqref{eqn: nlaw} with the periodic boundary 
             conditions~\eqref{eqn: per-bdy}. 
             The squared modulus of the linear solution 
             is~\eqref{eqn: lin-square-nlaw-1}, and the nonlinear solution is 
             obtained via the alternating numerical 
             scheme~\eqref{eqn: alternating-numerical-nlaw}.}
    \label{fig: nlaw-p-1torus-heat-slice}
\end{figure}
The asymmetry in crest height reflects the contribution of the second harmonic 
$\hat{u}_2(t)$ discussed above. 
Generated by the quadratic unbounded perturbation, this harmonic contributes 
with the same sign at the positive and negative extrema of the fundamental 
harmonic, modifying their magnitudes differently. 
Consequently, the corresponding crests of $|u|^2$ acquire unequal heights.

\paragraph{Quasi-periodic solutions ($m=2$)}
For $m=2$, the linear quasi-periodic solution for the NLAW equation~\eqref{eqn: nlaw} is 
a superposition of two real waves, which reads 
\begin{equation}
    \label{eqn: lin-nlaw-2}
    u(x, t) = 2a_1 \cos(\omega_1 t + n_1x) + 2a_2 \cos(\omega_2 t + n_2x). 
\end{equation}
The corresponding nonlinear solution can be written in terms of its spatial 
Fourier coefficients as
\begin{subequations}
    \label{eqn: nlaw-periodic-2}
    \begin{empheq}[left=\empheqlbrace]{align} 
       & \hat{u}_{n_1}(t) = \overline{\hat{u}_{-n_1}(t)} = a_1 e^{i\omega_1't} + \sum_{k \in \mathbb{Z}^2\setminus \{(1, 0)\}} \hat{u}(n_1, k) e^{i\langle k, \omega' \rangle t},    \label{eqn: nlaw-periodic-2-1}                \\
       & \hat{u}_{n_2}(t) = \overline{\hat{u}_{-n_2}(t)} = a_2 e^{i\omega_2't} + \sum_{k \in \mathbb{Z}^2\setminus \{(0, 1)\}} \hat{u}(n_2, k) e^{i\langle k, \omega' \rangle t},    \label{eqn: nlaw-periodic-2-2}                \\
       & \hat{u}_{n}(t)   = \sum_{k \in \mathbb{Z}^2} \hat{u}(n, k) e^{i \langle k, \omega' \rangle t}, \quad n \in \mathbb{Z} \setminus \{\pm n_1, \pm n_2\},                       \label{eqn: nlaw-periodic-2-3}                
    \end{empheq}
\end{subequations}   
where $\omega' = (\omega_1', \omega_2')$ denotes the shifted frequency vector. 
For the numerical test, we take $n_1 = 1$, $n_2 =2$, and 
$a_1=a_2 =1$, so that $\omega_1=\sqrt{2}$ and 
$\omega_2 = 2\sqrt{5}$ and set the perturbation quantity 
$\varepsilon = 0.5$.~\Cref{fig: nlaw-p-2tori-traj} compares the phase-plane 
trajectories of the initial linear solution and the computed nonlinear solution. 
As in the periodic case shown in~\Cref{fig: nlaw-p-1torus-traj}, the 
frequency shifts are reflected in the displacement of the marked time points 
along the trajectories. The nonlinear trajectories return near previously 
visited positions without exactly retracing their paths over the displayed time 
interval, consistent with quasi-periodic motion. 
Moreover, the representative normal mode, which is absent from the initial 
linear solution, develops small-amplitude oscillations driven by nonlinear 
interactions.
\begin{figure}[htb!]
    \centering
    \begin{subfigure}[t]{0.325\linewidth}
    \centering
    \includegraphics[scale=0.19]{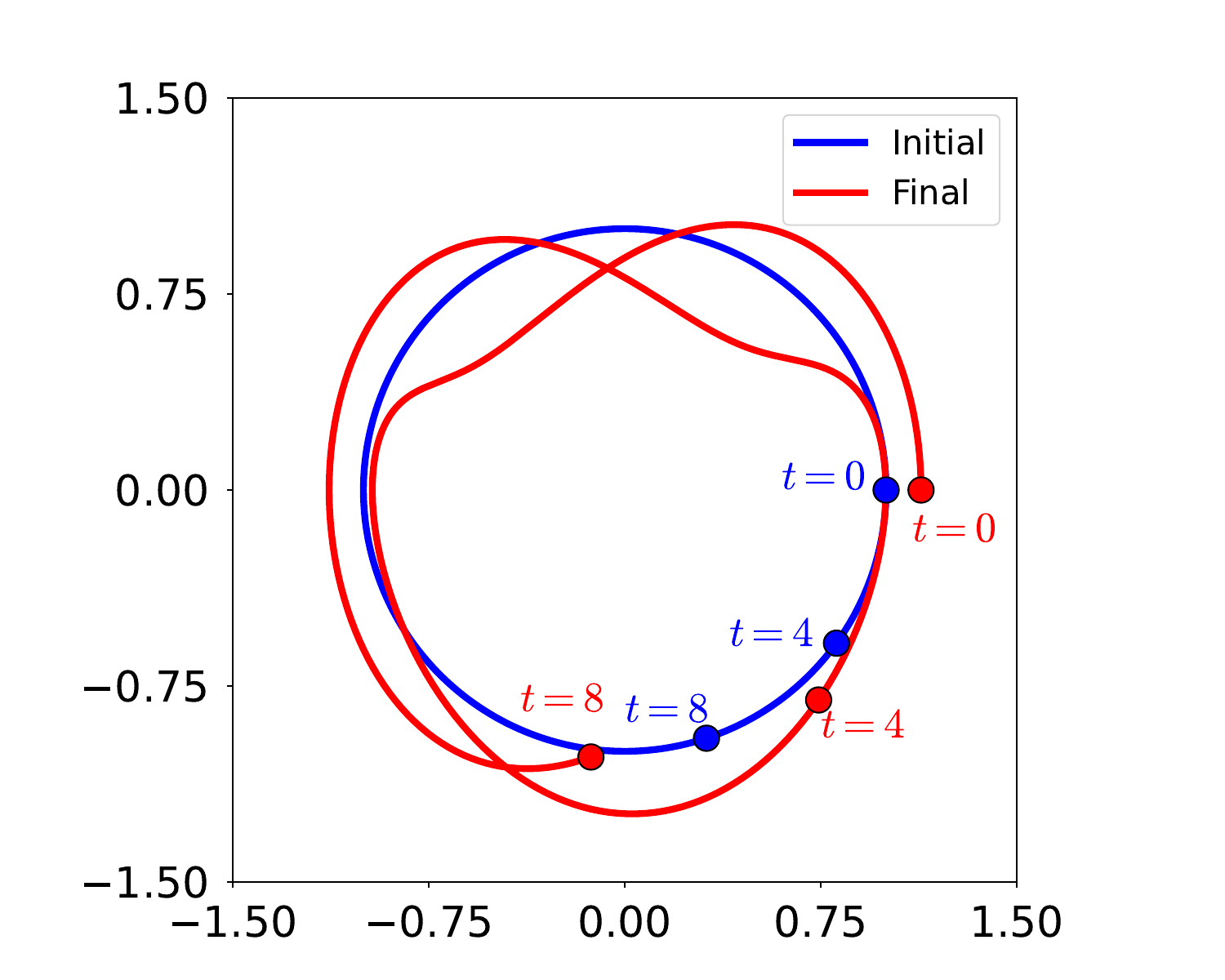}
    \caption{Tangential mode: $\hat{u}_1(t)$}
    \end{subfigure}
    \begin{subfigure}[t]{0.325\linewidth}
    \centering
    \includegraphics[scale=0.19]{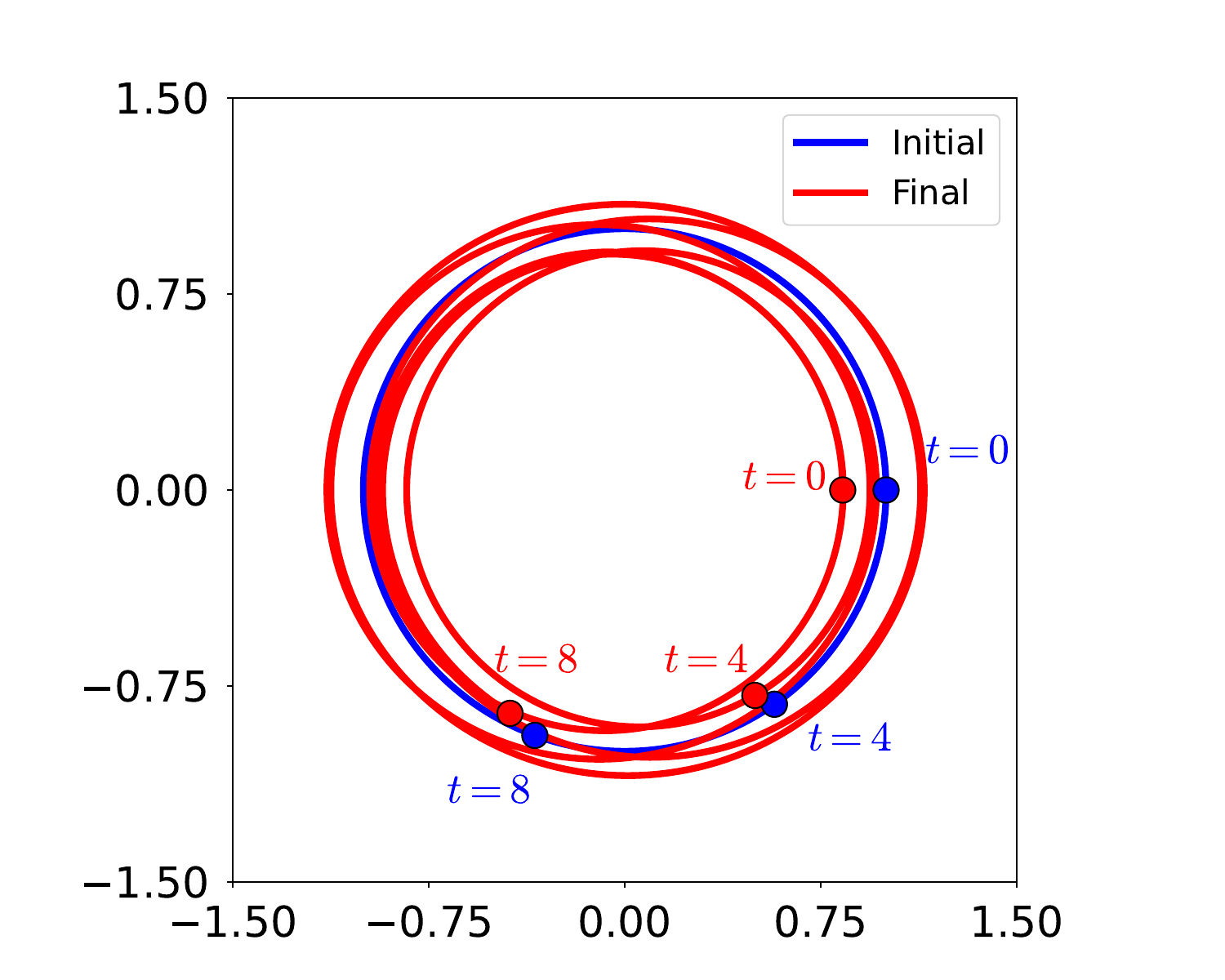}
    \caption{Tangential mode: $\hat{u}_2(t)$}
    \end{subfigure}
    \begin{subfigure}[t]{0.325\linewidth}
    \centering
    \includegraphics[scale=0.19]{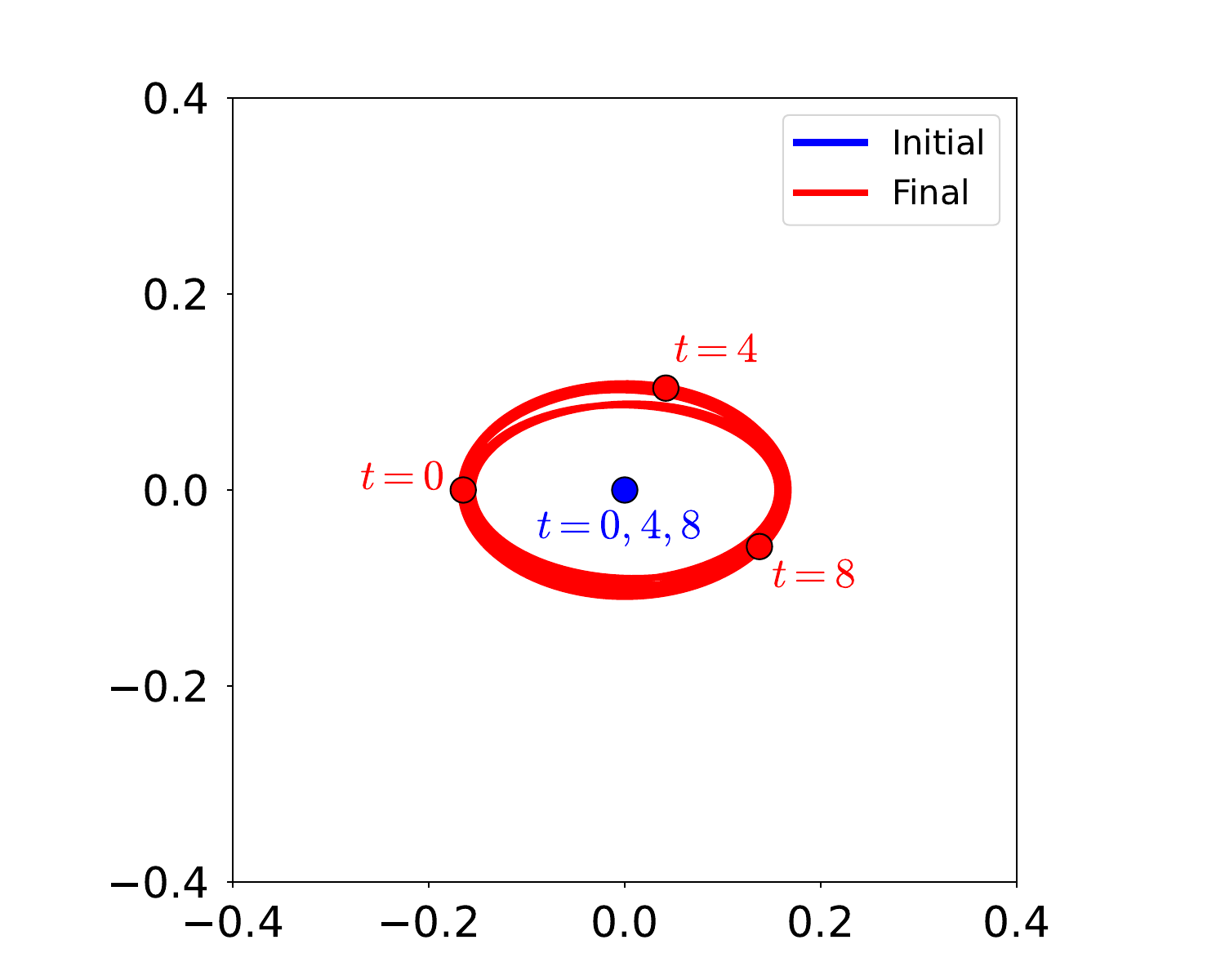}
    \caption{Normal mode: $\hat{u}_3(t)$}
    \end{subfigure}
    \caption{Phase space trajectories of the NLAW equation~\eqref{eqn: nlaw} with 
             the periodic boundary condition~\eqref{eqn: per-bdy}.  
             ``Initial'' and ``Final'' denote the initial linear solution and the 
             computed nonlinear solution, respectively. 
             Markers indicate specific time points at $t = 0$, $4$, and $8$.} 
    \label{fig: nlaw-p-2tori-traj}
\end{figure}

Turning to the spatiotemporal distributions, the squared modulus of the linear 
solution~\eqref{eqn: lin-nlaw-2}, with the parameters specified above, is
\begin{equation}
    \label{eqn: lin-square-nlaw-2}
    |u|^2 =  4\cos^2 \left( \sqrt{2}t + x \right) + 4\cos^2 \left( 2\sqrt{5}t + 2x \right) + 8\cos \left( \sqrt{2}t + x \right) \cos \left( 2\sqrt{5}t + 2x \right).
\end{equation}
Expanding the trigonometric products shows that, apart from a constant term, 
$|u(x, t)|^2$ contains four traveling-wave components with distinct phase speeds. 
Their superposition produces the diagonal bands shown 
in~\Cref{fig: nlaw-p-2tori-heat-slice-lin}. 
\begin{figure}[htb!]
    \centering
    \begin{subfigure}[t]{0.46\linewidth}
    \centering
    \includegraphics[scale=0.28]{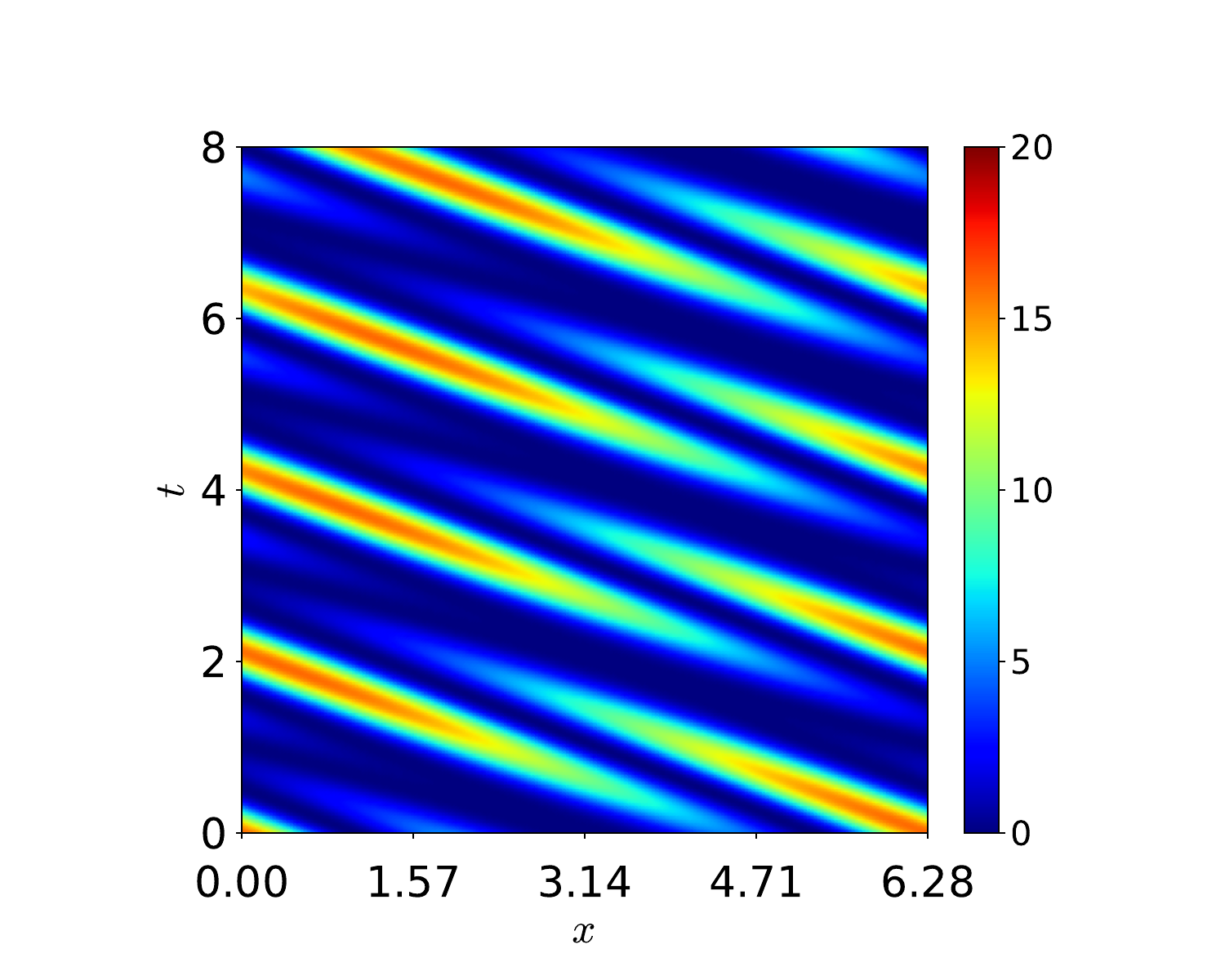}
    \caption{Linear solution}
    \label{fig: nlaw-p-2tori-heat-slice-lin}
    \end{subfigure}
    \begin{subfigure}[t]{0.46\linewidth}
    \centering
    \includegraphics[scale=0.28]{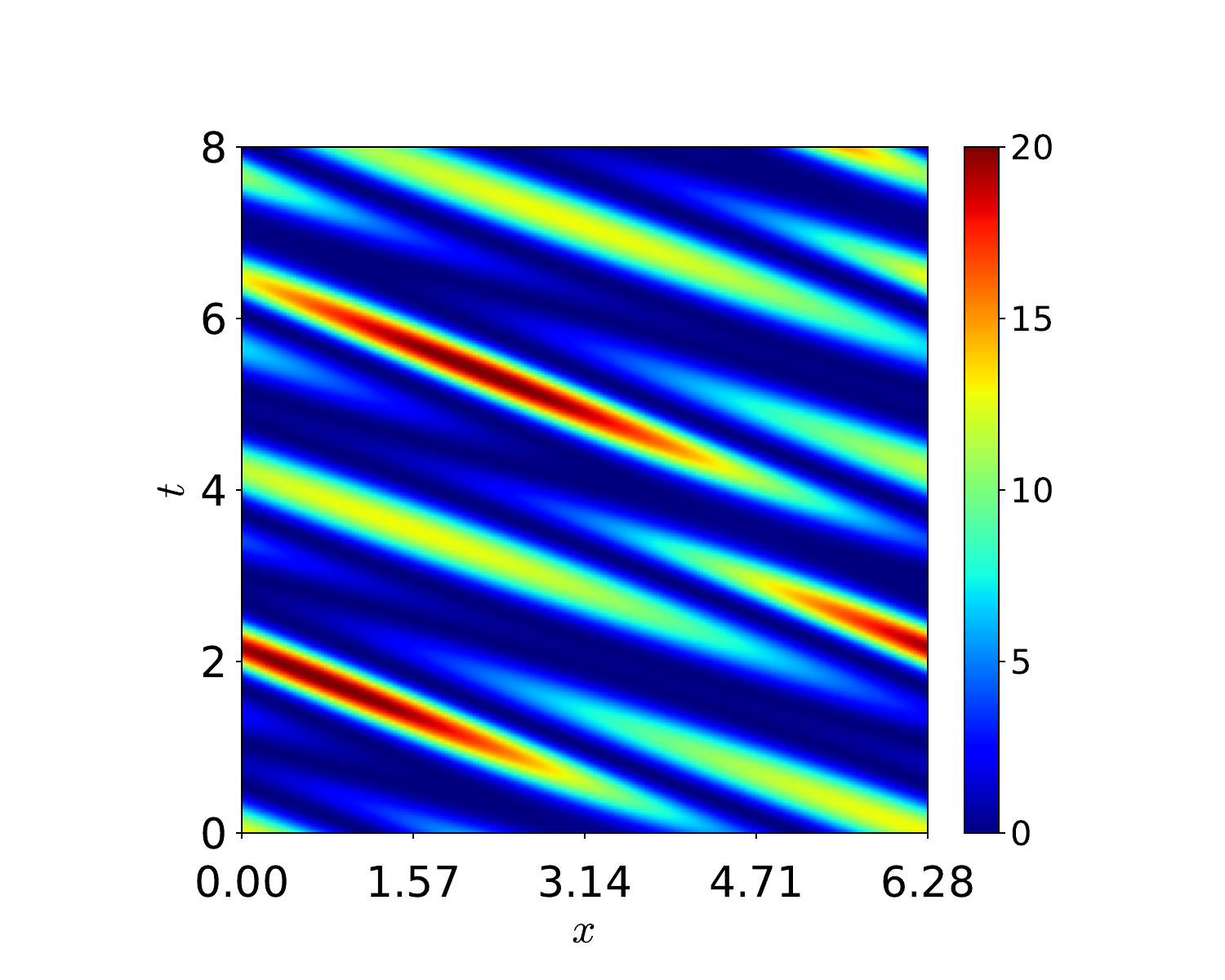}
    \caption{Nonlinear solution}
    \label{fig: nlaw-p-2tori-heat-slice-nonlin}
    \end{subfigure} 
    \caption{Spatiotemporal distributions of $|u(x,t)|^2$ for the NLAW 
             equation~\eqref{eqn: nlaw} with the periodic boundary 
             conditions~\eqref{eqn: per-bdy}. 
             The squared modulus of the linear solution is~\eqref{eqn: lin-square-nlaw-2}, 
             and the nonlinear solution is obtained via the alternating numerical 
             scheme~\eqref{eqn: alternating-numerical-nlaw}.}
    \label{fig: nlaw-p-2tori-heat-slice}
\end{figure}
With the nonlinear perturbation, the overall band structure persists, but the 
positions and relative intensities of the bands change. 
In~\Cref{fig: nlaw-p-2tori-heat-slice-nonlin}, some bands become markedly 
brighter, while others weaken, producing a more pronounced contrast between 
successive bands. 
The band spacing also changes, although the overall diagonal structure remains 
recognizable. 
These changes are consistent with the combined effects of nonlinear frequency 
shifts and the excitation of additional Fourier modes, which modify the 
interference pattern of the underlying waves. 
Thus, the nonlinear perturbation preserves the overall spatiotemporal 
organization while redistributing the intensity among the bands.

\subsection{The main theorem: a priori bound}
\label{subsec: main-state}

In this section, we state the convergence result and a priori bounds for the 
alternating numerical procedure~\eqref{eqn: alternating-numerical-nlaw}. 
For a Fourier vector $\hat{u} = \left\{ \hat{u}(n,k) \right\}_{(n,k) \in \mathcal{Z}}$, the standard $\ell_2$-norm is   
\begin{equation}
    \label{eqn: l2-norm}
    \|\hat{u}\|_2 = \left( \sum_{(n,k) \in \mathcal{Z}} |\hat{u}(n,k) |^2 \right)^{\frac12}. 
\end{equation}
Along the iteration, the Fourier coefficients are continuously differentiable with
respect to the drifted frequency vector $\omega'$. 
We therefore introduce a norm that incorporates both the coefficients and their 
frequency derivatives:
\begin{equation}
    \label{eqn: new-norm}
    \|\hat{u}\| = \left( \|\hat{u}\|_2^2 + \sum_{\ell=1}^{m} \| \partial_{\omega_{\ell}'} \hat{u} \|_2^2 \right)^{\frac12}. 
\end{equation}
Equivalently, this norm can be expressed componentwise as  
\[
    \|\hat{u}\|= \left[ \sum_{(n,k) \in \mathcal{Z}} \left( |\hat{u}(n, k)|^2 + \sum_{\ell = 1}^{m} |\partial_{\omega_{\ell}'}\hat{u}(n, k) |^{2} \right) \right]^{1/2}.
\]
For each $(n, k) \in \mathcal{Z}$, we also define the coefficient-wise norm  
\begin{equation}
    \label{eqn: new-norm-entries}
    \| \hat{u}(n, k) \| = \left( |\hat{u}(n, k)|^2 + \sum_{\ell = 1}^{m} |\partial_{\omega_{\ell}'}\hat{u}(n, k) |^{2} \right)^{\frac12},  
\end{equation}
which, unlike the modulus $|\hat{u}(n, k)|$, takes the frequency derivatives 
into account. 
The corresponding normed space of frequency-dependent Fourier vectors is
\begin{equation}
    \label{eqn: l2-space}
    \mathscr{H}(\mathcal{Z}) =
    \left\{
        \hat{u} = \left\{ \hat{u}(n,k) \right\}_{(n,k) \in \mathcal{Z}} ~\bigg|~ \hat{u}(n,k) \in \mathbb{R}, \; \|\hat{u}\| < \infty
    \right\}. 
\end{equation}
To quantify the regularity of solutions, we introduce, for $0<s<1$, the Gevrey 
decay set
\begin{equation}
    \label{eqn: gevrey-decay-set}
    \mathcal{K}(s) = \left\{ \hat{u} \in \mathscr{H}( \mathcal{Z}) ~\bigg|~ \sup_{(n,k) \in  \mathcal{Z} }\left( \| \hat{u}(n,k) \| \exp \left\{ (|n|+|k|)^s \right\} \right) \leq 1 \right\}.
\end{equation}
This condition imposes subexponential decay on both the Fourier coefficients and 
their drifted-frequency derivatives, thereby providing the regularity control 
needed throughout the iteration. 
We now state the main theorem, which establishes convergence of the alternating 
numerical procedure~\eqref{eqn: alternating-numerical-nlaw} and provides a 
priori error bounds for its iterates.

\begin{theorem}[Convergence and a priori bound]
    \label{thm: main}
    Fix $\delta_0 \in (0, 1)$ and $\tau > m + 2$, and let 
    $\mathscr{A} = [\delta_0, 1]^{m}$. 
    There exists a threshold $\varepsilon_0 > 0$ such that, for any 
    $0 < \varepsilon \leq \varepsilon_0$, the following statements hold: 
    \begin{itemize}
        \item[1.] \textbf{Measure of resonance}: 
                  There exist constants $\kappa = \kappa(m, \tau) >0$ and 
                  $\delta = \delta(\varepsilon) > 0$, with $\delta \rightarrow 0$ 
                  as $\varepsilon \rightarrow 0$, such that the excluded set 
                  $\mathscr{A}^{\star} \subset \mathscr{A}$ of initial amplitudes 
                  satisfies $\mathrm{mes}(\mathscr{A}^{\star}) \leq \kappa + \delta$. 

        \item[2.] \textbf{Iterative sequence}: 
                  For any $ a \in \mathscr{A} \setminus \mathscr{A}^{\star}$, the 
                  alternating numerical 
                  scheme~\eqref{eqn: alternating-numerical-nlaw}, initialized at 
                  $(\omega^{(0)}, \hat{u}_p^{(0)} ) = (\omega, 0)$, generates a 
                  well-defined sequence of iterates 
                  $\{ (\omega^{(r)}, \hat{u}_p^{(r)}) \}_{r=0}^{\infty}$. 
                  Moreover, for any $r > 0$, the updated frequency map 
                  $\omega^{(r+1)}: a \mapsto \omega^{(r+1)}(a)$ is a local 
                  diffeomorphism.

        \item[3.] \textbf{Convergence}: There exists a Gevrey exponent 
                  $s = s(\varepsilon) \in (0, 1)$ such that this sequence 
                  converges to an exact solution pair 
                  $(\omega^\star, \hat{u}_p^\star)$, where the limiting 
                  frequency map $\omega^{\star}: a \mapsto \omega^{\star}(a)$ is 
                  also a local diffeomorphism, and $\hat{u}_p^{\star} \in \mathcal{K}(s)$.

        \item[4.] \textbf{A priori error bounds}: For any $r \geq 0$, the 
                  iterative sequence $\{ \omega^{(r)}, \hat{u}_p^{(r)} \}_{r=0}^{\infty}$ 
                  satisfies the following error bounds:
                  \begin{subequations}
                      \label{eqn: convergence-vector-frequency}
                      \begin{empheq}[left=\empheqlbrace]{align} 
                          \|\hat{u}^{(r)}-\hat{u}^{\star}\| &\leq \exp\left\{ -\frac{3}{2}(M^s)^{r+1} \right\},                    \label{eqn: u-hat-converge}               \\
                          |\omega^{(r)}-\omega^{\star}| &\leq \exp\left\{ -\frac{3}{2}(M^s)^{r+1} \right\}.                   \label{eqn: frequency-converge}
                      \end{empheq}
                  \end{subequations}
    \end{itemize}
    Furthermore, define the $r$-th approximate solution by 
    \begin{equation}
        \label{eqn: approximate-pde-solution}
        u^{(r)}(x,t) = \sum_{ (n, k) \in \mathbb{Z}^{1+m} } \hat{u}^{(r)}(n, k) e^{i \langle k, \; \omega^{(r)} t + \theta_0 \rangle } e^{inx}. 
    \end{equation}
    Then, for any  $t \in [0, M^r]$, the following time-domain error bound holds:
    \begin{equation}
        \label{eqn: pde-solution-error-nlaw}
        \|u^{(r)}(x,t)-u^{\star}(x,t)\|_{L^2} \leq \exp \left\{ -(M^s)^{r+1} \right\},
    \end{equation}
    where $u=u^{\star}(x, t)$ is the exact quasi-periodic solution of the NLAW 
    equation~\eqref{eqn: nlaw} with periodic boundary 
    conditions~\eqref{eqn: per-bdy} associated with 
    $(\omega^{\star}, \hat{u}_p^{\star})$. 
    In particular, for any fixed $t\geq0$, $u^{(r)}(x,t)$ converges in $L^2$ to 
    $u^\star(x,t)$ as $r \rightarrow \infty$.
\end{theorem}

The proof of~\Cref{thm: main} is completed in~\Cref{sec: iterative-scheme} by 
combining the implementation framework developed in~\Cref{sec: implement} with 
the multiscale analysis presented in~\Cref{sec: multi-scale}.

\subsection{Organization of the paper}
\label{subsec: organization}

The remainder of this paper is organized as follows.~\Cref{sec: basic-setting} 
establishes uniform properties of the associated vectors and operators within 
the Gevrey decay class, which serve as the foundation for the subsequent 
analysis.~\Cref{sec: implement} formulates the implementation conditions 
corresponding to these uniform properties and verifies them by 
induction.~\Cref{sec: multi-scale} uses multiscale analysis, with inductive base 
cases covering two types of small-scale configurations, to establish the 
inversion and localization conditions for the tangent operator restricted to 
outer boxes.~\Cref{sec: iterative-scheme} establishes the iterative lemma and 
completes the proof of the main convergence results.~\Cref{sec: kdv} generalizes 
these results to the KdV equation~\eqref{eqn: kdv} under periodic boundary 
conditions and presents the corresponding numerical demonstrations. 
Finally,~\Cref{sec: conclu} concludes the paper and outlines directions for 
future research.

\section{Gevrey decay properties}
\label{sec: basic-setting}

In this section, we establish uniform Gevrey decay estimates and norm bounds for 
the nonlinear vector field and its associated operators on $\mathcal{K}(s)$. 
Let $\mathbb{T}^{1+m} = [0,2\pi]^{1+m}$ denote the extended $(1+m)$-dimensional 
space-time torus with coordinates $(x,\theta)$. 
In the full-lattice representation~\eqref{eqn: nonlin-soln-nlaw-full}, the 
Fourier coefficients $\hat{u}(n,k)$ for any $(n,k)\in \mathcal{Z}$ are regarded 
as independent variables, and the components of the nonlinear vector field 
$\hat{X}$ are given by 
\begin{equation} 
    \label{eqn: vector-field} 
    \hat{X}(n,k) = \frac{1}{(2\pi)^{1+m}} \iint_{\mathbb{T}^{1+m}} u^2 e^{-i(\langle k,\theta\rangle+nx)} d\theta dx. 
\end{equation} 
The symmetry constraint $\hat{u}(-n,-k)=\hat{u}(n,k)$ is imposed afterwards by 
restricting to the corresponding invariant 
subspace.\footnote{In the half-lattice 
representation~\eqref{eqn: nonlin-soln-nlaw}, the components of $\hat{X}$ take 
the form
\[ 
    \hat{X}(n,k) = \frac{1}{(2\pi)^{1+m}} \iint_{\mathbb{T}^{1+m}} u^2\cos(\langle k,\theta\rangle+nx) d\theta dx,   
\] 
and differentiation with respect to the independent half-lattice coefficients 
produces a Toeplitz-plus-Hankel operator. 
The two descriptions are equivalent: the Toeplitz-plus-Hankel operator is the 
full-lattice Toeplitz operator restricted to the symmetric subspace and written 
in half-lattice coordinates, with any fixed modes projected out. 
The half-lattice form, however, is less convenient for the multiscale analysis 
below, so we work on the full lattice throughout.} 
The following proposition, whose proof is given in~\Cref{subsec: vector-field}, 
provides uniform decay estimates and norm bounds for the vector field 
$\hat{X} = (\hat{X}_q, \hat{X}_p)^{\top}$.

\begin{prop}
    \label[proposition]{prop: vector-field}
    Suppose that $\hat{u} \in \mathcal{K}(s)$. 
    Then there exist constants $\gamma_1, \gamma_2 > 0$, depending on the number 
    $m$ of resonant modes, the Gevrey exponent $s$, the initial amplitude vector 
    $a$ and the non-resonant vector $\hat{u}_p$, such that 
    \begin{subequations}
        \label{eqn: vector-field-bound}
        \begin{empheq}[left=\empheqlbrace]{align} 
            & \sup_{(n,k) \in \mathcal{Z} } \left( \|\hat{X}(n,k)\| e^{ (|n| + |k|)^s } \right) \leq \gamma_1,    \label{eqn: vector-field-comp}    \\
            & \max \big \{ \|\hat{X}_q\|, \|\hat{X}_p\| \big \} \leq \|\hat{X}\| \leq \gamma_2.                   \label{eqn: vector-field-norm}
        \end{empheq}  
    \end{subequations}
\end{prop}

We next consider the derivative of $\hat{X}$ with respect to the non-resonant 
Fourier coefficients, 
$\partial \hat{X} / \partial \hat{u}_p = ( \partial \hat{X}_q / \partial \hat{u}_p, S)^{\top}$. 
Since the nonlinearity $u^2$ acts pointwise, the full linearization 
$\partial\hat{X}/\partial\hat{u}$ is a convolution (Toeplitz) operator on 
$\mathbb{Z}^{1+m}$. 
In particular, for non-resonant indices $(n,k), (n',k') \notin \mathcal{S}$, the 
integral expression~\eqref{eqn: vector-field} implies
\begin{align*}
    S\left( (n, k), (n',k') \right) 
    & = S\left( n - n', k - k' \right)    \\
    & = \frac{2}{(2\pi)^{1+m}} \iint_{\mathbb{T}^{1+m}} ue^{-i\left[ \langle k - k', \theta \rangle + (n- n')x \right] } d\theta dx.
\end{align*}
The following proposition establishes the symmetry, off-diagonal decay, and 
boundedness of $S$, with its formal proof given in~\Cref{subsec: convolution}.

\begin{prop}
    \label[proposition]{prop: convolution}
    Suppose that $\hat{u} \in \mathcal{K}(s)$. 
    Then there exist constants $\gamma_3, \gamma_4 > 0$, depending on the number 
    $m$ of resonant modes, the Gevrey exponent $s$, the initial amplitude vector 
    $a$ and the non-resonant vector $\hat{u}_p$, such that $S$ is symmetric and 
    satisfies:
    \begin{subequations}
        \label{eqn: tangent}
        \begin{empheq}[left=\empheqlbrace]{align} 
            & \sup_{ (n, k), (n', k') \notin \mathcal{S}} \left(\|S\left( (n, k), (n', k') \right)\| e^{ (|n - n'| + |k - k'|)^s } \right) \leq \gamma_3,                                   \label{eqn: tangent-comp}  \\
            & \max \left\{ \left\| \frac{\partial \hat{X}_q}{\partial \hat{u}_p} \right\|, \|S\| \right\} \leq \left\| \frac{\partial \hat{X}}{\partial \hat{u}_p} \right\| \leq \gamma_4.  \label{eqn: tangent-bound}
        \end{empheq}        
    \end{subequations}
\end{prop}

The iteration lemma (\Cref{lem: iteration}) also requires a bound on the second 
derivative of $\hat{X}$. 
We state this estimate below and provide its proof in~\Cref{subsec: tensor}.

\begin{prop}
    \label[proposition]{prop: tensor}
    Suppose $\hat{u} \in \mathcal{K}(s)$ and, in addition, that 
    $\mathrm{supp}~\hat{u} \subseteq \Lambda_N^{1+m}$. 
    Then we have the following estimate
    \begin{equation}
        \label{eqn: tensor}
        \max \left\{ \left\| \frac{\partial^2 \hat{X}_q}{\partial \hat{u}_p^2} \right\|, \left\| \frac{\partial S}{\partial \hat{u}_p } \right\| \right\} \leq \left\| \frac{\partial^2 \hat{X}}{\partial \hat{u}_p^2} \right\| \leq 4 (2N+1)^\frac{1+m}{2}. 
    \end{equation}
\end{prop}

To control the error associated with the frequency 
update~\eqref{eqn: alternating-frequency-nlaw}, we also require estimates for 
the auxiliary operator $B$. 
These estimates are stated in the following proposition, with the proof given 
in~\Cref{subsec: b-operator}.

\begin{prop}
    \label[proposition]{prop: B-operator}
    Suppose that $\hat{u} \in \mathcal{K}(s)$. 
    Then there exist constants $\gamma_5, \gamma_6 > 0$, depending on 
    the number $m$ of resonant modes, the Gevrey exponent $s$, the 
    initial amplitude vector $a$ and the non-resonant vector $\hat{u}_p$, such 
    that the operator $B$ satisfies: 
    \begin{subequations}
        \label{eqn: b-operator-bound}
        \begin{empheq}[left=\empheqlbrace]{align} 
            & \sup_{(n, k), (n', k') \notin \mathcal{S}} \left( \left\| B\left( (n, k), (n', k') \right) \right\| \exp\left\{ (|n|+|k|)^s + (|n'|+|k'|)^s \right\} \right) \leq \gamma_5, \label{eqn: b-comp} \\
            & \| B \| \leq \gamma_6.  \label{eqn: b-bound}
        \end{empheq}        
    \end{subequations}
\end{prop}

Since the number $m$ and the exponent $s$ are fixed throughout, we set 
$\gamma = \max_{1 \leq j \leq 6} \gamma_j$, which depends only on the resonant 
amplitude $a$ and the non-resonant vector $\hat{u}_p$. 
Since the amplitude set $\mathscr{A}$ is a compact set and the bounds 
in~\Cref{prop: vector-field}--\Cref{prop: B-operator} are uniform for 
$\hat{u}_p \in \mathcal{K}(s)$, $\gamma$ is bounded uniformly. 
We may therefore fix $\varepsilon_0 > 0$ such that 
$\sqrt{\varepsilon_0} \gamma \leq (1/2) \delta_0 $ without loss of generality. 
The bounds on $\hat{X}_q$ and its derivative in~\Cref{prop: vector-field}, 
together with the $Q$-equations~\eqref{eqn: q-eqn-nlaw}, yield the following 
estimates for the frequency drift.

\begin{prop}
    \label[proposition]{prop: frequency-drift}
    Suppose $\hat{u} \in \mathcal{K}(s)$. 
    The frequency drift determined by the $Q$-equations~\eqref{eqn: q-eqn-nlaw} 
    satisfies
    \begin{equation}
        \label{eqn: frequency-drift-bound}
        \left| \omega' - \omega \right| \leq \sqrt{\varepsilon}.
    \end{equation}
    Moreover, the $Q$-equations~\eqref{eqn: q-eqn-nlaw} locally define a continuously differentiable frequency map $\omega':a \mapsto \omega'(a)$, which is a local diffeomorphism wherever $\partial Q/ \partial a$ is invertible.
\end{prop}

Indeed, the $Q$-equations~\eqref{eqn: q-eqn-nlaw} define an implicit relation
between the amplitude vector $a$ and the frequency $\omega'$:
\begin{equation}
    \label{eqn: frequency-map}
    Q(\omega', a) = (-\omega'^2 + \omega^2)\odot a + \varepsilon W_q \hat{X}_q(\omega', a) = 0.
\end{equation}
Regarding the amplitude $a$ locally as a continuously differentiable function of $\omega'$ along the solution and 
differentiating~\eqref{eqn: frequency-map} with respect to $\omega'$
gives
\begin{equation}
    \label{eqn: frequency-map-derivative}
    \frac{\partial Q}{\partial \omega'} + \frac{\partial Q}{\partial a}~\frac{\partial a}{\partial \omega'} = 0.
\end{equation}
Moreover, we have
\begin{equation}
    \label{eqn: q-freq-derivative}
    \frac{\partial Q}{\partial \omega'} = -2\mathrm{diag}(a \odot \omega') + \varepsilon W_q \frac{\partial \hat{X}_q}{\partial \omega'}.
\end{equation}
Since each component of $a$ is bounded below by 
$\delta_0$, the uniform bound in~\Cref{prop: vector-field} ensures that 
$\partial Q/\partial \omega'$ is invertible. 
It follows from~\eqref{eqn: frequency-map-derivative} that the product of the two square
matrices $\partial Q/\partial a$ and $\partial a/\partial \omega'$ is invertible. Hence both matrices are invertible, and the locally defined map $\omega':a \mapsto \omega'(a)$ is a local diffeomorphism.

\section{Implementation conditions}
\label{sec: implement}

The heart of the alternating numerical 
scheme~\eqref{eqn: alternating-numerical-nlaw} is the update of the non-resonant 
vector via the dimension-enlarged Newton 
iteration~\eqref{eqn: alternating-coefficient-nlaw}. 
To overcome the small-divisor problem, we must rigorously control the inverse of 
the truncated operator $(T + \varepsilon B)_N$ as the truncation scale $N$ 
increases. 
To this end, we introduce the scale-dependent threshold:  
\begin{equation}
    \label{eqn: scale-dependent}
    \varepsilon_N := e^{ - (\log N)^{15} }. 
\end{equation}
We then impose the following implementation conditions, which guarantee the 
convergence of the alternating numerical 
scheme~\eqref{eqn: alternating-numerical-nlaw}. 
For clarity, the analysis in this section and~\Cref{sec: multi-scale} is carried 
out in $\ell_2(\mathcal{Z})$, without incorporating derivatives with respect to 
the drifted frequency vector. 
The estimates are extended to include these derivatives 
in~\Cref{subsec: induction}. 
For convenience, throughout the remainder of this paper, we write 
$\Gamma_{n_0, k_0, N} := (n_0, k_0) + \Lambda_N^{1+m}$, and denote by 
$T_{n_0, k_0, N}$ the restriction of the operator $T$ to $\Gamma_{n_0, k_0, N}$. 
In particular, when $k_0 = 0$, we abbreviate $\Gamma_{n_0, 0, N}$ as 
$\Gamma_{n_0, N}$; if, in addition, $n_0 = 0$, we simply write $\Gamma_N$. 
The same abbreviations will be used for the corresponding restrictions of 
the operators considered below.

\begin{tcolorbox}[breakable]
    \begin{condition}[Implementation Conditions]
        \label[condition]{cond: implementation}
        At each scale $N$, the restricted operator $(T + \varepsilon B)_N$ is 
        required to satisfy the following two conditions:
        \begin{itemize}
            \item[(1)] \textbf{Inversion condition}: The restricted operator $(T + \varepsilon B)_N$ is invertible, and its inverse satisfies 
                       \begin{equation}
                           \label{eqn: inverse-grow}
                           \| (T + \varepsilon B)_{N}^{-1} \|_2 \leq \frac{1}{\varepsilon_{N}}.
                       \end{equation}

            \item[(2)] \textbf{Localization condition}: For any $(n,k), (n', k') \in \Lambda_N^{1+m}$ with $|n - n'| + |k-k'| \geq N^{\frac12}$, the entries of the inverse matrix obey the Gevrey decay estimate
                       \begin{equation}
                           \label{eqn: off-diagonal}
                           | (T + \varepsilon B)_{N}^{-1}\left( (n, k), (n', k') \right) | \leq e^{ - \frac{ \left(|n - n'|+ |k-k'|\right)^{s}}{2} }.
                       \end{equation}
        \end{itemize}
    \end{condition}
\end{tcolorbox}

As in the case of bounded perturbations~\citep{Fu2026numerical-pde-1}, we seek to 
establish these bounds uniformly over all central boxes $\Gamma_{n_0, N}$ with 
$n_0 \in \mathbb{Z}$. 
This need for uniform control over all spatial centers is a key distinction from 
the finite-dimensional setting~\citep{Fu2026numerical-a, Fu2026numerical-b}. 
Crucially, the dispersion relation~\eqref{eqn: disper-nlaw} between wavenumber 
and frequency allows the measure of the union of the nearly-resonant sets 
associated with all spatial indices $n_0 \in \mathbb{Z}$ to be made sufficiently 
small, as shown in~\Cref{subsec: nearly-resonant}. 
Denoting the restriction of $T + \varepsilon B$ to such a box by 
$(T + \varepsilon B)_{n_0, N}$, we formulate this stronger requirement as 
follows.

\begin{tcolorbox}[breakable]
    \begin{condition}[Uniform Conditions]
        \label[condition]{cond: implementation-proved}
        At each scale $N$, the following two conditions are required to hold 
        uniformly for all $n_0 \in \mathbb{Z}$:
        \begin{itemize}
        \item[(1)] \textbf{Inversion condition}: The restricted operator $(T + \varepsilon B)_{n_0, N}$ is invertible, and its inverse satisfies 
                    \begin{equation}
                        \label{eqn: inverse-grow-proved}
                        \| (T + \varepsilon B)_{n_0, N}^{-1} \|_2 \leq \frac{1}{\varepsilon_{N}}.
                    \end{equation}

        \item[(2)] \textbf{Localization condition}: For any $(n,k), (n', k') \in \Gamma_{n_0, N}$ with $|n - n'| + |k-k'| \geq N^{\frac12}$, the entries of the inverse matrix obey the Gevrey decay estimate
                    \begin{equation}
                        \label{eqn: off-diagonal-proved}
                        | (T + \varepsilon B)_{n_0, N}^{-1}\left( (n, k), (n', k') \right) | \leq e^{ - \frac{ \left(|n - n'|+ |k-k'|\right)^{s}}{2} }.
                    \end{equation}
        \end{itemize}
    \end{condition}
\end{tcolorbox}

Clearly,~\Cref{cond: implementation-proved} reduces to~\Cref{cond: implementation} 
upon taking $n_0 = 0$. 
In what follows, our main objective is to establish~\Cref{cond: implementation-proved} 
through the systematic exclusion of ``bad'' frequencies. 
At a fixed iteration, suppressing the iteration index, we write 
$\Omega := \omega'(\mathscr{A})$ for the corresponding drifted frequency set, 
where $\omega': a \mapsto \omega'(a)$ denotes the local diffeomorphism 
established in~\Cref{prop: frequency-drift}. 
We divide the analysis into a small regime and a large-scale regime, separated by 
the threshold
\begin{equation}
    \label{eqn: M}
    M = e^{ \left( \log \frac{1}{\varepsilon} \right)^{\frac{1}{20}} }, 
\end{equation}
which is determined by the perturbation parameter $\varepsilon$. 
The small-scale regime corresponds to $M_0 \leq N \leq M$, whereas the 
large-scale regime corresponds to $N > M$, where the initial scale is defined by 
\begin{equation}
    \label{eqn: M0}
   M_0 := e^{ \left( \log M \right)^{\frac{1}{20}} } = e^{ \left( \log \frac{1}{\varepsilon} \right)^{\frac{1}{400}} }.
\end{equation}    
For notational simplicity, we write $L := T + \varepsilon B$ for the linearized 
operator throughout the remainder of the paper.

\subsection{Nearly-resonant sets}
\label{subsec: nearly-resonant}

We first consider the small-scale regime $M_0 \leq N \leq M$. 
To establish uniform bounds on the inverses of the restricted operators, we 
exclude frequencies associated with near resonances. 
The nearly-resonant sets defined below are analogous to those introduced 
by~\citet{Fu2026numerical-pde-1}; in the present setting, however, the infinitely 
many frequencies take the explicit form $\sqrt{n^4 + n^2}$ for 
$n \in \mathbb{Z} \setminus \{0\}$, as dictated by the dispersion 
relation~\eqref{eqn: disper-nlaw}. 
We now introduce the following nearly-resonant sets:

\begin{itemize}
    \item \textbf{Tangential nearly-resonant set}: 
    \begin{equation}
        \label{eqn: tangential-nearly-resonant}
        \Omega_{0,M}^{\tau} := \bigcup_{k \in \Lambda_{2M}^m \setminus \{0\}} \left\{ \omega'\in\Omega ~\bigg|~ |\langle k,\omega'\rangle| < \frac{1}{|k|^\tau} \right\}.
    \end{equation}

    \item \textbf{Single-mode nearly-resonant set}: 
    \begin{equation}
        \label{eqn: single-nearly-resonant}
        \Omega_{1,M}^{\tau} :=
        \bigcup_{k \in \Lambda_{2M}^m} \bigcup_{n \in \mathbb{Z} \setminus (\pm \mathcal{N} \cup \{0\}) }
        \left\{ \omega' \in \Omega ~\bigg|~ \left| \langle k,\omega'\rangle - \varpi_n \right| < \frac{|n|}{(1+|k|)^\tau} \right\}.
    \end{equation}

    \item \textbf{Difference nearly-resonant sets}:
    \begin{equation}
        \label{eqn: difference-nearly-resonant-1}
        \Omega_{2,M}^{\tau,-} :=
        \bigcup_{k\in \Lambda_{2M}^{m}} \bigcup_{\substack{n_1, n_2 \in \mathbb{Z} \setminus (\pm \mathcal{N} \cup \{0\})  \\  |n_1| \neq |n_2|}}
        \left\{ \omega' \in \Omega ~\bigg|~ \left| \langle k,\omega'\rangle - \varpi_{n_1} + \varpi_{n_2} \right| < \frac{|n_1|+|n_2|}{(2+|k|)^\tau} \right\},
    \end{equation}
    and 
    \begin{equation}
        \label{eqn: difference-nearly-resonant-2}
        \Omega_{2,M}^{\tau,+} :=
        \bigcup_{k\in \Lambda_{2M}^{m}} \bigcup_{\substack{n_1, n_2 \in \mathbb{Z} \setminus (\pm \mathcal{N} \cup \{0\})}}
        \left\{ \omega' \in \Omega ~\bigg|~ \left| \langle k,\omega'\rangle - \varpi_{n_1} - \varpi_{n_2} \right| < \frac{|n_1|+|n_2|}{(2+|k|)^\tau} \right\}.
    \end{equation}
\end{itemize}

The total nearly-resonant set is the union of the sets defined 
in~\eqref{eqn: tangential-nearly-resonant}--\eqref{eqn: difference-nearly-resonant-2} as
\begin{equation}
    \label{eqn: total-nearly-resonant}
    \Omega_M^{\tau} := \Omega_{0,M}^{\tau} \cup \Omega_{1,M}^{\tau} \cup \left( \Omega_{2,M}^{\tau,+} \cup \Omega_{2,M}^{\tau,-} \right).
\end{equation}
The following lemma provides a measure bound for this excluded set.

\begin{lemma}
    \label[lemma]{lem: near-measure-initial}
    Let $\tau > m + 2$ be fixed. 
    There exists a constant $\kappa = \kappa(m, \tau, \Omega) > 0$, depending on 
    the dimension $m$, the parameter $\tau$ and the domain $\Omega$, such that 
    the nearly-resonant set satisfies
    \begin{equation}
        \label{eqn: near-measure-initial}
        \mathrm{mes}
        \left( \Omega_M^{\tau} \right) \leq \kappa.
    \end{equation}
\end{lemma}

The proof follows the strategy used for the nonlinear wave equation 
in~\citet[Section 6]{Fu2026numerical-pde-1}, with an explicit remainder estimate 
for the first-order asymptotic expansion of the normal frequencies. 
Specifically, we estimate the deviation of $\sqrt{n^4+n^2}$ from $n^2+1/2$. 
For any $n\neq 0$, we have
\begin{equation}
    \label{eqn: nlaw-freq-integer-estimate}
    \left| \sqrt{n^4 + n^2} - n^2 - \frac{1}{2} \right| = \frac{1}{4 \left(\sqrt{n^4+n^2} + n^2 + \frac12 \right) } \leq \frac{1}{8n^2}.
\end{equation}
The remainder of the proof consists of elementary estimates and is given 
in~\Cref{sec: near-measure-initial}.

\subsection{Restricted operators on small boxes}
\label{subsec: restricted-small}

In the small-scale regime, we establish invertibility and Gevrey decay estimates 
for the restrictions of $L$ to boxes centered at $(n_0, 0)$, with frequencies 
$\omega' \in \Omega \setminus \Omega_{M}^{\tau}$. 
The following theorem provides bounds that are uniform in the spatial center 
$n_0 \in \mathbb{Z}$.

\begin{theorem}
    \label{thm: small-size}
    Suppose that $\hat{u} \in \mathcal{K}(s)$ and 
    $\omega' \in \Omega \setminus \Omega_{M}^{\tau}$. 
    Then, for any $n_0 \in \mathbb{Z}$ and any scale $M_0 \leq N \leq M$, the 
    restricted operator $L_{n_0, N}$ is invertible and satisfies the operator 
    norm bound
    \begin{equation}
        \label{eqn: small-size-inverse}
        \| L_{n_0, N}^{-1} \|_2 \leq \frac{1}{\varepsilon_N}.
    \end{equation}
    Moreover, its off-diagonal entries satisfy the Gevrey decay estimate
    \begin{equation}
        \label{eqn: small-size-localization}
        \left| L_{n_0, N}^{-1} \left( (n, k), (n',k') \right) \right| \leq e^{ - ( |n-n'| + |k-k'|)^{s} },  
    \end{equation} 
    for any $(n, k) \neq (n', k') \in \Gamma_{n_0, N}$. 
\end{theorem}

By the exclusion of $\Omega_{M}^{\tau}$, the diagonal entries of $D_{n_0, N}$ 
satisfy
\begin{equation}
    \label{eqn: uniform-lower-original-small}
    | D_{n_0, N}(n, k) | \geq \frac{\left| - \langle k, \omega' \rangle^2 + n^4 + n^2 \right|}{n^2} \geq \frac{1}{(1 + mN)^{2\tau}},   
\end{equation}
for any non-resonant index $(n, k) \in \Gamma_{n_0, N} \setminus \mathcal{S}$. 
Thus, at each scale $N$, the diagonal entries of $D_{n_0, N}$ are bounded away 
from zero uniformly in $n_0 \in \mathbb{Z}$. 
This lower bound allows us to estimate $L_{n_0, N}^{-1}$ via a Neumann series. 
The Gevrey decay estimate then follows from standard convolution estimates. 
The detailed proof is deferred to~\Cref{sec: small-size-original}.

\subsection{Verification of the uniform conditions}
\label{subsec: verification}

In this section, we verify that the restrictions of $L$ to the central boxes 
satisfy the uniform conditions (\Cref{cond: implementation-proved}) throughout 
the large-scale regime $N > M$. 
The proof proceeds by induction, with~\Cref{thm: small-size} providing the base 
case in the small-scale regime $M_0 \leq N \leq M$. 
Assuming that~\Cref{cond: implementation-proved} holds at any scale 
$M_0 \leq N \leq M^r$, we extend these bounds to the next range of scales, 
$M^r < N \leq M^{r+1}$. 
The inductive step relies on estimates for the restrictions of $L$ to the outer 
lattice boxes $\Gamma_{n_0, k_0, N}$, where $n_0 \in \mathbb{Z}$ and 
$k_0 \notin \Lambda_{2N}^m$. 
The following theorem establishes the invertibility and off-diagonal decay 
estimates required for the induction. 
Its proof, based on multi-scale analysis, is presented in~\Cref{sec: multi-scale}.

\begin{theorem}
    \label{thm: off-side-box}
    Suppose $\hat{u} \in \mathcal{K}(s)$ and 
    $\omega' \in \Omega \setminus (\Omega_{M}^{\tau} \cup G)$, where the 
    exceptional set $G$ is defined in~\Cref{sec: multi-scale}. 
    Then for any $n_0 \in \mathbb{Z}$ and any $k_0 \notin \Lambda_{2N}^m$, the 
    restricted operator $L_{n_0, k_0, N} $ is invertible and satisfies: 
    \begin{equation}
        \label{eqn: off-side-inverse}
        \| L_{n_0, k_0, N}^{-1} \|_2 \leq \frac{1}{\varepsilon_N}.
    \end{equation}
    Moreover, for any $(n,k), (n', k') \in \Gamma_{n_0, k_0, N}$ such that 
    $|n-n'| + |k-k'| \geq N^{1/2}$, its entries satisfy:
    \begin{equation}
        \label{eqn: off-side-localization}
        \left| L_{n_0, k_0, N}^{-1} \left( (n, k), (n', k') \right) \right| \leq e^{ - \frac{ \left(|n-n'| + |k-k'| \right)^{s}}{2} }.
    \end{equation} 
\end{theorem}

For $M^r < N \leq M^{r+1}$, we set $K = N^{1/10}$. 
To obtain the spatially uniform bounds required 
by~\Cref{cond: implementation-proved}, we take the union over all integer spatial 
centers in $n_0 + \Lambda_N$, which leads to the following covering  
\begin{equation}
    \label{eqn: center-gluing-9k-implementation}
    \Gamma_{n_0, N} = \left( \bigcup_{n \in n_0 + \Lambda_N } \Gamma_{n, 10K} \right) \bigcup \left( \bigcup_{k \in \Lambda_N^m \setminus \Lambda_{9K}^m } \bigcup_{n \in n_0 + \Lambda_N } \Gamma_{n, k, K} \right). 
\end{equation}
Thus, each subbox of the temporal covering is replaced by a union of $2N+1$ 
spatial translates. 
The dispersion relation~\eqref{eqn: disper-nlaw} ensures that the measure of the 
union of the associated nearly-resonant sets remains controlled, as 
demonstrated by~\Cref{lem: near-measure-initial} in the small-scale regime. 
The gluing procedure is detailed below.

\paragraph{Central region ($n \in n_0 + \Lambda_N$ and $k \in \Lambda_{9K}^m$)} 
For each central subbox, we consider the decomposition 
$\Gamma_{n_0, N} = \Gamma_{n, 10K} \cup \Gamma_{n, 10K}^{c} $, where 
$\Gamma_{n, 10K}^{c} =  \Gamma_{n_0, N} \setminus \Gamma_{n, 10K}$.
Applying the resolvent identity, we obtain the block representation:
\begin{align}
    \begin{pmatrix} 
        L^{-1}_{n, 10K} & 0  \\ 
        0 & L^{-1}_{n, 10K,-} 
    \end{pmatrix} 
    = & \left[ \mathrm{Id} + \varepsilon 
    \begin{pmatrix} 
        0 & L_{n, 10K}^{-1}P^{*}  \\  
        L^{-1}_{n, 10K,-}P & 0 
    \end{pmatrix}  
    \right]  \nonumber \\
    & \cdot
    \begin{pmatrix} 
        (L^{-1}_{n_0, N})_{ \Gamma_{n, 10K} \times \Gamma_{n, 10K}  }    & (L^{-1}_{n_0, N})_{\Gamma_{n, 10K} \times \Gamma_{n, 10K}^{c}}  \\ 
        (L^{-1}_{n_0, N})_{ \Gamma_{n, 10K}^{c} \times \Gamma_{n, 10K} } & (L^{-1}_{n_0, N})_{\Gamma_{n, 10K}^{c} \times \Gamma_{n, 10K}^{c}} 
    \end{pmatrix},      \label{eqn: resolvent-center}
\end{align}
where $P$ and $P^*$ are the coupling operators inherited from $S+B$, and 
$L^{-1}_{n, 10K, -}$ denotes the inverse of the operator $L$ restricted to 
$\Gamma_{n, 10K} ^{c}$. 
Comparing the corresponding blocks in the first row, we obtain
\begin{subequations}
    \label{eqn: resolvent-center-relation}
    \begin{empheq}[left=\empheqlbrace]{align} 
        & (L^{-1}_{n_0, N})_{\Gamma_{n, 10K} \times \Gamma_{n, 10K} } = L^{-1}_{n, 10K} - \varepsilon L_{n, 10K}^{-1}P^* (L^{-1}_{n_0, N})_{ \Gamma_{n, 10K}^{c} \times \Gamma_{n, 10K} },    \label{eqn: resolvent-center-relation-1} \\
        & (L^{-1}_{n_0, N})_{\Gamma_{n, 10K} \times \Gamma_{n, 10K}^{c}} = - \varepsilon  L_{n, 10K}^{-1}P^*(L^{-1}_{n_0, N})_{\Gamma_{n, 10K}^{c} \times\Gamma_{n, 10K}^{c}}.                \label{eqn: resolvent-center-relation-2}
    \end{empheq}        
\end{subequations}
Combining~\eqref{eqn: resolvent-center-relation-1} 
and~\eqref{eqn: resolvent-center-relation-2} and extending the row operators by 
zero to the full column index set $\Gamma_{n_0, N}$, we obtain
\begin{equation}
    \label{eqn: resolvent-center-final}
    (L^{-1}_{n_0, N})_{ \Gamma_{n, 10K} \times \Gamma_{n_0, N} } = \begin{pmatrix} L_{n, 10K}^{-1} & 0 \end{pmatrix} - \varepsilon \begin{pmatrix} 0 & L_{n, 10K}^{-1}P^* \end{pmatrix} L_{n_0, N}^{-1}.
\end{equation}
The inductive hypothesis, together with~\Cref{prop: convolution} 
and~\Cref{prop: B-operator}, yields entrywise estimates for 
$L^{-1}_{n, 10K}P^{*}$.  
For $(n, k) \in \Gamma_{n, 9K} $ and 
$(n', k') \in (n_0 + \Lambda_N) \times (\Lambda_N^{m} \setminus \Lambda^m_{10K})$, 
we have $|n-n'| + |k-k'| \geq |k-k'| \geq K$. 
The corresponding entries of $(L^{-1}_{n, 10K}P^{*})\left( (n,k), (n',k') \right)$ 
therefore exhibit Gevrey decay. 
More precisely, splitting the defining sum into short-range and long-range 
contributions gives
\begin{equation}
    \label{eqn: inversion-localization-center}
    \left| (L^{-1}_{n, 10K}P^*)\left( (n,k), (n',k') \right) \right| \leq \left\{ 
    \begin{aligned}
        & 0,                                                            &&\mathrm{if}\; (n, k) = (n', k') ; \\
        & e^{ - \frac{  \left(|n - n'| + |k - k'| \right) ^{s}}{2} },   &&\mathrm{if}\; 0 <|n-n'| + |k-k'| < 100K^{\frac32},  \\
        & e^{ - \frac{ 3\left(|n - n'| + |k - k'| \right)^{s} }{4} },   &&\mathrm{if}\; |n-n'| + |k-k'| \geq 100K^{\frac32}.
    \end{aligned} 
    \right.
\end{equation}
Finally, restricting~\eqref{eqn: resolvent-center-final} to the rows indexed by 
$\{n\} \times \Lambda_{9K}^m$, we obtain
\begin{equation}
    \label{eqn: resolvent-center-final-9k}
    (L^{-1}_{n_0, N})_{ (\{n\} \times \Lambda_{9K}^m) \times \Gamma_{n_0, N} } = \begin{pmatrix} L_{n, 10K}^{-1} & 0 \end{pmatrix}_{ (\{n\} \times \Lambda_{9K}^m) \times \Gamma_{n_0, N} } - \varepsilon \begin{pmatrix} 0 & L_{n, 10K}^{-1}P^* \end{pmatrix}_{ (\{n\} \times \Lambda_{9K}^m) \times \Gamma_{n_0, N} } L_{n_0, N}^{-1},   
\end{equation}
where the entries of the zero-extended coupling operator 
$( 0, \; L_{n, 10K}^{-1}P^*)$, restricted to 
$(\{n\} \times \Lambda_{9K}^m) \times \Gamma_{n_0, N}$, satisfy the Gevrey decay 
bounds~\eqref{eqn: inversion-localization-center}.

\paragraph{Outer region ($n \in n_0 + \Lambda_N$ and $k \in \Lambda_{N}^{m} \setminus \Lambda_{9K}^m$)}  
For each lattice point outside the central region, we consider the decomposition 
$\Gamma_{n_0, N} = \Gamma_{n, k, K} \cup \Gamma_{n, k, K} ^{c}$, where 
$\Gamma_{n, k, K}^{c} = \Gamma_{n_0, N} \setminus \Gamma_{n, k, K}$. 
Repeating the block decomposition and resolvent identity argument used 
in~\eqref{eqn: resolvent-center-relation}--\eqref{eqn: resolvent-center-final-9k}, 
and restricting to the row indexed by $(n, k)$, we obtain
\begin{equation}
    \label{eqn: resolvent-outside-final-9k}
    (L^{-1}_{n_0, N})_{\{(n, k)\} \times \Gamma_{n_0, N} } = \begin{pmatrix} (L_{n, k, K}^{-1})_{\{(n,k)\} \times \Gamma_{n, k, K} } & 0 \end{pmatrix} - \varepsilon \begin{pmatrix} 0 & (L_{n, k, K}^{-1}P^*)_{\{ (n, k)\} \times \Gamma_{n, k, K}^{c}} \end{pmatrix} L_{n_0, N}^{-1},     
\end{equation}
where $P^*$ denotes the coupling operator associated with this decomposition, 
and both row operators are extended by zero to the full column index set 
$\Gamma_{n_0, N}$. 
As in the central region, splitting the defining sum into short-range and 
long-range contributions yields the entrywise Gevrey decay estimate
\begin{equation}
    \label{eqn: inversion-localization-outside}
    \left| (L^{-1}_{n, k, K}P^*)\left( (n,k), (n',k') \right) \right| \leq 
    \left\{ 
    \begin{aligned}
        & 0,                                                            &&\mathrm{if}\; |n-n'| + |k-k'| < K,                              \\
        & e^{ - \frac{  \left(|n - n'| + |k - k'| \right)^{s}}{2} },    &&\mathrm{if}\; K \leq |n-n'| + |k-k'| < K^{\frac32},     \\
        & e^{ - \frac{ 3\left(|n - n'| + |k - k'| \right)^{s}}{4} },    &&\mathrm{if}\; |n-n'| + |k-k'| \geq K^{\frac32}.
    \end{aligned} 
    \right.
\end{equation}

Assembling the central-row identities~\eqref{eqn: resolvent-center-final-9k} 
over all spatial centers $n \in n_0 + \Lambda_N$ and the outer-row 
identities~\eqref{eqn: resolvent-outside-final-9k} over all 
$(n,k) \in (n_0 + \Lambda_N) \times (\Lambda_{N}^{m} \setminus \Lambda_{9K}^{m})$, 
we express the global inverse as
\begin{equation}
    \label{eqn: resolvent-relation}
    L_{n_0, N}^{-1} = \Psi_{n_0,N} - \varepsilon \Phi_{n_0,N} L_{n_0, N}^{-1},
\end{equation}
where $\Psi_{n_0,N}$ is assembled row by row from the lower-scale inverses 
$L_{n, 10K}^{-1}$ and $L_{n, k, K}^{-1}$, with each local row extended by zero 
to the full column index set $\Gamma_{n_0, N}$. 
Specifically, for any $(n,k), (n',k') \in \Gamma_{n_0, N}$, its entries are 
given by 
\begin{multline}
    \Psi_{n_0, N}\left( (n,k), (n',k') \right) \\
    = \left\{ 
    \begin{aligned}
        & L_{n, 10K}^{-1}\left( (n,k), (n',k') \right),   &&\mathrm{if}\; (n,k) \in \left( n_0 + \Lambda_N \right) \times \Lambda_{9K}^m ,\; (n', k') \in \Gamma_{n, 10K},        \\
        & L_{n, k, K}^{-1}\left( (n,k), (n',k') \right),  &&\mathrm{if}\; (n,k) \notin \left( n_0 + \Lambda_N \right) \times \Lambda_{9K}^m ,\; (n', k') \in \Gamma_{n, k, K},    \\
        & 0,                                              &&\mathrm{otherwise}.   
    \end{aligned} 
    \right.                    \label{eqn: entries-psi}
\end{multline}
The operator $\Phi_{n_0, N}$ is assembled from the corresponding coupling rows. 
By~\eqref{eqn: inversion-localization-center} 
and~\eqref{eqn: inversion-localization-outside}, its entries satisfy
\begin{equation}
    \label{eqn: entries-phi}
    \left| \Phi_{n_0,N}\left( (n,k), (n',k') \right) \right| \leq \left\{ 
    \begin{aligned} 
        & 0,                                                        &&\mathrm{if}\; (n, k) = (n', k'), \\
        & e^{ - \frac{  \left(|n-n'| + |k-k'| \right)^{s}}{2} },    &&\mathrm{if}\; 0 < |n-n'| + |k-k'| < 100K^{\frac32},  \\ 
        & e^{ - \frac{ 3\left(|n-n'| + |k-k'| \right)^{s}}{4} },    &&\mathrm{if}\; |n-n'| + |k-k'| \geq 100K^{\frac32},
    \end{aligned} 
    \right.
\end{equation}                                    
Thus, the operator $\mathrm{Id} + \varepsilon \Phi_{n_0, N}$ is invertible, with 
its inverse represented by a norm-convergent Neumann series. 
To estimate its entries, we consider lattice paths 
$(n,k), (n_1,k_1), \ldots, (n',k')$ and distinguish between short-range and 
long-range jumps at the threshold $100K^{\frac32}$.
By~\eqref{eqn: entries-phi} and the concavity of the power function for $0<s<1$, 
the cumulative decay exponent along each path is bounded below by
\begin{align}
    \label{eqn: inverse-decay-short-inequality}
    \min\left\{ \frac34(|n-n'|+|k-k'|)^s, \; \frac{|n-n'|+|k-k'|}{2(100K^{\frac32})^{1-s}} \right\} \geq \frac34(|n-n'|+|k-k'|)^s, 
\end{align}
whenever $|n-n'| + |k-k'| > K^2$, for sufficiently large $K$. 
Allocating part of this decay to absorb the summation over intermediate lattice 
sites into the powers of $\varepsilon$, we obtain
\begin{equation}
    \label{eqn: inverse-decay-long}
    \left| (\mathrm{Id}+\varepsilon\Phi_{n_0,N})^{-1} \left((n,k),(n',k')\right) \right| \leq e^{-\frac58(|n-n'|+|k-k'|)^s}, \qquad \forall~|n-n'|+|k-k'|>K^2.
\end{equation}
The argument applies to every fixed $0<s<1$.
In the analytic case $s=1$, however, the strict concavity 
underlying~\eqref{eqn: inverse-decay-long} is absent. 
A stronger smallness condition on the perturbation parameter $\varepsilon$ is 
therefore required to compensate for the loss of this additional decay along 
purely short-range paths.

For sufficiently small $s>0$, the transition threshold 
in~\eqref{eqn: inversion-localization-center} can be lowered from 
$100K^{\frac32}$ to $100K$, making it linear in $K$. More generally, a refined 
Gevrey analysis involving~\Cref{lem: app-1} yields a two-regime decay estimate 
for the off-diagonal entries of 
$(\mathrm{Id} + \varepsilon \Phi_{n_0,N})^{-1}\left( (n,k), (n',k') \right)$: 
the long-range decay rate in~\eqref{eqn: inverse-decay-long} is retained, 
while the corresponding short-range estimate holds with the transition threshold 
increased from $100K$ to $K^2$. 
We can now rewrite~\eqref{eqn: resolvent-relation} as
\begin{equation}
    \label{eqn: l-n-inverse}
    L_{n_0, N}^{-1} = (\mathrm{Id} + \varepsilon \Phi_{n_0, N})^{-1} \Psi_{n_0,N}. 
\end{equation}
Finally,~\eqref{eqn: entries-psi} and the inductive bounds on the lower-scale 
inverses yield the $\ell_2$-operator norm estimate 
$\|\Psi_{n_0,N}\|_2 \leq (20K)^{1+m} \varepsilon_{10K}^{-1}$, while a bound for 
$\|\Phi_{n_0,N}\|_2$ follows from the Gevrey entry 
bound~\eqref{eqn: entries-phi}. 
Combining these bounds with the Neumann-series norm bound, the entrywise decay 
estimate~\eqref{eqn: inverse-decay-long} and the local support of $\Psi_{n_0, N}$, 
we conclude from~\eqref{eqn: l-n-inverse} that $L_{n_0, N}^{-1}$ satisfies the 
required norm and localization bounds at scale $N$, uniformly in the spatial 
center $n_0$. 
This completes the verification.

\section{Multi-scale analysis}
\label{sec: multi-scale}

In this section, we employ multiscale analysis to 
establish~\Cref{thm: off-side-box} for the ``outer'' boxes 
$\Gamma_{n_0, k_0, N}$, whose centers satisfy $n_0 \in \mathbb{Z}$ and 
$k_0 \notin \Lambda^m_{2N}$. 
For any $(n,k), (n',k') \in \Gamma_{n_0, k_0, N}$, the condition 
$k_0 \notin \Lambda^m_{2N}$ implies that $|k| + |k'| \geq 2N$. 
Consequently, the restricted operator satisfies
\[
    \|B_{n_0, k_0, N}\|_2 \leq \sqrt{ \|B_{n_0, k_0, N}\|_1 \cdot \|B_{n_0, k_0, N}\|_\infty } \leq (2N+1)^{1+m} e^{ -N^{s} }.
\]
We therefore regard $L_{n_0, k_0, N} = T_{n_0, k_0, N} + \varepsilon B_{n_0, k_0, N}$ 
as a small perturbation of $T_{n_0, k_0, N}$. 
Since $T_{n_0, k_0, N} $ is symmetric under interchange of the indices $(n, k)$ 
and $(n', k')$, its inverse norm is the reciprocal of the smallest absolute 
eigenvalue. 
In particular, if this eigenvalue is bounded below by $2\varepsilon_N$, then 
$\|T_{n_0, k_0, N}^{-1} \|_2 \leq \varepsilon_N^{-1}/2$. 
Together with the rapid decay of $B_{n_0, k_0, N}$, we therefore obtain
\[
    \varepsilon \|T_{n_0, k_0, N}^{-1}\|_2 \|B_{n_0, k_0, N}\|_2 \leq \varepsilon (2N+1)^{1+m} e^{ (\log N)^{15} - N^s} \leq \frac{1}{2}.
\]
Applying the Neumann series to the factorization 
\[
    L_{n_0, k_0, N} = T_{n_0, k_0, N} \left( \mathrm{Id} + \varepsilon T^{-1}_{n_0, k_0, N} B_{n_0, k_0, N} \right),
\]
we conclude that $L_{n_0, k_0, N}$ is invertible with 
$\|L_{n_0, k_0, N}^{-1}\|_2 \leq \varepsilon_{N}^{-1}$. 
Thus, to establish the desired inverse norm bound~\eqref{eqn: off-side-inverse}, 
it suffices to show that $\|T_{n_0, k_0, N}^{-1} \|_2 \leq \varepsilon_N^{-1}/2$,
and the remainder of this section is devoted to establishing this bound together 
with the corresponding localization estimate.

\subsection{Inductive base for small boxes}
\label{subsec: small-outer}

To establish the base case for the induction in~\Cref{sec: inversion-localization}, 
we require inverse estimates for the small boxes in the regime 
$M_0 \leq N \leq M$. 
The two preliminary results of this subsection, the classification of 
singular sites (\Cref{prop: small-eig}) and the inverse estimate on the 
nonsingular set (\Cref{lem: small-singular}), are in fact valid in the wider 
range $1 \leq N \leq M$; only the ``inversion implies localization'' argument 
for the full box (\Cref{lem: small-inverse-localization}) requires 
$M_0 \leq N \leq M$.

For any $(n_1, k_1), (n_2, k_2) \in \Gamma_{n_0, k_0, N}$, we have 
$|k_1 - k_2| \leq 2N \leq 2 \min\{|k_1|, |k_2|\}$. 
Factorizing $\langle k, \omega' \rangle^2 - \varpi_n^2 = \left( \langle k, \omega' \rangle + \varpi_n \right) \left( \langle k, \omega' \rangle - \varpi_n \right)$ 
and applying the triangle inequality, we obtain the following inclusions
\begin{align}
    & \left\{ \omega' \in \Omega \bigg| \left| \langle k_1,\omega'\rangle^2 - \varpi_{n_1}^2 \right| < \frac{n_1^2}{4(2+2|k_1|)^{2\tau}}~\mathrm{and}~\left| \langle k_2,\omega'\rangle^2 - \varpi_{n_2}^2 \right| < \frac{n_2^2}{4(2+2|k_2|)^{2\tau}} \right\}   \nonumber  \\
    & \subseteq \left( \left\{ \omega' \in \Omega \bigg| \left| \langle k_1,\omega' \rangle + \varpi_{n_1} \right| < \frac{|n_1|}{2(2+2|k_1|)^{\tau}}\right\} \bigcup \left\{ \omega' \in \Omega \bigg| \left| \langle k_1,\omega' \rangle - \varpi_{n_1} \right| < \frac{|n_1|}{2(2+2|k_1|)^{\tau}}\right\} \right)    \nonumber  \\
    & \quad \bigcap \left( \left\{ \omega' \in \Omega \bigg| \left| \langle k_2,\omega' \rangle + \varpi_{n_2} \right| < \frac{|n_2|}{2(2+2|k_2|)^{\tau}}\right\} \bigcup \left\{ \omega' \in \Omega \bigg| \left| \langle k_2,\omega' \rangle - \varpi_{n_2} \right| < \frac{|n_2|}{2(2+2|k_2|)^{\tau}}\right\} \right)    \nonumber  \\
    & \subseteq \left\{ \omega' \in \Omega \bigg| \left| \langle k_1 - k_2,\omega'\rangle - \varpi_{n_1} + \varpi_{n_2} \right| < \frac{|n_1|+|n_2|}{(2+|k_1 - k_2|)^{\tau}}\right\}    \nonumber  \\
    & \quad \bigcup \left\{ \omega' \in \Omega \bigg| \left| \langle k_1 - k_2,\omega'\rangle - \varpi_{n_1} - \varpi_{n_2} \right| < \frac{|n_1|+|n_2|}{(2+|k_1 - k_2|)^{\tau}}\right\}.       \label{eqn: inclusion-small}
\end{align}
For any $\omega' \in \Omega \setminus \Omega_M^{\tau}$, the exclusion 
of~\eqref{eqn: tangential-nearly-resonant}--\eqref{eqn: difference-nearly-resonant-2} 
yields the following dichotomy: either 
\[
    | - \langle k_1 - k_2, \omega' \rangle + \varpi_{n_1} \mp \varpi_{n_2} | \geq \frac{|n_1|+|n_2|}{(4mN)^{\tau}}
\]
or the two corresponding factors form a nearly resonant pair, satisfying
\begin{equation}
    \label{eqn: d-eigenvalue-equal}
    | - \langle k_1 - k_2, \omega' \rangle + \varpi_{n_1} \mp \varpi_{n_2} | \leq 2\sqrt{\varepsilon}. 
\end{equation}
In the latter case, the exclusion of the single-mode nearly-resonant set 
$\Omega_{1,M}^{\tau}$ in~\eqref{eqn: single-nearly-resonant} rules 
out mixed tangential-normal pairs: the spatial indices $n_1$ and $n_2$ must 
either both belong to $\pm \mathcal{N}$ or both lie outside $\pm \mathcal{N}$. 
These two cases have the following structure. 
\begin{itemize}
    \item \textbf{Tangential case}. 
    Suppose $n_1, n_2 \in \pm \mathcal{N}$. 
    Without loss of generality, assume $\varpi_{n_1} = \omega_1$ and 
    $\varpi_{n_2} = \omega_2$.
    Then the difference can be rewritten as 
    \begin{align}
        - \langle k_1 - k_2, \omega' \rangle + \varpi_{n_1} \mp \varpi_{n_2} 
        &= - \langle k_1 - k_2, \omega' \rangle + \omega_{1} \mp \omega_{2} \nonumber\\
        &= - \langle k_1 - k_2 - e_{1} \pm e_{2}, \omega' \rangle  + (\omega_1-\omega_1') \mp (\omega_2-\omega_2').  \label{eqn: d-eigenvalue-gap-1}
    \end{align}   
    By the exclusion of the tangential nearly-resonant 
    set~\eqref{eqn: tangential-nearly-resonant}, the shifted temporal difference 
    must vanish. 
    Thus, $k_1 - e_{1} = k_2 - e_{2}$ in the minus case and 
    $k_1 - e_{1} = k_2 + e_{2}$ in the plus case. 
    Consequently, all sites in the same tangential equality class share a common 
    shifted temporal index $k_{\star}$. 
    The temporal indices associated with the $j$-th tangential mode therefore 
    satisfy $k_j^{+}=k_{\star} + e_j$ and $k_j^{-} = k_{\star} - e_j$ for 
    $j=1, \ldots, m$ represented in the class. 

    \item \textbf{Normal case}. 
    Suppose $n_1, n_2 \notin \pm \mathcal{N}$. 
    The exclusion of $\Omega_{2,M}^{\tau,-}$~\eqref{eqn: difference-nearly-resonant-1} 
    and $\Omega_{2,M}^{\tau,+}$ in~\eqref{eqn: difference-nearly-resonant-2} 
    forces $\varpi_{n_1} = \varpi_{n_2}$, and the 
    equality~\eqref{eqn: d-eigenvalue-equal} therefore reduces to
    \begin{equation}
        \label{eqn: d-eigenvalue-gap-2}
        - \langle k_1 - k_2, \omega' \rangle = 0.
    \end{equation}
    The exclusion of~\eqref{eqn: tangential-nearly-resonant} then gives 
    $k_1 = k_2:=k_{\star}$. 
    Moreover, $\varpi_{n_1} = \varpi_{n_2}$ implies $n_1 = \pm n_2$. 
    Consequently, all sites in the same normal equality class are contained in 
    a pair of the form $\{ (n_{\star}, k_{\star}), (-n_{\star}, k_{\star}) \}$, 
    where $n_{\star} \notin \pm \mathcal{N}$. 
\end{itemize}
For any box $\Gamma_{n_0, k_0, N}$ with $n_0 \in \mathbb{Z}$ and 
$1 \leq N \leq M$, we define the set of \textit{singular sites} as: 
\begin{equation}
    \label{eqn: singular-sites}
    \Pi := \left\{ (n, k) \in \Gamma_{n_0, k_0, N} ~\big|~ |D_{n_0, k_0, N}(n, k)| \leq 4 \varepsilon_N \right\}. 
\end{equation}
The preceding classification therefore gives the following proposition.

\begin{prop}
    \label[proposition]{prop: small-eig}
    Suppose $\hat{u} \in \mathcal{K}(s)$ and 
    $\omega' \in \Omega \setminus \Omega_M^\tau$. 
    If $\Pi$ is nonempty, then it is contained in a resonant configuration of 
    exactly one of the following two types:
    \begin{itemize}
        \item[\textbf{(I)}] \textbf{Tangential configuration}: There exists  
        a common shifted temporal index $k_{\star} \in k_0 + \Lambda_N^m$ 
        such that 
        \begin{equation}
            \label{eqn: singular-sites-tangent}
            \Pi \subseteq \left\{ (\pm n_j, k_{\star} \pm  e_j ) ~\big|~ j=1, \ldots, m \right\},
        \end{equation}
        where $n_j \in \mathcal{N}$ for $j=1, \ldots, m$ are the tangential 
        indices. 

        \item[\textbf{(II)}] \textbf{Normal configuration}: There exist 
        a temporal index $k_{\star} \in k_0 + \Lambda_N^m$ and a spatial 
        index $n_{\star} \notin \pm \mathcal{N}$ such that
        \begin{equation}
            \label{eqn: singular-sites-normal}
            \Pi \subseteq \left\{ ( n_{\star}, k_{\star} ), \; ( - n_{\star}, k_{\star} ) \right\}.
        \end{equation}
    \end{itemize}
    In particular, $\Pi$ cannot contain both tangential and normal sites. 
\end{prop}

\paragraph{Inverse on the nonsingular set} We next estimate the inverse of the 
restriction of $T$ to the nonsingular set $\Gamma_{n_0, k_0, N} \setminus \Pi$ 
for $k_0 \notin \Lambda_{2N}^m$. 
Applying~\Cref{prop: small-eig} and arguing as in the proof 
of~\Cref{thm: small-size}, we obtain the following lemma.

\begin{lemma}
    \label[lemma]{lem: small-singular}
    Suppose that $\hat{u} \in \mathcal{K}(s)$ and 
    $\omega' \in \Omega \setminus \Omega_{M}^{\tau}$. 
    Then, for $k_0 \notin \Lambda_{2N}^m$, the 
    restricted operator $T_{\Gamma_{n_0, k_0, N} \setminus \Pi}$ is invertible, 
    and its inverse satisfies 
    \begin{equation}
        \label{eqn: small-size-inverse-prepare}
        \| T_{\Gamma_{n_0, k_0, N} \setminus \Pi}^{-1} \|_2 \leq \frac{1}{2\varepsilon_N}.
    \end{equation}
    Moreover, its off-diagonal entries satisfy 
    \begin{equation}
        \label{eqn: small-size-off-diagonal-prepare}
        \left| T_{\Gamma_{n_0, k_0, N} \setminus \Pi}^{-1} \left( (n,k), (n',k') \right) \right| \leq e^{ - (|n-n'| + |k-k'|)^{s} } < e^{ - \frac{ (|n-n'| + |k-k'|)^{s}}{2} },
    \end{equation}
    for any $(n,k) \neq (n', k') \in \Gamma_{n_0, k_0, N} \setminus \Pi$.
\end{lemma}

\paragraph{Inversion implies localization}
The preceding analysis applies 
throughout the range $1 \leq N \leq M$. 
We now restrict attention to the small-scale regime $M_0 \leq N \leq M$ and 
establish the estimates needed for the base case of the induction 
in~\Cref{sec: inversion-localization}. 
For this purpose, we only need the Gevrey decay estimates 
from~\Cref{lem: small-singular} for the inverse of the nonsingular restriction 
at separations $|n-n'| + |k-k'| \geq N^{\frac12}$. 
We apply the ``inversion implies localization'' argument to transfer these 
decay estimates to the full inverse. 
With respect to the decomposition 
$\Gamma_{n_0, k_0, N} = (\Gamma_{n_0, k_0, N} \setminus \Pi) \cup \Pi$, the 
resolvent identity gives
\begin{align}
    \begin{pmatrix} 
        T_{\Gamma_{n_0, k_0, N} \setminus \Pi}^{-1} & 0 \\ 
        0 & T_{\Pi}^{-1} 
    \end{pmatrix} 
    = & \left[ \mathrm{Id} + \varepsilon 
    \begin{pmatrix} 
        0 & T_{\Gamma_{n_0, k_0, N} \setminus \Pi}^{-1} P^* \\ 
        T_{\Pi}^{-1} P & 0 
    \end{pmatrix}  
    \right]    \nonumber \\
    & \cdot  \left( \begin{array}{ll} 
        (T^{-1}_{n_0, k_0, N})_{(\Gamma_{n_0,k_0, N} \setminus \Pi) \times (\Gamma_{n_0, k_0, N} \setminus \Pi)} & (T^{-1}_{n_0, k_0, N})_{(\Gamma_{n_0,k_0, N} \setminus \Pi) \times \Pi}  \\ 
        (T^{-1}_{n_0, k_0, N})_{\Pi \times (\Gamma_{n_0, k_0, N} \setminus \Pi)} \mathrel{\phantom{=l}} & (T^{-1}_{n_0, k_0, N})_{\Pi \times \Pi} \mathrel{\phantom{=l}}
    \end{array} \right).     \label{eqn: resolvent-small-inversion-local}
\end{align}
Comparing the corresponding blocks 
in~\eqref{eqn: resolvent-small-inversion-local}, we obtain:
\begin{subequations}
    \label{eqn: resolvent-small-inversion-local-relation}
    \begin{empheq}[left=\empheqlbrace]{align} 
        & (T_{n_0,  k_0, N}^{-1})_{(\Gamma_{n_0, k_0, N} \setminus \Pi) \times (\Gamma_{n_0, k_0,N} \setminus \Pi)} \nonumber \\
        & \mathrel{\phantom{=============}}= T_{\Gamma_{n_0, k_0, N} \setminus \Pi}^{-1} - \varepsilon T_{\Gamma_{n_0, k_0, N} \setminus \Pi}^{-1} P^* (T_{n_0, k_0, N}^{-1})_{\Pi \times (\Gamma_{n_0, k_0, N} \setminus \Pi)},                                                                                                                                                                                \label{eqn: resolvent-small-inversion-local-relation-1} \\
        & (T_{n_0, k_0, N}^{-1})_{(\Gamma_{n_0, k_0, N} \setminus \Pi) \times \Pi} = - \varepsilon T_{\Gamma_{n_0, k_0, N} \setminus \Pi}^{-1} P^* (T_{n_0, k_0, N}^{-1})_{\Pi \times \Pi}.            \label{eqn: resolvent-small-inversion-local-relation-2}
    \end{empheq}        
\end{subequations}
The decay estimates~\eqref{eqn: small-size-off-diagonal-prepare} for 
$T_{\Gamma_{n_0, k_0,N} \setminus \Pi}^{-1} $ and~\eqref{eqn: tangent-comp} for
$P^*$ imply that, whenever $|n-n'| + |k-k'| \geq N^{\frac12}/4$, the coupling 
operator satisfies the following estimate 
\begin{equation}
    \label{eqn: coupling-operator-gevrey-small}
    \left| \left( T_{\Gamma_{n_0, k_0,N} \setminus \Pi}^{-1} P^* \right) \left( (n, k), (n', k') \right) \right| \leq e^{ - \frac{2 (|n-n'| + |k-k'|)^s}{3}}.
\end{equation}

As in the bounded perturbation case treated in~\citet{Fu2026numerical-pde-1}, we first consider the tangential 
configuration~\eqref{eqn: singular-sites-tangent}. 
By the definition of the set of tangential modes $\mathcal{N}$, we have
\[
    M_0 \geq \left(2 \max_{n_1, n_2 \in \mathcal{N}} |n_{1} \pm n_{2}| + 2 \right)^{3}. 
\]
Consequently, any two singular sites $(n', k'), (n'', k'') \in \Pi$ satisfy 
$|n'-n''| + |k'-k''| \leq M_0^{\frac13}$. 
Let $(n, k) \in \Gamma_{n_0, k_0,N} \setminus \Pi$ be a nonsingular site and 
$(n', k') \in \Pi$ be a singular site such that $|n-n'| + |k-k'| \geq N^{\frac12}/2$. 
In the block 
identities~\eqref{eqn: resolvent-small-inversion-local-relation-1}--\eqref{eqn: resolvent-small-inversion-local-relation-2}, 
the potentially problematic contributions involve intermediate singular sites 
$(n'', k'') \in \Pi$. 
The diameter bound above ensures, however, that each such site remains 
sufficiently far from $(n, k)$. 
Indeed, provided that $N^{\frac12} \geq 6M_0^{\frac13}$, the triangle 
inequality yields
\begin{equation}
    \label{eqn: jump-singular-small-tangential}
    |n-n''| + |k-k''| \geq |n-n'| + |k-k'| - M_0^{\frac13} \geq \frac{2 \left( |n-n'| + |k-k'|\right)}{3}. 
\end{equation}
Applying the estimate~\eqref{eqn: jump-singular-small-tangential} in the block 
identities~\eqref{eqn: resolvent-small-inversion-local-relation-1}--\eqref{eqn: resolvent-small-inversion-local-relation-2}, 
together with the symmetry of $T_{n_0, k_0, N}^{-1}$, we extend the decay 
estimates to the full inverse. 
Thus, under the inverse bound 
$\|T_{n_0, k_0, N}^{-1} \|_2 \leq \varepsilon_{N}^{-1}/2$, we obtain the 
localization estimate~\eqref{eqn: full-gevrey-small} whenever 
$|n-n'| + |k-k'| \geq N^{\frac12}$.

We next consider the normal configuration~\eqref{eqn: singular-sites-normal}, 
distinguishing between the cases $|n_{\star}| \leq N^{\frac13}/2$ and 
$|n_{\star}| > N^{\frac13}/2$. 
In the first case, $\Pi$ contains at most two singular sites whose distance is at 
most $2|n_{\star}| \leq N^{\frac13}$. 
Let $(n, k) \in \Gamma_{n_0, k_0, N} \setminus \Pi$ and $(n', k') \in \Pi$ 
satisfying $|n-n'| + |k-k'| \geq N^{\frac12}/2$, and denote the other singular 
site in $\Pi$ by $(n'', k'')$. 
Provided that $N^{\frac12} \geq 6N^{\frac13}$, the triangle inequality yields   
\begin{equation}
    \label{eqn: jump-singular-small}
    |n-n''| + |k-k''| \geq |n-n'| + |k-k'| - N^{\frac13} \geq \frac{2 \left( |n-n'| + |k-k'| \right) }{3}. 
\end{equation}
Inserting the estimate~\eqref{eqn: jump-singular-small} into the block 
identities~\eqref{eqn: resolvent-small-inversion-local-relation-1}--\eqref{eqn: resolvent-small-inversion-local-relation-2}, 
and using the symmetry of $T_{n_0, k_0, N}^{-1}$, we extend the decay 
estimates of~\Cref{lem: small-singular} from $T_{\Gamma_{n_0, k_0, N} \setminus \Pi}^{-1}$ 
to the full inverse. 
Thus, under the inverse bound 
$\|T_{n_0, k_0, N}^{-1} \|_2 \leq \varepsilon_{N}^{-1}/2$, we obtain the 
localization estimate
\begin{equation}
    \label{eqn: full-gevrey-small}
    \left| \left( T_{n_0, k_0,N}^{-1} \right) \left( (n, k), (n', k') \right) \right| \leq e^{ - \frac{(|n-n'| + |k-k'|)^s}{2} }. 
\end{equation}
for any $|n-n'| + |k-k'| \geq N^{\frac12}$. In the second case, the two possible singular sites are separated by a distance 
$2|n_{\star}| > N^{\frac13}$. 
If both lie in the lattice box $\Gamma_{n_0, k_0,N}$, the diameter of $\Pi$ 
need not be small relative to $N^{\frac{1}{2}}$, so a triangle-inequality 
estimate of the form~\eqref{eqn: jump-singular-small-tangential} and~\eqref{eqn: jump-singular-small} is no longer guaranteed. 
To overcome this difficulty, we exploit the fact that~\Cref{lem: small-singular} 
holds at any scale between $1$ and $M$, and consider subboxes 
$\Gamma_{n,k,K} \subset \Gamma_{n_0, k_0,N}$ of scale $K = N^{\frac{1}{10}}$. 
Since their diameters are smaller than $N^{\frac{1}{3}}$, each subbox 
$\Gamma_{n,k,K}$ contains at most one singular site. Moreover, if $\|T_{n_0, k_0, N}^{-1}\|_2 \leq \varepsilon_N^{-1}/2$, then the 
Schur complement formula yields the local inverse bound $\|T_{n, k, K}^{-1}\|_2 \leq \varepsilon_N^{-1}/2$. 
Hence the block-identity argument applies within each such subbox. 
Under the local inverse bound $\|T_{n, k, K}^{-1} \|_2 \leq \varepsilon_{N}^{-1}/2$, 
we obtain the corresponding localization 
estimates~\eqref{eqn: full-gevrey-small}, with $N$ replaced by $K$ whenever 
$|n-n'| + |k-k'| \geq K^{\frac12}$. 
Finally, we invoke the gluing argument of~\Cref{subsec: verification} to 
propagate these local inverse estimates to the full box $\Gamma_{n_0, k_0,N}$. 
The details are given as follows.

With respect to the decomposition 
$\Gamma_{n_0, k_0, N} = \Gamma_{n, k, K} \cup \Gamma_{n, k, K}^{c}$, where 
$\Gamma_{n, k, K}^{c} = \Gamma_{n_0, k_0, N} \setminus \Gamma_{n, k, K}$, the 
resolvent identity gives the block representation  
\begin{align}
    \begin{pmatrix} 
        T^{-1}_{n,k, K} & 0  \\ 
        0 & T^{-1}_{n, k, K,-} 
    \end{pmatrix} 
    = & \left[ \mathrm{Id} + \varepsilon 
    \begin{pmatrix} 
        0 & T_{n,k, K}^{-1}P^{*}  \\  
        T^{-1}_{n,k, K,-}P & 0 
    \end{pmatrix}  
    \right]  \nonumber \\
      &  \cdot
    \begin{pmatrix} 
        (T^{-1}_{n_0, k_0, N})_{\Gamma_{n, k, K} \times \Gamma_{n, k, K} } & (T^{-1}_{n_0, k_0, N})_{\Gamma_{n, k, K} \times \Gamma_{n, k,K}^{c}}  \\ 
        (T^{-1}_{n_0, k_0, N})_{\Gamma_{n, k, K}^{c} \times \Gamma_{n, k, K} } & (T^{-1}_{n_0,k_0, N})_{\Gamma_{n, k, K}^c \times \Gamma_{n, k, K}^c} 
    \end{pmatrix},      \label{eqn: resolvent-small-offside}
\end{align}
where $P$ and $P^*$ are the coupling operators inherited from $S$, and 
$T^{-1}_{n, k, K, -}$ denotes the inverse of the operator $T$ restricted to 
$\Gamma_{n, k, K}^{c}$. 
Combining the two identities in the first row 
of~\eqref{eqn: resolvent-small-offside}, extending the terms to the full 
index set $\Gamma_{n_0, k_0, N}$, and extracting the row indexed by $(n,k)$, 
we obtain
\begin{align}
    \left( T_{n_0, k_0,N}^{-1} \right)_{ \{ (n, k) \} \times \Gamma_{n_0, k_0, N} } = 
    & \begin{pmatrix} (T_{n,k, K}^{-1})_{ \{ (n, k) \} \times \Gamma_{n, k, K}} & 0 \end{pmatrix} \nonumber \\ 
    & - \varepsilon \begin{pmatrix} 0 & (T_{n,k,K}^{-1}P^*)_{\{ (n, k) \} \times \Gamma_{n, k, K}^{c}} \end{pmatrix} T_{n_0, k_0, N}^{-1}.  \label{eqn: block-representation-small-offside}
\end{align}
Assembling~\eqref{eqn: block-representation-small-offside} over all indices 
$(n, k) \in \Gamma_{n_0, k_0, N}$ yields the operator identity
\begin{equation}
    \label{eqn: resolvent-relation-small-offside}
    T_{n_0, k_0, N}^{-1} = \Psi_{n_0, k_0, N} - \varepsilon \Phi_{n_0, k_0, N} T_{n_0, k_0, N}^{-1}, 
\end{equation}
where the linear operators $\Psi_{n_0, k_0, N}$ and $\Phi_{n_0, k_0, N}$ satisfy
\begin{equation}
    \Psi_{n_0, k_0, N}\left( (n,k), (n',k') \right) 
    = \left\{ 
    \begin{aligned}       
        & T_{n,k, K}^{-1}\left( (n,k), (n',k') \right),  &&\mathrm{if}\;  (n', k')  \in \Gamma_{n,k, K},    \\
        & 0,                                             &&\mathrm{otherwise},   
    \end{aligned} 
    \right.   \label{eqn: entries-psi-small-offside}
\end{equation}
and
\begin{equation}
    \label{eqn: entries-phi-small-offside}
    \left| \Phi_{n_0, k_0, N}\left( (n,k), (n',k') \right) \right| \leq \left\{ 
    \begin{aligned} 
        & 0,                                                               &&\mathrm{if}\; |n - n'| + |k - k'| < K, \\
        & e^{- \frac{ \left(|n - n'| + |k - k'| \right)^{s}}{2} },         &&\mathrm{if}\;  K \leq |n - n'| + |k - k'|  < K^{\frac32},  \\ 
        & e^{ - \frac{ 3\left(|n - n'| + |k - k'| \right)^{s}}{4} },       &&\mathrm{if}\; |n - n'| + |k - k'| \geq K^{\frac32}.
    \end{aligned} 
    \right.
\end{equation} 
Following the procedure in~\Cref{subsec: verification}, we divide the paths 
into two classes and show that $(\mathrm{Id} + \varepsilon \Phi_{n_0, k_0, N})^{-1}$ satisfies an estimate of 
the form~\eqref{eqn: inverse-decay-long} whenever 
$|n-n'| + |k - k'| \geq N^{\frac14}$. 
From~\eqref{eqn: entries-psi-small-offside}, we further obtain 
$\|\Psi_{n_0, k_0, N}\|_{2} \leq N^{m+1} \varepsilon_N^{-1}$. 
We thus obtain the required localization bounds for $T_{n_0, k_0, N}^{-1}$ at 
any scale $N$ in the small scale regime, uniformly in the 
spatial center $n_0$. 
We summarize this conclusion in the following ``inversion implies localization'' 
lemma.

\begin{lemma}
    \label[lemma]{lem: small-inverse-localization}
    Suppose that $\hat{u} \in \mathcal{K}(s)$, 
    $\omega' \in \Omega \setminus \Omega_M^\tau$, and $M_0 \leq N \leq M$. 
    If the restricted operator $T_{n_0, k_0, N}$ is invertible and satisfies 
    the norm bound:
    \begin{equation}
        \label{eqn: trans-small-size-original-inverse}
        \| T_{n_0,  k_0, N}^{-1} \|_2 \leq \frac{1}{2\varepsilon_N},
    \end{equation}
    then its entries satisfy the Gevrey decay estimate
    \begin{equation}
        \label{eqn: trans-small-size-off-diagonal}
        \left| T_{n_0, k_0, N}^{-1}\left( (n, k), (n', k') \right) \right| \leq e^{-\frac{\left( |n - n'| + |k - k'| \right)^s}{2}}, 
    \end{equation}
    for any $|n-n'| + |k-k'| \geq N^{\frac12}$. 
\end{lemma}

\subsection{Scale reduction: temporal-index clustering of singular sites}
\label{sec: scale-reduce-cluster}

We now turn to the large-scale regime $N > M$. 
When the singular set $\Pi$ has either the tangential 
configuration~\eqref{eqn: singular-sites-tangent} or the normal 
configuration~\eqref{eqn: singular-sites-normal} with 
$|n_{\star}| \leq N^{\frac13}/2$, we work with lattice boxes of size $N$. 
For the normal configuration~\eqref{eqn: singular-sites-normal} of 
$|n_{\star}| > N^{\frac13}/2$, we instead work with boxes of size 
$N^{\frac{1}{10}}$. 
The analysis below applies to both cases; for notational convenience, we 
denote the corresponding box size by $N$ throughout this subsection and the 
next, and return to the distinction between the two scales at the end of the 
next subsection.

As in the small-scale analysis of~\Cref{subsec: small-outer}, we seek to exclude 
the set of drifted frequencies for which two distinct eigenvalues of the 
linearized operator $T$ are simultaneously small. 
The large-scale regime, however, requires a different approach. 
At small scales, the diagonal operator $D$ can be treated as the dominant part 
of $T$. At large scales, the small eigenvalues of $D$ may be comparable to, 
or even smaller than, the perturbation parameter $\varepsilon$. 
The dominance required by the small-scale argument is therefore no longer 
available, and we must work directly with the perturbed eigenvalues of $T$.

The eigenvalues of the restricted operator $T_{n_0, k_0, N}$ are given by
\begin{equation}
    \label{eqn: t-n0-N-eigen}
    T_{n_0, k_0, N}(n,k) = \frac{- \langle k, \omega' \rangle^2 + \varpi_n^2}{n^2} + \nu(n, k), 
\end{equation}
where the perturbation $\nu(n, k)$ depends on the drifted frequency $\omega'$. 
To estimate the measure of the excluded frequency set, we instead consider
\begin{equation}
    \label{eqn: t-n0-N-eigen-n-sqare}
    n^2 T_{n_0, k_0, N}(n,k) = - \langle k, \omega' \rangle^2 + \varpi_n^2 + n^2 \nu(n, k).    
\end{equation} 
To factor~\eqref{eqn: t-n0-N-eigen-n-sqare} into two linear expressions in 
$\langle k, \omega' \rangle$, we take the square root of the perturbed spatial 
part and write
\begin{equation}
    \label{eqn: equal-freq-expan-nlaw}
    \sqrt{\varpi_n^2 + n^2 \nu(n, k)} = \varpi_n + \mu(n,k), 
\end{equation}
where, for $n \neq 0$, the perturbation $\mu(n,k)$ is given by 
\begin{equation}
    \label{eqn: nu-nlaw}
    \mu(n, k) = \frac{\nu(n,k)}{\sqrt{1 + \frac{1 + \nu(n, k)}{n^2} } + \sqrt{1 + \frac{1}{n^2}}}. 
\end{equation}
By~\Cref{prop: convolution}, we have the derivative estimate 
\begin{equation}
    \label{eqn: perturb-omega-t}
    \max_{(n,k) \in \mathbb{Z}^{1+m}} \max \big\{ |\nu(n,k)|, \; \left\| \partial_{\omega'} \nu(n,k) \right\|_2 \big\} \leq \varepsilon \left\| S\right\| \leq \frac{\sqrt{\varepsilon}}{2}. 
\end{equation}
Applying~\eqref{eqn: perturb-omega-t} to~\eqref{eqn: nu-nlaw} and differentiating 
with respect to $\omega'$, we obtain
 \begin{equation}
    \label{eqn: mu-nu-nlaw}
    |\mu(n,k)| \leq |\nu(n,k)| \qquad \mathrm{and} \qquad \|\partial_{\omega'}\mu(n,k)\|_2 \leq \|\partial_{\omega'}\nu(n,k)\|_2.
\end{equation}
Combining~\eqref{eqn: equal-freq-expan-nlaw},~\eqref{eqn: nu-nlaw}, 
and~\eqref{eqn: mu-nu-nlaw}, we see that although the NLAW 
equation~\eqref{eqn: nlaw} carries an unbounded perturbation, the measure 
exclusion in the large-scale regime is essentially the same as in the 
bounded-perturbation setting of~\citet{Fu2026numerical-pde-1}. 
A more substantial difference arises for the KdV equation~\eqref{eqn: kdv}, as 
shown in~\Cref{subsec: main-kdv}.

Arguing as in the derivation of the small-scale 
inclusion~\eqref{eqn: inclusion-small}, we obtain
\begin{align}
    & \left\{ \omega' \in \Omega \bigg| \left| \langle k_1, \omega'\rangle^2 - \varpi_{n_1}^2 - n_1^2 \nu(n_1, k_1) \right| < 2n_1^2\varepsilon_N~\mathrm{and}~\left| \langle k_2, \omega'\rangle^2 - \varpi_{n_2}^2 - n_2^2 \nu(n_2, k_2) \right| < 2n_2^2\varepsilon_N \right\}   \nonumber  \\
    & \subseteq \left( \left\{ \omega' \in \Omega \bigg| \left| \langle k_1, \omega' \rangle + \varpi_{n_1} + \mu(n_1, k_1) \right| < |n_1|\sqrt{2\varepsilon_N} \right\} \right.  \nonumber \\ 
    & \mathrel{\phantom{=====}} \left. \bigcup \left\{ \omega' \in \Omega \bigg| \left| \langle k_1, \omega' \rangle - \varpi_{n_1} - \mu(n_1, k_1) \right| < |n_1|\sqrt{2\varepsilon_N} \right\} \right)    \nonumber  \\
    & \quad \bigcap \left( \left\{ \omega' \in \Omega \bigg| \left| \langle k_2, \omega' \rangle + \varpi_{n_2} + \mu(n_2, k_2) \right| < |n_2|\sqrt{2\varepsilon_N} \right\} \right. \nonumber \\
    & \mathrel{\phantom{=====}} \left. \bigcup \left\{ \omega' \in \Omega \bigg| \left| \langle k_2, \omega' \rangle - \varpi_{n_2} - \mu(n_2, k_2) \right| < |n_2|\sqrt{2\varepsilon_N} \right\} \right)    \nonumber  \\
    & \subseteq \left\{ \omega' \in \Omega \bigg| \left| \langle k_1 - k_2, \omega' \rangle - \varpi_{n_1} + \varpi_{n_2} - \mu(n_1, k_1) + \mu(n_2, k_2) \right| < (|n_1|+|n_2|)\sqrt{2\varepsilon_N} \right\}    \nonumber  \\
    & \quad \bigcup \left\{ \omega' \in \Omega \bigg| \left| \langle k_1 - k_2, \omega' \rangle - \varpi_{n_1} - \varpi_{n_2} - \mu(n_1, k_1) - \mu(n_2, k_2) \right| < (|n_1|+|n_2|)\sqrt{2\varepsilon_N}  \right\}.       \label{eqn: inclusion-multiscale}
\end{align}
Set $k = k_1 - k_2$ and impose the separation condition 
$k \notin \Lambda_{2N'}^m$, so that $k \in \Lambda_{2N}^m \setminus \Lambda_{2N'}^m$,
where $N'$ is the reduced scale associated with $N$, to be defined
in~\eqref{eqn: 10-log-size}.

We first consider the case in which there exist two spatiotemporal indices 
$(n_1, k_1), (n_2, k_2) \in \mathbb{Z} \times (k_0 + \Lambda_N^m)$ with 
$|n_1| \neq |n_2|$, such that both eigenvalues $T_{n_0, k_0, N}(n_1, k_1)$ and 
$T_{n_0, k_0, N}(n_2, k_2)$ are ``bad'', i.e., their absolute values are less than
$2\varepsilon_{N'}$.
To control the corresponding small divisor, we exclude 
the resonant set along the line parallel to $k / \| k \|_2$. 
By~\eqref{eqn: mu-nu-nlaw}, the absolute value of the directional 
derivative of the difference is bounded below by $\|k\|_2 - 2\sqrt{\varepsilon}$. 
Consequently, along any such line segment within the parameter domain, the 
one-dimensional measure of the resonant set is bounded above by 
$2 (\| k\|_2 - 2\sqrt{\varepsilon})^{-1}$ times the resonance threshold.

We introduce a reduced scale $N'$, derived from the primary scale $N$ via 
the relation:
\begin{equation}
    \label{eqn: 10-log-size}
    N' = e^{ (\log N)^{\frac{1}{10}} }. 
\end{equation}
The difference nearly-resonant set for the large-scale regime is then defined 
as the union of its components across all scales $N>M$. 
\begin{defn}
    \label[definition]{defn: second-melnikov}
    The difference nearly-resonant set in the large-scale regime is defined as 
    \[
        G_2 = \bigcup_{N > M} G_{2, N} = \bigcup_{N > M} \left( G_{2, N}^{+} \cup G_{2, N}^{-} \right),
    \] 
    where the component sets $G_{2, N}^{+}$, $G_{2, N}^{-}$ are given by
    \begin{multline}
        G_{2, N}^{+} = \\ \bigcup_{k \in \Lambda^m_{2N} \setminus \Lambda^{m}_{2N'}} \bigcup_{n_1, n_2 \in \mathbb{Z}} \left\{ \omega' \in \Omega \bigg| \left| -\langle k, \omega' \rangle + \varpi_{n_1} + \varpi_{n_2} + \mu(n_1, k_1) + \mu(n_2, k_2) \right| < (|n_1|+|n_2|)\sqrt{2\varepsilon_{N'}} \right\},    \label{eqn: difference-nearly-resonant-multiscale_+}
    \end{multline} 
    and
    \begin{multline}
        G_{2, N}^{-} = \\ \bigcup_{k \in \Lambda^m_{2N} \setminus \Lambda^{m}_{2N'}} \bigcup_{\substack{n_1, n_2 \in \mathbb{Z}  \\  |n_1| \neq |n_2|}} \left\{ \omega' \in \Omega \bigg| \left| -\langle k, \omega' \rangle + \varpi_{n_1} - \varpi_{n_2} + \mu(n_1, k_1) - \mu(n_2, k_2) \right| < (|n_1|+|n_2|)\sqrt{2\varepsilon_{N'}} \right\},    \label{eqn: difference-nearly-resonant-multiscale_-}
    \end{multline}
\end{defn}

The remainder of the argument parallels the small-scale analysis 
in~\Cref{subsec: small-outer}. 
As detailed in~\Cref{sec: near-measure-initial}, it suffices to count the 
admissible values of the pairs $(n_1, n_2)$ themselves to estimate the measure
of $G_{2, N}$. 
Together with the scale relation~\eqref{eqn: 10-log-size}, this yields
\begin{align}
    \label{eqn: bound-measure-g2}
    \mathrm{mes}(G_{2}) &\leq \sum_{N>M} \mathrm{mes}(G_{2, N})  \nonumber  \\
                        &\leq \sum_{N>M} [\text{diam}(\Omega)]^{m-1} 16R(8R\sqrt{m}+1)^2(2N+1)^m e^{- \frac{(\log N)^{3/2}}{2}} < \infty,
\end{align}
where $R = \sup_{\omega' \in \Omega} \|\omega'\|_2$. 
In particular, since the series on the right converges and 
$M \rightarrow \infty$ as $\varepsilon \rightarrow 0$, its tail tends to zero. 
Consequently, $\mathrm{mes}(G_2) \rightarrow 0$ as $\varepsilon \rightarrow 0$.

\paragraph{``Bad'' temporal line}  Because the perturbation is unbounded, we must handle one more case:  two eigenvalues of $T$ may share the same spatial index $n$. Only one spatial index is involved here, we denote it by $n_{\star}$, which plays a role similar to the one in~\Cref{subsec: measure-estimate}. For any temporal index paired with $n_{\star}$, the corresponding eigenvalue may be ``bad.''  We call this set of indices the ``bad'' temporal line (see~\Cref{fig: bad-temporal-line}). 
\begin{figure}[htb!]
    \centering
    \begin{tikzpicture}[scale=0.95, >=stealth]

        \draw[line width=1.0pt] (-4.0, -2.0) rectangle (4.0, 2.0);

        \draw[line width=1.6pt] (-4.0, 0.8) -- (4.0, 0.8);

        \draw[decorate, decoration={brace, amplitude=5pt}, line width=0.8pt]
            (-4.0, -2.0) -- (-4.0, 2.0);
        \node[left] at (-4.15, 0.0) {\small $n_0+\Lambda_N$};

        \draw[decorate, decoration={brace, mirror, amplitude=5pt}, line width=0.8pt]
            (-4.0, -2.0) -- (4.0, -2.0);
        \node at (0.0, -2.5) {\small $k_0+\Lambda_N^m$};

        \node[right] at (4.1, 0.8) {\small $n_\star$};

        \draw[->, line width=0.8pt] (5.0, -1.2) -- (1.5, 0.7);
        \node[right] at (5.0, -1.2) {\small ``bad'' temporal line};

    \end{tikzpicture}
    \caption{The ``bad'' temporal line in the outer lattice box 
             $\Gamma_{n_0,k_0,N}$. 
            }
    \label{fig: bad-temporal-line}
\end{figure}
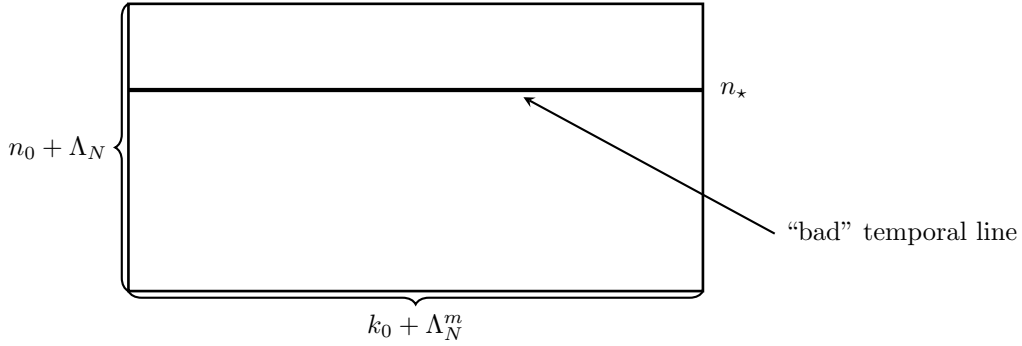
Thus, we need to exclude the ``bad'' frequencies for which every indices within the “bad” temporal line gives an eigenvalue whose absolute value is less than $2\varepsilon_{N'}$.  We count all such ``bad'' indices in $\Gamma_{n_0, k_0, N}$ and in any outer lattice box needed.   The additional ``bad'' frequency set is defined as follows.

\begin{defn}
    \label[definition]{defn: additional-bad-freq-set}
    The additional ``bad'' frequency set in the large-scale regime is defined as 
    \[
        G_\mathrm{add} = \bigcup_{N > M} G_{\mathrm{add}, N} = \bigcup_{N > M} \left( G_{\mathrm{add}, N}^{+} \cup G_{\mathrm{add}, N}^{-} \right),
    \] 
    where the component sets $G_{\mathrm{add}, N}^{+}$, $G_{\mathrm{add}, N}^{-}$ are given by
    \begin{equation}
        G_{\mathrm{add}, N}^{+} = \bigcup_{(n_0, k_0) \in \Lambda_{N^{10}}^{1+m}} \bigcup_{k \in k_0 + \Lambda^m_{N}} \left\{ \omega' \in \Omega \bigg| \left| \langle k, \omega' \rangle + \varpi_{n_\star} + \mu(n_\star, k) \right| < |n_\star|\sqrt{2\varepsilon_{N'}} \right\},    \label{eqn: additional-bad-freq-set_+}
    \end{equation} 
    and
    \begin{equation}
        G_{\mathrm{add}, N}^{-} = \bigcup_{(n_0, k_0) \in \Lambda_{N^{10}}^{1+m}} \bigcup_{k \in k_0 + \Lambda^m_{N}} \left\{ \omega' \in \Omega \bigg| \left| -\langle k, \omega' \rangle + \varpi_{n_\star} + \mu(n_\star, k) \right| < |n_\star|\sqrt{2\varepsilon_{N'}} \right\},    \label{eqn: additional-bad-freq-set_-}
    \end{equation}
\end{defn}

Next, we estimate the measure of the additional ``bad” frequency set $G_{\mathrm{add}}$. The resonance condtion in~\eqref{eqn: additional-bad-freq-set_+} and~\eqref{eqn: additional-bad-freq-set_-} imply $R\|k\|_2 \geq |n_\star|$.
As in the exclusion of $G_2$, we work along lines parallel to $k / \|k\|_2$. 
By~\eqref{eqn: mu-nu-nlaw}, the directional 
derivative  inside the absolute value is bounded below by $\|k\|_2 - \sqrt{\varepsilon}$. 
So on any such line segment inside the parameter domain, the one-dimensional measure of the resonant set is at most
$(\| k\|_2 - \sqrt{\varepsilon})^{-1}$ times the resonance threshold 
$|n_\star| \sqrt{2\varepsilon_{N'}}$.
Therefore, we obtain the estimate
\begin{align}
    \label{eqn: bound-measure-g-add}
    \mathrm{mes}(G_{\mathrm{add}}) 
    \leq \sum_{N>M} \mathrm{mes}(G_{\mathrm{add}, N}) 
    \leq \sum_{N>M} [\text{diam}(\Omega)]^{m-1} 4R(2N^{10}+1)^{2m+1} e^{- \frac{(\log N)^{3/2}}{2}} < \infty. 
\end{align}
In particular, since the series on the right converges and 
$M \rightarrow \infty$ as $\varepsilon \rightarrow 0$, its tail tends to zero. 
Consequently, $\mathrm{mes}(G_\mathrm{add}) \rightarrow 0$ as $\varepsilon \rightarrow 0$. 
Once the additional ``bad” frequency set $G_{\mathrm{add}}$ associated with the ``bad'' temporal line has been excluded, it suffices to work in the outer lattice box $\Gamma_{n_0, k_0, N}$. The remaining ``gluing'' step, analogous to that in~\Cref{subsec: verification}, is trivial here and is therefore omitted.



With the spacial case of ``bad'' temporal line settled, it remains to treat the general case $|n_1| \neq |n_2|$. 
We consider two subboxes $\Gamma_{n_{0,1}, k_{0,1}, N'}$ and 
$\Gamma_{n_{0,2}, k_{0,2}, N'}$ with centers $(n_{0,1}, k_{0,1})$, 
$(n_{0,2}, k_{0,2}) \in \Gamma_{n_0, k_0, N}$. 
If their centers satisfy the separation condition 
$|k_{0,1} - k_{0,2}|_{\infty} > 4N'$, then at least one of the restricted 
operators $T_{n_{0,1}, k_{0,1}, N'}$ and $T_{n_{0,2}, k_{0,2}, N'}$ has all its 
eigenvalues bounded away from zero by $2\varepsilon_{N'}$. 
Equivalently, at most one of the resolvent norms, 
$\|T_{n_{0,1}, k_{0,1}, N'}^{-1}\|_2$ and 
$\|T_{n_{0,2}, k_{0,2}, N'}^{-1}\|_2$, can exceed $\varepsilon_{N'}^{-1}/2$. 
Thus, the centers of the subboxes on which the restricted operator is 
nearly singular must be close in the temporal direction.  
To describe the resulting clustering in the temporal indices, we define the set 
of singular centers within $\Gamma_{n_0, k_0, N}$ as
\begin{equation}
    \label{eqn: eqn-violated-center-pde}
    \Sigma(N, N') = \left\{ (n_{0}', k_{0}') \in \Gamma_{n_0, k_0, N} ~\bigg|~ \| T_{n_{0}', k_{0}', N'}^{-1} \|_2 > \frac{1}{2\varepsilon_{N'}} \right\}.
\end{equation}
The following lemma makes this clustering property precise.
\begin{lemma}
    \label[lemma]{lem: second-melnikov}
    Suppose that $\hat{u} \in \mathcal{K}(s)$ and 
    $\omega' \in \Omega \setminus (\Omega_M^\tau \cup G_2)$. 
    If $(n_{0,1}, k_{0, 1}), (n_{0,2}, k_{0, 2}) \in \Sigma(N, N')$, then 
    they must be clustered such that:
    \begin{equation}
        \label{eqn: center-clustering}
        k_{0, 1} - k_{0, 2} \in \Lambda_{4N'}^{m}.
    \end{equation}
\end{lemma}
Note that~\Cref{lem: second-melnikov} concerns only the temporal indices: for any $(n_0,k_0) \in  \Sigma(N, N')$, all temporal indices of singular centers lie in $k_{0}' + \Lambda_{4N'}^{m}$; it imposes no additional restriction on 
the spatial index $n_0'$, which may range over all of $n_0 + \Lambda_N$.

\subsection{Multi-scale induction: inversion implies localization}
\label{sec: inversion-localization}

Based on~\Cref{lem: second-melnikov}, the singular centers are contained in a 
union of subboxes $\Gamma_{n_0', k_0',4N'}$ that share a common temporal index 
$k_0' \in k_0 + \Lambda_N^{m}$, while the spatial centers $n_0'$ ranges over $n_0 + \Lambda_N$. 
More precisely, there exists $k_0' \in k_0 + \Lambda_N^{m}$ such that
\[
    \Sigma(N, N') \subseteq \bigcup_{n_0' \in n_0 + \Lambda_N} \Gamma_{n_0', k_0',4N'} \subseteq \bigcup_{n_0' \in n_0 + \Lambda_N} \Gamma_{n_0', k_0',10N'},
\]  
where the second inclusion is immediate from 
$\Gamma_{n_0', k_0',4N'} \subseteq \Gamma_{n_0', k_0',10N'}$. 
We refer to the enlarged boxes $\Gamma_{n_0', k_0', 10N'}$ as singular boxes, 
and we call a subbox regular if its center does not belong to $\Sigma(N, N')$. 
With this classification, the box $\Gamma_{n_0, k_0, N} $ is covered by the 
singular boxes, all of which share the common temporal index $k_0'$ while the 
spatial centers $n_0'$ ranges over $n_0 + \Lambda_N$, together with a collection 
of regular boxes:
\begin{equation}
    \label{eqn: multi-scale-decomposition}
    \Gamma_{n_0, k_0, N} = \left( \bigcup_{n_0' \in n_0 + \Lambda_N} \Gamma_{n_0', k_0',10N'} \right) \bigcup \left( \bigcup_{k_0'' \notin k_0' + \Lambda_{9N'}^{m}} \bigcup_{n_0'' \in n_0 + \Lambda_N^m} \Gamma_{n_0'', k_0'',N'} \right).
\end{equation}
Every box in the second union is indeed regular. Note that a box $\Gamma_{n_0'', k_0'', N'}$ with 
$k_0'' \in (k_0 + \Lambda_{N}^{m}) \setminus ( k_0' + \Lambda^{m}_{9N'})$ may 
overlap some singular boxes $\Gamma_{n_0', k_0', 10N'}$; it is 
nevertheless regular. Indeed, since $|k''_0 - k_0'| \geq 9N'$, its temporal range lies 
at distance no less than $8N'$ from $k_0'$, so that it is disjoint from any smaller 
box $\Gamma_{n_0', k_0', 4N'}$. Since it is precisely these smaller boxes that contain 
$\Sigma(N, N') $, its center does not belong to $\Sigma(N, N') $. 
This deliberate enlargement facilitates the subsequent gluing procedure, 
in accordance with the strategy described in~\Cref{subsec: verification}.

We repeat the scale reduction process for each box in the 
decomposition~\eqref{eqn: multi-scale-decomposition}. 
This iteration proceeds through a sequence of scales satisfying the hierarchy~\eqref{eqn: 10-log-size}, until the scales enter the small-scale regime $[M_0, M]$. 
For each terminal singular box in the small-scale regime,
\Cref{prop: small-eig} provides a singular set of the form~\eqref{eqn: singular-sites-normal},
so that the arguments of~\Cref{subsec: small-outer} apply.
Within the small-scale regime, we have already established the following two results.
\begin{itemize}
    \item[(i)] For a singular box, once the singular set $\Pi$ has been
    removed,~\Cref{lem: small-singular} guarantees that the inversion 
    condition~\eqref{eqn: small-size-inverse-prepare} and the localization 
    condition~\eqref{eqn: small-size-off-diagonal-prepare} hold 
    simultaneously. 
    
    \item[(ii)] For a regular box,~\Cref{lem: small-inverse-localization} 
    establishes that the inversion 
    condition~\eqref{eqn: trans-small-size-original-inverse} directly implies 
    the localization condition~\eqref{eqn: trans-small-size-off-diagonal}. 
\end{itemize}

Taking these small-scale results as the base of the induction, we now show 
that inversion implies localization throughout the large-scale regime $N > M$.  
The precise statement is the following lemma.

\begin{lemma}
    \label[Lemma]{lem: multi-scale-singular}
    Suppose that $\hat{u} \in \mathcal{K}(s)$ and 
    $\omega' \in \Omega \setminus (\Omega_{M}^{\tau} \cup G_{2})$. 
    If the restricted operator $T_{n_0, k_0, N}$ satisfies the inversion 
    condition~\eqref{eqn: trans-small-size-original-inverse} for any scale 
    $M^r < N \leq M^{r+1}$, then the entries of its inverse satisfy the 
    localization condition
    \begin{equation}
        \label{eqn: trans-large-size-off-diagonal}
        \left| T_{n_0, k_0, N}^{-1}\left( (n, k), (n', k') \right) \right| \leq e^{ - \frac{ (|n-n'| + |k-k'|)^s}{2} },
    \end{equation}
    for any $|n-n'| + |k-k'| \geq N^{\frac12}$. 
\end{lemma}

The proof consists of two successive steps:  a ``gluing'' procedure, followed 
by a demonstration that inversion implies localization.

\paragraph{The ``gluing'' procedure} 
By the inductive hypothesis for the scales $M_0 \leq N \leq M^{r}$, the following two
statements hold:
\begin{itemize}
    \item For each singular box $\Gamma_{n_0', k_0' ,10N'}$ with 
    $M_0 \leq 10N' \leq M^r$, the operator restricted to 
    $\Gamma_{n_0', k_0' ,10N'} \setminus \Pi$, that is, to the box with the 
    singular sites removed, satisfies the inversion 
    condition~\eqref{eqn: small-size-inverse-prepare} and the localization condition~\eqref{eqn: small-size-off-diagonal-prepare}. 
    
    \item For each regular box $\Gamma_{n_0'', k_0'' ,N'}$ with 
    $k_0'' \in (k_0 + \Lambda_{N}^{m}) \setminus ( k_0' + \Lambda^{m}_{9N'})$ 
    and $M_0 \leq N' \leq M^r$, the inversion 
    condition~\eqref{eqn: trans-small-size-original-inverse} for the restricted 
    operator implies the localization 
    condition~\eqref{eqn: trans-small-size-off-diagonal}. 
\end{itemize}
Now let $M^r < N \leq M^{r+1}$. 
By the decomposition~\eqref{eqn: multi-scale-decomposition}, the set 
$\Gamma_{n_0, k_0, N} \setminus \Pi$ is covered by the two types of boxes 
described above: singular boxes with the singular sites removed, and regular 
boxes. 
Gluing the corresponding local inverses into a global inverse, exactly as in 
the procedure of~\Cref{subsec: verification}, we conclude 
$T_{\Gamma_{n_0, k_0, N} \setminus \Pi}$ satisfies 
the inversion condition~\eqref{eqn: trans-small-size-original-inverse} and the 
localization condition~\eqref{eqn: trans-small-size-off-diagonal} for $M^{r} < N \leq M^{r+1}$. 
 Since the argument is identical to that of~\Cref{subsec: verification}, we omit the routine details.

\paragraph{Inversion implies localization} This step proceeds exactly as in the 
argument for the small-scale regime $M_0 \leq N \leq M$ 
treated in~\Cref{subsec: small-outer} 
(see the paragraph ``Inversion implies localization'' there). 
For the configuration~\eqref{eqn: singular-sites-normal} of 
$|n_{\star}| > N^{\frac13}/2$, the estimates obtained there concern boxes of 
size $K = N^{\frac{1}{10}}$; repeating the 
procedure~\eqref{eqn: resolvent-small-offside}--\eqref{eqn: entries-phi-small-offside}, 
we complete the proof.

\subsection{Measure estimate for ``bad'' frequencies}
\label{subsec: measure-estimate}

The final stage of the multiscale analysis consists in quantifying the set of 
``bad''  frequencies that must be excluded in order for the inversion 
condition~\eqref{eqn: trans-small-size-original-inverse} to remain valid. 
In this section, we establish the rigorous measure estimates for these 
resonant sets. 
Recall the restricted operator $T_{n_0, k_0, N}$ admits the block 
decomposition 
\[
    T_{n_0, k_0, N} = 
    \begin{pmatrix} 
        T_{\Gamma_{n_0, k_0, N} \setminus \Pi} & \varepsilon P^*  \\ 
        \varepsilon P & T_{\Pi} 
    \end{pmatrix}.
\]
In terms of this block structure, the Schur complement of the block 
$T_{\Gamma_{n_0, k_0, N} \setminus \Pi} $ is
\begin{equation}
    \label{eqn: schur-complement}
    U = T_{\Pi} - \varepsilon^2 P T_{\Gamma_{n_0, k_0, N} \setminus \Pi}^{-1} P^*.
\end{equation}
From~\eqref{eqn: schur-complement}, we obtain the following upper bound for the 
$\ell_2$-norm of the inverse operator:
\[
    \| T_{n_0, k_0, N}^{-1} \|_2 \leq 4 \| T_{\Gamma_{n_0, k_0, N} \setminus \Pi}^{-1} \|_2^2 \cdot \| U^{-1} \|_2 + \| T_{\Gamma_{n_0, k_0, N} \setminus \Pi}^{-1} \|_2. 
\]
Since the bound $\| T_{\Gamma_{n_0, k_0, N} \setminus \Pi}^{-1} \|_2 \leq (20N'+1)^{1+m} e^{( \log 10N')^{15}}$ 
has already been established in~\Cref{sec: inversion-localization}, we deduce the 
inclusion of resonant sets
\begin{equation}
    \label{eqn: inverse-relation-schur-complement}
    \left\{ \omega' \in \Omega ~\bigg|~ \| T_{n_0, k_0, N}^{-1} \|_2 > \frac{1}{2\varepsilon_N} \right\} \subseteq \left\{ \omega' \in \Omega ~\bigg|~ \| U^{-1} \|_2 > \frac{1}{\sqrt{\varepsilon_N}} \right\}.
\end{equation}
Thus, for the measure estimates, it suffices to analyze the Schur complement 
$U$. 
We now express~\eqref{eqn: schur-complement} explicitly in each of the two 
configurations. 
In the tangential configuration~\eqref{eqn: singular-sites-tangent},
\begin{equation}
    \label{eqn: u-tangential}
    U = \mathrm{diag}\left( - \langle k_{\star}, \omega' \rangle^2 \pm\omega_j' \langle k_{\star}, \omega' \rangle -{\omega'_j}^2 + \omega_j^2 \right) + \varepsilon S_{\Pi} - \varepsilon^2 P T_{\Gamma_{n_0, k_0, N} \setminus \Pi}^{-1} P^*. 
\end{equation}
In the normal configuration~\eqref{eqn: singular-sites-normal},
\begin{equation}
    \label{eqn: u-normal}
    U = \mathrm{diag}\left( - \langle k_{\star}, \omega' \rangle^2 + n_{\star}^4 + n_{\star}^2 \right) + \varepsilon S_{\Pi} - \varepsilon^2 P T_{\Gamma_{n_0, k_0, N} \setminus \Pi}^{-1} P^*. 
\end{equation} 
Accordingly, we set
$\sigma_{\star}=-\langle k_{\star},\omega'\rangle^2$ in the tangential 
configuration, whereas $\sigma_{\star}=-\langle k_{\star},\omega'\rangle^2 +n_{\star}^4+n_{\star}^2$ 
in the normal configuration.
Then, in both cases, we can write
\begin{equation}
    \label{eqn: u-bar-form}
    U = \sigma_{\star} \mathrm{Id} + \Upsilon.
\end{equation}
where $\Upsilon = \mathrm{diag}( \pm\omega_j' \langle k_{\star}, \omega' \rangle -{\omega'_j}^2 + \omega_j^2 ) + \varepsilon S_{\Pi} - \varepsilon^2 P T_{\Gamma_{n_0, k_0, N} \setminus \Pi}^{-1} P^*$ 
in the tangential configuration, and 
$\varepsilon S_{\Pi} - \varepsilon^2 P T_{\Gamma_{n_0, k_0, N} \setminus \Pi}^{-1} P^*$ 
in the normal configuration. Since $\| T_{\Gamma_{n_0, k_0, N} \setminus \Pi}^{-1} \|_2 \leq (20N'+1)^{1+m} e^{( \log 10N')^{15}}$, 
both~\eqref{eqn: u-tangential} and~\eqref{eqn: u-normal} implies that 
$\| \Upsilon \|_2 \leq (20N'+1)^{2+m} e^{( \log 10N')^{15}}$. 
Combining~\eqref{eqn: inverse-relation-schur-complement} 
and~\eqref{eqn: u-bar-form}, we define the single-mode nearly-resonant set at 
the scale $N$ as:
\begin{equation}
    \label{eqn: single-mode-nearly-resonant-multiscale}
    G_{1,N} = \bigcup_{(n, k) \in \Lambda_{N^{10}}^{1+m}} \left\{ \omega' \in \Omega ~\bigg|~ \big| |\sigma_{\star}| - \|\Upsilon\|_2 \big| < \sqrt{\varepsilon_N} \right\}. 
\end{equation}
Indeed, for $\omega' \notin G_{1,N}$, we have $\|U^{-1}\|_2 \leq1/\sqrt{\varepsilon_N} $ and hence, by~\eqref{eqn: inverse-relation-schur-complement}, the inversion condition~\eqref{eqn: trans-small-size-original-inverse} holds for all frequencies 
$\omega' \notin G_{1,N}$. 
To estimate the measure of this set, we utilize a simplified version of the 
Cartan estimate~\citep{levin1996lectures}.

\begin{lemma}[Theorem 4 of Lecture 11.3 in~\citet{levin1996lectures}]
    \label[lemma]{lem: cartan}
    Let $f(z)$ be an analytic function in the complex disk 
    $\{ z \in \mathbb{C} \mid |z| \leq 2eR \}$ with $|f(0)| = 1$. 
    Let $M_f(r) = \max_{|z|=r} |f(z)|$, and let $\eta > 0$ be an arbitrarily 
    small positive number. 
    Then the estimate
    \[
        \log |f(z)| > - H(\eta) \log M_f(2eR), \qquad H(\eta) = \log \frac{15e^3}{\eta},
    \]
    holds everywhere in the disk $\{z \in \mathbb{C} \mid |z| \leq R\}$ except 
    for a family of disks $(C_j)$ with sum of radii $\sum r_j \leq \eta R$.
\end{lemma}

To complete the measure estimates, we distinguish two regimes according to the size of the residual term $\Upsilon$. 
For small residuals, namely $ \| \Upsilon \|_2 < \varepsilon_N^{\frac{1}{4}}$, 
the single-mode nearly-resonant set satisfies 
\begin{equation}
    \label{eqn: single-msa-small-residual}
    G_{1,N} \subseteq \bigcup_{(n, k) \in \Lambda_{N^{10}}^{1+m}} \left\{ \omega' \in \Omega ~\bigg|~ |\sigma_{\star}| < 2\varepsilon_N^{\frac{1}{4}} \right\}. 
\end{equation}
This case also includes the small-scale regime, where the sharper threshold 
$|\sigma_\star| < \varepsilon_N$suffices. 
The measure of the set on the right-hand side 
of~\eqref{eqn: single-msa-small-residual} is bounded by
\begin{equation}
    \label{eqn: mes-g-1-N-1}
    \mathrm{mes}(G_{1,N}) \leq 4\left[\text{diam}(\Omega)\right]^{m-1} (2N^{10}+1)^{1+m} \varepsilon_N^{\frac{1}{4}}. 
\end{equation}
For large residuals, namely $\| \Upsilon \|_2 \geq \varepsilon_N^\frac{1}{4}$, 
we introduce an analytic function:
\[
    f(\sigma_{\star}) = \frac{|\sigma_{\star}|}{\| \Upsilon \|_2 } - 1,
\]
and choose the disk radius $R$ so that $\| \Upsilon \|_2 = 2e(1+e)^{-1}R$. 
With this choice, $|f(0)| = 1$ and $M_f(2eR) = e$.  
The single-mode nearly-resonant set then satisfies:
\begin{equation}
    \label{eqn: single-msa-large-residual}
    G_{1,N} \subseteq \bigcup_{(n, k) \in \Lambda_{N^{10}}^{1+m}} \left\{ \omega' \in \Omega \bigg | |f(\sigma_{\star})| < \varepsilon_N^{1/4} \right\}.
\end{equation}
Applying~\Cref{lem: cartan} with $\eta = 15e^3 \varepsilon_N^{1/4}$, we obtain 
the measure estimate
\begin{align}
    \mathrm{mes}(G_{1,N}) 
    & \leq 8(1+e)e^2 \left[\text{diam}(\Omega)\right]^{m-1} (2N^{10}+1)^{1+m} (20N'+1)^{1+m} e^{(\log 10N')^{15}} \varepsilon_{N}^{\frac{1}{4}}  \nonumber  \\
    & \leq 8(1+e)e^2 \left[\text{diam}(\Omega)\right]^{m-1} (2N^{10}+1)^{2+2m} \varepsilon_{N}^{\frac{1}{5}}.   \label{eqn: mes-g-1-N-2}
\end{align}
In view of~\eqref{eqn: single-mode-nearly-resonant-multiscale}, the total 
single-mode resonant set is the union over all scales $N \geq M_0$ as 
\begin{equation}
    \label{eqn: final-msa-single-mode-nearly-resonant}
    G_1 = \bigcup_{N \geq M_{0}} G_{1,N},
\end{equation}
and combining~\eqref{eqn: mes-g-1-N-1} and~\eqref{eqn: mes-g-1-N-2}, we obtain
\[
    \mathrm{mes}(G_1) \leq \sum_{N \geq M_0} 8(1+e)e^2 \left[\text{diam}(\Omega)\right]^{m-1} (2N^{10}+1)^{2+2m} \varepsilon_{N}^{\frac{1}{5}} < \infty.
\]
Provided that the summands admit a summable bound independent of $\varepsilon$, 
the fact that $M_0 \rightarrow \infty$ implies $\varepsilon \rightarrow 0$ 
implies that $\mathrm{mes}(G_1) \rightarrow 0$. 
Finally, setting $G = G_1 \cup G_2 \cup G_{\mathrm{add}}$, as in~\Cref{thm: off-side-box}, and 
combining this bound with the measure estimate for $G_2$ in~\eqref{eqn: bound-measure-g2}, 
we conclude that $\text{mes}(G) \leq \delta$ with 
$\delta \to 0$ as $\varepsilon \to 0$, as asserted in~\Cref{thm: main}.

\section{Convergence of the numerical scheme}
\label{sec: iterative-scheme}

In this section, we establish  the convergence of 
the alternating numerical 
scheme~\eqref{eqn: alternating-numerical-nlaw}.~\Cref{subsec: induction} is 
devoted to the iteration lemma, which constitutes the technical core of the 
analysis. 
In~\Cref{subsec: proof-main}, we then apply this lemma to complete 
the proof of the main statement.

\subsection{The iteration lemma}
\label{subsec: induction}

Before proceeding to the inductive estimates, we need to characterize the 
derivative of the inverse operator $\partial L_N^{-1}$, specifically its 
operator norm and the off-diagonal decay of its entries. 
Throughout this section, the derivative is taken with respect to the drifted 
frequency $\omega'$; to lighten the notation, this dependence is not indicated 
explicitly. 
Differentiating the identity $L_N^{-1} L_N = I$ yields:
\begin{equation}
    \label{eqn: t-n-derivative-inverse}
    \partial L_{N}^{-1} = - L_{N}^{-1}(\partial L_{N} ) L_{N}^{-1}.
\end{equation}
From this identity, we obtain the following norm bound and localization estimate.

\begin{lemma}
    \label{lem: inverse-derivative}
    Suppose that the restricted operator $L_{N}$ 
    satisfies~\Cref{cond: implementation}. 
    Then the derivative of its inverse satisfies the norm bound:
    \begin{equation}
        \label{eqn: inverse-restricted-norm}
        \| \partial L_{N}^{-1} \|_2 \leq \frac{2RmN^2}{\varepsilon_N^2}, 
    \end{equation}
    where $R = \sup_{\omega' \in \Omega} \|\omega'\|_2$. 
    Furthermore, for lattice sites in $\Lambda_N^{1+m}$ with 
    $|n-n'| + |k-k'| \geq N^{\frac34}$, the entries of the 
    derivative of the inverse exhibit Gevrey-type decay:
    \begin{equation}
        \label{eqn: inverse-restricted-entries}
        \left| \partial L_{N}^{-1}\left( (n, k), (n', k') \right) \right| \leq e^{ - \frac{ \left( |n-n'| + | k - k' |\right)^{s}}{4}} .
    \end{equation}
\end{lemma}

The norm bound~\eqref{eqn: inverse-restricted-norm} is an immediate consequence 
of the inversion condition~\eqref{eqn: inverse-grow} and the 
identity~\eqref{eqn: t-n-derivative-inverse}. 
The localization estimate~\eqref{eqn: inverse-restricted-entries} relies on 
elementary inequalities associated with Gevrey regularity; its proof is deferred 
to~\Cref{subsec: derivative-app}. 
Given that~\Cref{cond: implementation} has been verified in the preceding 
sections,~\Cref{lem: inverse-derivative} implies the iteration lemma.

\begin{theorem}[The Iteration Lemma]
    \label{lem: iteration}
    Let $\omega' \in \Omega \setminus (\Omega_{M}^{\tau} \cup G)$, and let 
    $\varepsilon_0 > 0$ be the threshold provided in~\Cref{thm: main}. 
    For any $0 < \varepsilon \leq \varepsilon_0$, there exists a Gevrey exponent 
    $s(\varepsilon) > 0$ with the following property. 
    Suppose that at the $r$-th iteration, the approximate solution 
    $\hat{u}_p^{(r)}$ is supported in: 
    \begin{equation}
        \label{eqn: vector-support-r}
        \mathrm{supp} \; \hat{u}_p^{(r)} \subseteq  \Lambda^{1+m}_{N_r}, 
    \end{equation}
    that for any index $i \in \{ 0, 1, \ldots, r\}$ and any 
    $(n,k) \in \Lambda^{1+m}_{N_i} \setminus \Lambda^{1+m}_{N_{i-1}}$, 
    it satisfies the Gevrey decay estimate 
    \begin{equation}
        \label{eqn: vector-decay-r}   
        \| \hat{u}^{(r)}(n, k) \| \leq \sum_{\ell=i}^{r}e^{ - \frac32 N_{\ell}^{s} },
    \end{equation}
    and the non-resonant vector field $F$ satisfies 
    \begin{equation}
        \label{eqn: vector-field-decay-r}  
        \big\| F\big( \hat{u}^{(r)}_p, \omega^{(r+1)}; a \big) \big\| \leq e^{-2N_{r}^{s} }.                                                                                                                                                                                                                                                                   
    \end{equation}
    Then the iterate $\hat{u}_p^{(r+1)}$ produced by the alternating numerical 
    scheme~\eqref{eqn: alternating-numerical-nlaw} satisfies: 
    \begin{equation}
        \label{eqn: vector-support-r+1}
        \mathrm{supp} \; \hat{u}_p^{(r+1)} \subseteq \Lambda^{1+m}_{N_{r+1}}.
    \end{equation}
    For any $i \in \{ 0, 1, \ldots, r+1\}$ and any 
    $(n, k) \in \Lambda^{1+m}_{N_i} \setminus \Lambda^{1+m}_{N_{i-1}}$, the 
    following decay property holds: 
    \begin{equation}
        \label{eqn: vector-decay-r+1}   
        \| \hat{u}^{(r+1)}(n, k) \| \leq \sum_{\ell=i}^{r+1}e^{ - \frac32 N_{\ell}^s },
    \end{equation}
    and 
    \begin{equation}
        \label{eqn: vector-field-decay-r+1}  
        \big\| F \big( \hat{u}_p^{(r+1)},\omega^{(r+2)}; a \big) \big\| \leq e^{-2N_{r+1}^{s} }.                                                                                                                                
    \end{equation}
    In other words, properties~\eqref{eqn: vector-support-r}--\eqref{eqn: vector-field-decay-r} 
    are preserved under the alternating update~\eqref{eqn: alternating-numerical-nlaw}.
\end{theorem}

\begin{proof}[Proof of~\Cref{lem: iteration}]
    For conciseness, we suppress the auxiliary parameters from the 
    dimension-enlarged Newton scheme~\eqref{eqn: alternating-coefficient-nlaw} 
    and write the increment as:
    \begin{equation}
        \label{eqn: nash-moser-single-simplification}
        \Delta^{(r)} = \hat{u}^{(r+1)}_p - \hat{u}^{(r)}_p = - L_{N_{r+1}}^{-1} \big(  \hat{u}^{(r)}_p \big) F\big(  \hat{u}^{(r)}_p \big).
    \end{equation}
    The proof consists of three steps.

    \paragraph{Support enlargement} 
    The first step verifies the support inclusion~\eqref{eqn: vector-support-r+1}. 
    Since the perturbation $(u^2)_{xx}$ is quadratic, the support 
    assumption~\eqref{eqn: vector-support-r} implies
    \begin{equation}
        \label{eqn: support-d-enlarge}
        \mathrm{supp} \; F\big( \hat{u}^{(r)}_p \big) \subseteq [-2N_{r}, 2N_{r}]^{1+m} \subseteq \left[ -\frac{N_{r+1}}{4}, \frac{N_{r+1}}{4} \right]^{1+m},
    \end{equation}
    which, combined with the update rule~\eqref{eqn: alternating-numerical-nlaw}, 
    establishes~\eqref{eqn: vector-support-r+1}.

    \paragraph{Lattice vector decay} This step establishes the 
    bound~\eqref{eqn: vector-decay-r+1}. 
    By the inversion condition~\eqref{eqn: inverse-grow} and the residual 
    estimate~\eqref{eqn: vector-field-decay-r}, the increment 
    satisfies
    \[
        \| \Delta^{(r)} \|_2 \leq \big \| L_{N_{r+1}}^{-1} \big( \hat{u}^{(r)}_p \big) \big\|_2 \big \| F\big( \hat{u}^{(r)}_p\big) \big\|_2 \leq \frac{1}{\varepsilon_{N_{r+1}}} e^{-2N_{r}^{s}}.
    \]
    Furthermore,~\Cref{lem: inverse-derivative} provides the derivative estimate
    \begin{align*}
        \| \partial \Delta^{(r)} \|_2 
        & \leq \big\| \partial L_{N_{r+1}}^{-1} \big( \hat{u}^{(r)}_p \big) \big\|_2 \big\| F\big( \hat{u}^{(r)}_p\big) \big\|_2 + \big\| L_{N_{r+1}}^{-1} \big( \hat{u}^{(r)}_p \big) \big\| _2\big\| \partial F\big( \hat{u}^{(r)}_p\big) \big\|_2                                                            \nonumber \\
        & \leq \left( \frac{2RmN_{r+1}^2 + \varepsilon_{N_{r+1}}}{\varepsilon_{N_{r+1}}^2}\right) e^{-2N_{r}^{s}}.       
    \end{align*} 
    Combining the two bounds above and reducing $\varepsilon_0$ if necessary, we 
    obtain
    \begin{equation}
        \label{eqn: difference-derivative-estimate}
        \| \Delta^{(r)} \| \leq \exp\left\{-\frac{7N_{r}^{s}}{4}\right\},
    \end{equation}
    for any $0 < \varepsilon \leq \varepsilon_0$. 
    We now choose $s = s(\varepsilon)>0$ so that $6M^{s} \leq 7$. 
    Consequently, at the $(r+1)$-th iteration, for any index 
    $i \in \{0, 1, \ldots, r+1\}$ and any 
    $(n, k) \in \Lambda^{1+m}_{N_i} \setminus \Lambda^{1+m}_{N_{i-1}}$, the 
    following holds:
    \begin{align*}
        \| \hat{u}^{(r+1)}_p \| \leq \| \hat{u}^{(r)}_p \| + \|\Delta^{(r)}\| \leq \sum_{\ell=i}^{r+1} e^{ - \frac32 N_{\ell}^{s} },
    \end{align*}
    which confirms that the decay property~\eqref{eqn: vector-decay-r+1} is 
    preserved across successive iterations.

    \paragraph{Super-exponential convergence} To establish the super-convergence 
    property~\eqref{eqn: vector-field-decay-r+1}, we analyze the residual of the 
    vector field after the numerical update~\eqref{eqn: alternating-numerical-nlaw}. 
    Expanding the vector field and the frequency by Taylor's formula, and 
    invoking~\Cref{prop: tensor}, we bound the post-update residual by the 
    truncation error together with the quadratic terms of the iteration: 
    \begin{equation}
        \label{eqn: vector-estimate-taylor-r+1}
        \| F(\hat{u}_p^{(r+1)}, \omega^{(r+2)}) \|_2 \leq \| [ L - L_{N_{r+1}}] \Delta^{(r)} \|_2 + (2N_{r+1}+1)^\frac{1+m}{2} e^{- 3N_{r+1}^{s} }. 
    \end{equation}
    Similarly, the derivative of the residual satisfies:
    \begin{align}
        \| \partial F(\hat{u}_p^{(r+1)}, \omega^{(r+2)}) \|_2 
        \leq & \| \partial [ L - L_{N_{r+1}}] \Delta^{(r)} \|_2 + \| [ L - L_{N_{r+1}} ]\partial \Delta^{(r)} \|_2 \nonumber \\ 
        & + 3(2N_{r+1}+1)^\frac{1+m}{2}e^{- 3N_{r+1}^{s} }.     \label{eqn: vector-estimate-derivative-taylor-r+1}
    \end{align}
    To estimate the truncation terms in~\eqref{eqn: vector-estimate-taylor-r+1} 
    and~\eqref{eqn: vector-estimate-derivative-taylor-r+1}, we exploit the 
    off-diagonal decay of the operators $L $ and $L_{N_{r+1}}^{-1}$ established 
    in~\Cref{prop: convolution} and~\Cref{lem: inverse-derivative}, respectively. 
    Since the estimates for these terms are entirely analogous, we focus on 
    $[L- L_{N_{r+1}}] \Delta^{(r)}$ as a representative example. 
    We utilize the following decomposition:
    \begin{align}
        [L- L_{N_{r+1}}] \Delta^{(r)} 
        =  & \left[ (\mathrm{Id} - P_{N_{r+1}}) LP_{N_{r+1}} \right] \left[ (\mathrm{Id} - P_{N_{r+1}/3}) \Delta^{(r)} \right]  \nonumber \\
           & + \left[ (\mathrm{Id} - P_{N_{r+1}}) L P_{N_{r+1}/3} \right] \Delta^{(r)}.  \label{eqn: decomposition-t-tr}
    \end{align}
    In view of the off-diagonal decay just mentioned, there exists a threshold 
    $\varepsilon_0 > 0$ such that for any $0 < \varepsilon \leq \varepsilon_0$, 
    the following estimates hold:
    \begin{equation}
        \label{eqn: final-super-convergence-bound-1}
        \|(\mathrm{Id} - P_{N_{r+1}}) L P_{N_{r+1}/3} \| \leq \frac13 e^{ - \frac{7}{12} N_{r+1}^{s} },
    \end{equation}
    and
    \begin{equation}
        \label{eqn: final-super-convergence-bound-2}
        \|(\mathrm{Id} - P_{N_{r+1}/3}) L_{N_{r+1}}^{-1} P_{N_{r+1}/4} \| \leq \frac13 e^{ - \frac{1}{48} N_{r+1}^{s} }.
    \end{equation}
    Noting further that 
    $\|(\mathrm{Id} - P_{N_{r+1}}) LP_{N_{r+1}} \| \leq \varepsilon \|S+B\| \leq 1$, 
    we combine the bounds~\eqref{eqn: final-super-convergence-bound-1} 
    and~\eqref{eqn: final-super-convergence-bound-2} with the 
    decomposition~\eqref{eqn: decomposition-t-tr} to control the truncation 
    errors in~\eqref{eqn: vector-estimate-taylor-r+1} 
    and~\eqref{eqn: vector-estimate-derivative-taylor-r+1}. 
    Hence, after possibly shrinking $\varepsilon_0 > 0$, we obtain
    \begin{equation}
        \| F(\hat{u}_p^{(r+1)}, \omega^{(r+2)}) \| \leq \frac{2}{3} e^{ -\left( \frac{2}{M^{s}} + \frac{1}{48}\right) N_{r+1}^{s} } + 4(2N_{r+1}+1)^\frac{1+m}{2}e^{ - 3N_{r+1}^{s} }.   \label{eqn: final-super-convergence-1}
    \end{equation}
    By choosing $s = s(\varepsilon)>0$ such that $95 M^{s} \leq 96$, we arrive at 
    the desired super-convergence bound:
    \begin{equation}
        \label{eqn: final-super-convergence-2}
        \| F(\hat{u}_p^{(r+1)}, \omega^{(r+2)}) \| \leq e^{ - 2N_{r+1}^{s} }, 
    \end{equation}
    which completes the proof of the Iteration Lemma~(\Cref{lem: iteration}). 
\end{proof}

The decay estimates in~\Cref{lem: iteration} show that, although individual 
increments may fluctuate from one step to the next, the cumulative decay remains 
strictly controlled. 
Indeed, for any multi-index 
$(n, k) \in \Lambda^{1+m}_{N_r} \setminus \Lambda^{1+m}_{N_{r-1}}$, the tail of 
the series is bounded by
\[
    \sum_{\ell = r}^{\infty} e^{-\frac{3N_{\ell}^{s}}{2} } \leq e^{-\frac{5}{4}N_{r}^{s} } \leq e^{ - \left(|n|+|k|\right)^{s} },
\]
where the last inequality follows from $|n|+|k| \leq N_r$. 
This guarantees that the sequence $\hat{u}^{(r)}$ remains uniformly bounded in 
the space $\mathcal{K}(s)$. 
Consequently, the approximate solutions retain their Gevrey regularity, or 
equivalently, the localization of their Fourier coefficients, throughout the 
iteration, and no loss of regularity occurs as the lattice size $N_r$ grows.

\subsection{Proof of the main theorem}
\label{subsec: proof-main}

Since the resonant component $\hat{u}_q^{(r)} = a$ is held fixed throughout the 
iteration, the estimate~\eqref{eqn: difference-derivative-estimate} remains 
valid with $\hat{u}_p$ replaced by $\hat{u}$. 
In particular, the displacement between successive iterates is bounded as:
\[
    \| \hat{u}^{(r+1)} - \hat{u}^{(r)} \| = \| \Delta^{(r)} \| \leq e^{ - \frac{7}{4} N_{r}^s } = e^{ - \frac{7}{4} (M^s)^{r+1} },
\]
so $\{ \hat{u}^{(r)} \}$ is a Cauchy sequence in $\mathcal{K}(s)$ and therefore 
converges to a limit $\hat{u}^{\star} \in \mathcal{K}(s)$. 
Summing the tail of the sequence yields the rate of convergence:
\[
    \| \hat{u}^{(r)} - \hat{u}^\star \| \leq \sum_{\ell = 0}^{\infty} \| \hat{u}^{(r+\ell +1)} - \hat{u}^{(r+\ell )} \| \leq \sum_{\ell = 0}^{\infty} e^{ - \frac{7}{4} (M^s)^{r+\ell +1} } \leq e^{ - \frac{3}{2} (M^s)^{r+1} },
\]
which establishes \eqref{eqn: u-hat-converge}. 
We now turn to the frequencies. 
By the frequency update~\eqref{eqn: alternating-frequency-nlaw} and the uniform 
bound~\eqref{eqn: tangent-bound}, we have:
\begin{equation*}
    | \omega^{(r+2)} - \omega^{(r+1)} | \leq \varepsilon \gamma \| \hat{u}^{(r+1)} - \hat{u}^{(r)} \| \leq \varepsilon \gamma e^{ - \frac{7}{4} (M^s)^{r+1} }.
\end{equation*}
Thus, $\{ \omega^{(r)} \}$ converges to some 
$\omega^{\star} \in \Omega \setminus (\Omega_{M}^{\tau} \cup G)$, and, arguing 
as above,
\[
    | \omega^{(r)} - \omega^\star | \leq \sum_{\ell = 0}^{\infty} | \omega^{(r+\ell +1)} - \omega^{(r+\ell )} | \leq \sum_{\ell = 0}^{\infty} \varepsilon \gamma e^{ - \frac{7}{4} (M^s)^{r+\ell} } \leq e^{ - \frac{3}{2} (M^s)^{r+1} },
\]
which establishes~\eqref{eqn: frequency-converge}. 
It remains to prove the time-domain error bound~\eqref{eqn: pde-solution-error-nlaw}. 
The limit $\hat{u}^{\star}$ is represented by its Fourier expansion as
\[
    u^\star(x, t) = \sum_{(n,k) \in \mathbb{Z}^{1+m}} \hat{u}^\star(n, k) e^{i \langle k, \omega^\star t + \theta_0 \rangle} e^{inx},
\]
while the $r$-th approximate solution $u^{(r)}(x,t)$ is given 
by~\eqref{eqn: approximate-pde-solution}. 
Based on the support property \eqref{eqn: vector-support-r}, the error can be 
split into two terms:
\begin{align*}
    \| u^{(r)}(x, t) - u^\star(x, t) \|_{L^2} 
    \leq & \left\| \sum_{(n,k) \in \mathbb{Z}^{1+m}} \left( \hat{u}^\star(n, k) - \hat{u}^{(r)}(n, k) \right) e^{i \left( \langle k, \omega^\star t + \theta_0 \rangle + nx \right)} \right\|_{L^2} \\
    & + \left\| \sum_{(n,k) \in \Lambda_{N_{r}}^{1+m}} \hat{u}^{(r)}(n, k) e^{inx} \left( e^{ i \langle k, \omega^\star t + \theta_0 \rangle } - e^{ i \langle k, \omega^{(r)} t + \theta_0 \rangle } \right) \right\|_{L^2}  
    \defi I_1 + I_2.
 \end{align*}
For the term $I_1$, combining the support property~\eqref{eqn: vector-support-r}, 
the Gevrey decay~\eqref{eqn: vector-decay-r}, and the Cauchy-Schwarz inequality, 
we obtain
\begin{align*}
    I_1 &\leq \sum_{(n, k) \in \Lambda_{N_{r}}^{1+m} } | \hat{u}^\star(n,k) - \hat{u}^{(r)}(n,k) | + \sum_{(n,k) \notin \Lambda_{N_{r}}^{1+m} } | \hat{u}^\star(n,k) |  \\
        &\leq \| \hat{u}^{(r)} - \hat{u}^\star \|_2 \cdot (2N_r + 1)^{(1+m)/2} + \sum_{(n, k) \notin \Lambda_{N_{r}}^{1+m} } e^{-\frac{5}{4} (|n|+|k|)^s} \leq \frac{1}{2} e^{ - N_r^s } = \frac{1}{2} e^{ - (M^s)^{r+1} }.
\end{align*}
where the last inequality follows from the bound~\eqref{eqn: u-hat-converge} 
together with $|n|+|k| > N_r$ on the complement of $\Lambda_{N_r}^{1+m}$. 
For the term $I_2$, by the uniform boundedness of $\| \hat{u}^{(r)} \|_2$ and 
the Cauchy-Schwarz inequality once more, we have 
\begin{align*}
    I_2 &\leq \| \hat{u}^{(r)} \|_2 \left( \sum_{(n,k) \in \Lambda_{N_{r}}^{1+m}} \left| e^{i \langle k, \omega^\star t + \theta_0 \rangle } - e^{i \langle k, \omega^{(r)} t + \theta_0 \rangle } \right|^2 \right)^\frac{1}{2}   \\
        &\leq \left( \| \hat{u}^\star \|_2 + e^{ - \frac{3}{2} (M^s)^{r+1} } \right) \cdot | \omega^{(r)} - \omega^\star | \cdot \left( \sum_{(n,k) \in \Lambda_{N_{r}}^{1+m}} \|k t \|^2 \right)^\frac{1}{2} \leq \frac12 e^{-N_r^s} = \frac{1}{2} e^{- (M^s)^{r+1} },
\end{align*}
which holds uniformly for any $t \in [0, N_r] = [0, M^{r+1}]$, since 
$|k| \leq N_r$ on the lattice $\Lambda_{N_r}^{1+m}$ and the polynomial factor 
$(2N_r + 1)^{(1+m)/2} N_r^2$ is absorbed by~\eqref{eqn: frequency-converge}. 
Combining the estimates for $I_1$ and $I_2$, we arrive at the final 
time-domain error bound \eqref{eqn: pde-solution-error-nlaw}. 
This completes the proof of~\Cref{thm: main}.

\section{The KdV equation}
\label{sec: kdv}

In this section, we extend the numerical construction of quasi-periodic 
solutions to the KdV equation~\eqref{eqn: kdv} subject to the periodic boundary 
condition~\eqref{eqn: per-bdy}. 
The associated linear equation, obtained by setting $\varepsilon = 0$, reads
\begin{equation}
    \label{eqn: lin-kdv}
    u_t + u_{xxx} = 0.
\end{equation}
Substituting~\eqref{eqn: fourier-expan} into~\eqref{eqn: lin-kdv}, we obtain the 
dispersion relation 
\begin{equation}
    \label{eqn: lin-disper-kdv}
    \varpi_n = n^3,
\end{equation} 
for any $n \in \mathbb{Z}$. 
Since $u$ is real-valued, its spatial Fourier coefficients satisfy 
$\hat{u}_n(t) = \overline{\hat{u}_{-n}(t)}$ for any $n \in \mathbb{Z}$. 
Hence, the solution can be written as
\begin{equation}
    \label{eqn: kdv-fourier-1}
    u(x, t) = \hat{u}_0(t) + 2 \sum_{n=1}^{\infty} \mathrm{Re} \left( \hat{u}_n(t) e^{inx} \right).
\end{equation}
Likewise, the space-time Fourier coefficients satisfy the reality condition 
$\hat{u}(-n, -k) = \overline{ \hat{u}(n, k) }$ for any 
$(n, k) \in \mathbb{Z}^{1+m}$. 
In what follows, we seek solutions whose Fourier coefficients
$\hat{u}(n, k)$ are all real. 
This restriction is compatible with the reversibility of the KdV 
equation~\eqref{eqn: kdv}, if $u(x,t)$ is a solution, then so is $u(-x,-t)$. 
For such solutions, $\hat{u}(n, k) = \hat{u}(-n, -k)$, and the 
expansion~\eqref{eqn: kdv-fourier-1} reduces to 
\begin{equation}
    \label{eqn: kdv-fourier-2}
    u(x, t) = \sum_{k \in \mathbb{Z}^m} \hat{u}(0, k) \cos(\langle k, \omega' t + \theta_0 \rangle) + 2 \sum_{n=1}^{\infty} \sum_{k \in \mathbb{Z}^m} \hat{u}(n, k) \cos\left( \langle k, \omega' t + \theta_0 \rangle + nx \right). 
\end{equation}
As in~\eqref{eqn: 0-k-mode} for the NLAW equation~\eqref{eqn: nlaw}, the zeroth 
spatial mode of the nonlinear term $\varepsilon (u^2)_x$ vanishes identically. 
Thus, the equation for the $k$-th temporal Fourier mode of $\hat{u}_0$ reduces 
to $i \langle k, \omega'\rangle \hat{u}(0, k) = 0$. 
For drifted frequency vectors $\omega'$ outside the excluded set, we have 
$\langle k, \omega' \rangle \neq 0$ for all $k \neq 0$, and hence 
$\hat{u}(0, k) = 0$ for all $k \neq 0$; only the constant mode $\hat{u}(0,0)$ 
remains undetermined. 
Let $b_0 :=\hat{u}(0, 0)$ denote the constant mode and write $u = b_0 + v$, 
where $v$ has zero spatial mean. 
Substituting this decomposition into~\eqref{eqn: kdv} yields 
\begin{equation}
    \label{eqn: kdv-non-zero}
    v_t +2\varepsilon b_0v_x + v_{xxx} + \varepsilon (v^2)_{x} = 0. 
\end{equation}
Thus, a nonzero constant mode introduces only an additional first-order 
transport term $2\varepsilon b_0v_x$, or equivalently, shifts the linear 
frequencies from $n^3$ to $n^3 - 2\varepsilon b_0 n$, while leaving the 
dispersive and nonlinear terms unchanged. 
In what follows, we restrict attention to the zero-mean case $b_0=0$, for 
which~\eqref{eqn: kdv-fourier-2} reduces to 
\begin{equation}
    \label{eqn: kdv-fourier-3}
    u(x, t) = 2 \sum_{n=1}^{\infty} \sum_{k \in \mathbb{Z}^m} \hat{u}(n, k) \cos\left( \langle k, \omega' t + \theta_0 \rangle + nx \right).
\end{equation}
The case $b_0 \neq 0$ is discussed at the end of the paper.

Given that  $M=M(\varepsilon)$ is a truncation parameter defined 
in~\eqref{eqn: M} depending on the perturbation parameter $\varepsilon$, we 
select a finite set of $m$ distinct wavenumbers in the low-frequency range, 
referred to as the tangential modes,
\begin{equation}
    \label{eqn: tan-kdv}
    \mathcal{N} = \left\{ 0 < n_1 < n_2 < \ldots < n_m \right\} \in \mathbb{N}^+, 
\end{equation}
with $ n_{m} \leq M_{0}^{\frac13}/2 - 1$, where $M_0=M_0(\varepsilon)$ is a 
truncation parameter defined in~\eqref{eqn: M0}, which depends on the 
perturbation parameter $\varepsilon$. 
Accordingly, the linear solution takes the form
\begin{equation}
    \label{eqn: lin-soln-kdv}
    u(x, t) =  2 \sum_{j=1}^{m} a_j \cos\left( \omega_j t + \theta_{0,j} + n_j x \right),
\end{equation}
where $a_j \geq \delta_0$ and $\omega_j = \varpi_{n_j}$ for $j=1, \ldots, m$. 
We define the resonant set as
\begin{equation}
    \label{eqn: resonant-kdv}
    \mathcal{S} := \left\{ (n_j, e_j) \in \mathbb{N}^+ \times \mathbb{Z}^{m} ~\big|~ j =1, \ldots, m, ~ n_j \in \mathcal{N}  \right\} 
\end{equation}
Our aim is to construct quasi-periodic solutions of the KdV 
equation~\eqref{eqn: kdv} of the form
\begin{equation}
    \label{eqn: nonlin-soln-kdv}
    u(x, t) = 2 \sum_{n=1}^{\infty} \sum_{k \in \mathbb{Z}^m} \hat{u}(n, k) \cos \left( \langle k, \omega't + \theta_0 \rangle + nx \right),
\end{equation}
with real coefficients $\hat{u}(n, k)$ and a drifted frequency vector 
$\omega' \in \mathbb{R}^m$ close to $\omega$. 
The Fourier coefficients on the resonant set $\mathcal{S}$ are prescribed by 
$\hat{u}(n_j, e_j) = a_j$ for $j=1, \ldots, m$, whereas the frequency vector 
$\omega'$ and the remaining Fourier coefficients $\hat{u}(n,k)$ for any 
$(n,k) \in \mathcal{Z} \setminus \mathcal{S}$ are determined successively 
through the iteration.

\subsection{Numerical scheme}
\label{subsec: numerical-scheme-kdv}

In this section, we propose an alternating scheme that progressively deforms the 
linear solution~\eqref{eqn: lin-soln-kdv} into the nonlinear quasi-periodic 
solution~\eqref{eqn: nonlin-soln-kdv}. 
For computational convenience, we extend the Fourier coefficients 
$\hat{u}(n, k)$ to the whole lattice $\mathcal{Z}$, requiring them to be real 
and to satisfy $\hat{u}(n,k) = \hat{u}(-n,-k)$ for any $(n, k) \in \mathcal{Z}$. 
We then use the full Fourier expansion
\begin{equation}
    \label{eqn: nonlin-soln-kdv-full}
    u(x, t) = \sum_{(n,k) \in \mathcal{Z} } \hat{u}(n, k) e^{i \left( \langle k, \omega't + \theta_0 \rangle + nx \right)}.
\end{equation}
These conditions ensure that $u$ is real-valued and the 
representation~\eqref{eqn: nonlin-soln-kdv-full} agrees with the cosine 
expansion~\eqref{eqn: nonlin-soln-kdv}. 
The resonant set~\eqref{eqn: resonant-kdv} is correspondingly extended to the 
$2m$-point set
\begin{equation}
    \label{eqn: resonant-double-kdv}
    \mathcal{S} := \left\{ (n_j, e_j), (-n_j, -e_j) \in  \mathcal{Z} ~\big|~ j =1, \ldots, m, ~ n_j \in \mathcal{N} \right\}, 
\end{equation}
where $n_1, \ldots, n_m \in \mathbb{N}^+$ are the spatial modes fixed 
in~\eqref{eqn: tan-kdv}. 
On the resonant set $\mathcal{S}$, we prescribe 
$\hat{u}(n_j, e_j) = \hat{u}(-n_j, - e_j) = a_j$ for $j=1,\ldots, m$, in 
accordance with~\eqref{eqn: resonant-kdv}. 
The extended amplitude vector 
$a = (a_1, \ldots, a_m; a_1, \ldots, a_m) \in \mathbb{R}^{2m}$ therefore 
contains only $m$ independent parameters. 
Substituting~\eqref{eqn: nonlin-soln-kdv-full} into the KdV 
equation~\eqref{eqn: kdv} and matching the coefficients, we obtain 
\begin{equation}
    \label{eqn: lattice-kdv}
    \left( \langle k, \omega' \rangle - \varpi_n  \right) \hat{u}(n, k) + \varepsilon n \widehat{u^2}(n, k) = 0,
\end{equation}
for any multi-index $(n, k) \in \mathcal{Z}$, where the Fourier coefficients of 
$u^2$ are given by the convolution
\begin{equation}
    \label{eqn: pert-kdv}
    \widehat{u^2}(n, k) = \sum_{n_1 + n_2 = n} \sum_{k_1 + k_2 = k}\hat{u}(n_1, k_1) \hat{u}(n_2, k_2). 
\end{equation}
This convolution preserves both the reality and the symmetry of the Fourier 
coefficients.

Since $n\neq 0$ for $(n,k)\in\mathcal{Z}$, we divide~\eqref{eqn: lattice-kdv} by 
$n$ and work with the equivalent lattice equation
\begin{equation}
    \label{eqn: lattice-kdv-new}
    \frac{\langle k, \omega' \rangle - \varpi_n}{n} \hat{u}(n,k) + \varepsilon \widehat{u^2}(n,k) = 0.
\end{equation}
Setting $\hat{X}:=\widehat{u^2}$, we collect the coefficients $\hat{u}(n,k)$ into 
a vector $\hat{u}$ and rewrite~\eqref{eqn: lattice-kdv-new} as
\begin{equation}
    \label{eqn: lattice-vector-kdv}
    \mathcal{D}\hat{u} + \varepsilon\hat{X} = 0,
\end{equation}
where $\mathcal{D}$ is a diagonal operator indexed by $\mathcal{Z}$, with 
diagonal entries $\mathcal{D}(n,k)=(\langle k, \omega' \rangle - \varpi_n) / n$.
As in the NLAW case, we decompose the Fourier vector and the nonlinear vector 
field into resonant and non-resonant components, $\hat{u}=(\hat{u}_q,\hat{u}_p)$ 
and $\hat{X}=(\hat{X}_q,\hat{X}_p)$, according to the resonant set $\mathcal{S}$ 
defined in~\eqref{eqn: resonant-double-kdv}.
The $Q$-equations then read
\begin{equation}
    \label{eqn: q-eqn-kdv}
    (\omega'-\omega) + \varepsilon A^{-1}W_q\hat{X}_q=0,
\end{equation}
where $A=\mathrm{diag}(a)$ and $W_q$ is the diagonal operator whose entries at 
$(n_j,e_j)$ and $(-n_j,-e_j)$ equal $n_j$.
The $P$-equations for the non-resonant indices read
\begin{equation}
    \label{eqn: p-eqn-kdv}
    D\hat{u}_p + \varepsilon\hat{X}_p = 0,
\end{equation}
where $D$ is the restriction of $\mathcal{D}$ to $\mathcal{Z} \setminus \mathcal{S}$.
For convenience, we define the nonlinear operator $F$, acting on the 
non-resonant Fourier coefficients $\hat{u}_p$, as
\begin{equation}
    \label{eqn: total-vector-kdv}
    F(\hat{u}_p, \omega'; a)=D \hat{u}_p + \varepsilon \hat{X}_p.
\end{equation}
The associated tangent operator is $T=D+\varepsilon S$, where $S$ is a Laurent 
operator (see~\citet{trefethen2005spectra}) with entries 
$S((n,k),(n',k')) = 2\hat{u}(n-n',k-k')$.
Setting the truncation scale $N_r:=M^{r+1}$, we propose the following alternating 
numerical scheme, built from the $Q$-equations~\eqref{eqn: q-eqn-kdv} and the
$P$-equations~\eqref{eqn: p-eqn-kdv}:
\begin{subequations}
    \label{eqn: alternating-numerical-kdv}
    \begin{empheq}[left=\empheqlbrace]{align}
       \omega^{(r+1)} &=\omega - \varepsilon A^{-1}W_q \hat{X}_q(\hat{u}_p^{(r)}; a),    \label{eqn: alternating-frequency-kdv}\\
       \hat{u}_p^{(r+1)} &=\hat{u}_p^{(r)} -\left[(T + \varepsilon B)_{N_{r+1}} (\hat{u}_p^{(r)}, \omega^{(r+1)}; a)\right]^{-1} F(\hat{u}_p^{(r)}, \omega^{(r+1)}; a).    \label{eqn: alternating-coefficient-kdv}
    \end{empheq}
\end{subequations}
where the correction operator $B$, introduced to ensure the convergence, is 
given by
\begin{equation}
    \label{eqn: operator-b-kdv}
    B = - \frac{\partial \langle k, A^{-1}W_q\hat{X}_q \rangle}{\partial \hat{u}_p} W_p^{-1}\hat{u}_p^{\top}.
\end{equation}

\subsection{The main theorem: a priori bound }
\label{subsec: main-kdv}

In this section, we establish convergence and a priori bounds for the 
alternating numerical procedure~\eqref{eqn: alternating-numerical-kdv}. 
As in the NLAW case, the linear frequencies $\omega_j= n_j^3$ remain fixed, 
while the updated frequency $\omega'$ is determined at each iteration by the 
$Q$-equations
\[
    Q(\omega', a) = (\omega' - \omega) \odot a +\varepsilon W_q \hat{X}_q(\omega',a)=0.
\]
For a fixed iteration, suppressing the iteration index, this equation defines an 
implicit relation between the amplitude $a$ and the updated frequency $\omega'$. 
For sufficiently small $\varepsilon$ and $\min_{1\leq j\leq m}a_j\geq\delta_0>0$, 
the frequency map $\omega': a \mapsto \omega'(a)$ is a local diffeomorphism.

Before stating the main convergence theorem, we highlight the key difference 
between the KdV equation~\eqref{eqn: kdv} and the NLAW equation~\eqref{eqn: nlaw}, 
which lies in the exclusion of the ``bad'' frequencies. 
For the KdV equation~\eqref{eqn: kdv}, the diagonal entries of $D$ depend 
linearly on $\langle k,\omega'\rangle$, so the difference nearly-resonant set 
$\Omega_{2,M}^{+}$ defined in~\eqref{eqn: difference-nearly-resonant-2} is not 
needed.
Moreover, in the small-scale regime $M_0\leq N\leq M$, the measure of the 
excluded frequencies can be estimated by directly counting the spatial modes. 
The main difficulty therefore lies in the large-scale regime $N>M$.

As shown in~\Cref{sec: scale-reduce-cluster}, the eigenvalues of the tangent 
operator $T$ admit the representation
\begin{equation}
    \label{eqn: t-eigenvalue-kdv}
    T_{n_0,k_0,N}(n,k) =\frac{\langle k, \omega' \rangle-n^3}{n} + \nu(n,k).
\end{equation}
By~\Cref{prop: convolution}, the perturbation satisfies 
$|\nu(n,k)|\leq\sqrt{\varepsilon}$ and
$\|\partial_{\omega'}\nu(n,k)\|_2\leq\sqrt{\varepsilon}$ for any 
$(n,k)\in\mathcal{Z}$.
Since $T$ is a perturbation of the diagonal operator $D$, we need to exclude 
frequencies for which two eigenvalues of $T$ are simultaneously small. 
For two such eigenvalues with absolute values below $2\varepsilon_{N'}$, 
multiplying them by $n_1$ and $n_2$, respectively, gives
\begin{equation}
    \label{eqn: t-eigenvalue-kdv-diff}
    \left| -\langle k, \omega' \rangle + n_1^3 - n_2^3 -n_1\nu(n_1,k_1) + n_2\nu(n_2,k_2) \right| < (|n_1|+|n_2|)\varepsilon_{N'},
\end{equation}
where $k = k_1-k_2 \in \Lambda_{2N}^m \setminus \Lambda_{2N'}^m$.

Suppose that $\|\omega'\|_2\leq R$ throughout $\Omega$. 
The key observation is that, for distinct integers $n_1$ and $n_2$,
\[
    |n_1^3-n_2^3| =|n_1-n_2| (n_1^2+n_1n_2+n_2^2) \geq\frac12(n_1^2+n_2^2) \geq\frac14(|n_1|+|n_2|)^2.
\]
Consequently, if $\|k\|_2R\leq\frac18(|n_1|+|n_2|)^2$, then~\eqref{eqn: t-eigenvalue-kdv-diff} 
cannot hold for sufficiently small $\varepsilon$. 
Thus, no ``bad'' frequencies need to be excluded in this regime.
In the complementary regime, $\|k\|_2R\geq\frac18(|n_1|+|n_2|)^2$, since 
$k\notin\Lambda_{2N'}^m$, we can choose $\varepsilon$ sufficiently small so that
\[
    \|k\|_2-\sqrt{\varepsilon}(|n_1|+|n_2|) \geq \frac12\|k\|_2.
\]
It follows that the directional derivative, along $k/\|k\|_2$, of the expression 
inside the absolute value in~\eqref{eqn: t-eigenvalue-kdv-diff} has magnitude at 
least $\|k\|_2/2$. 
The desired measure estimate then follows by the same argument as 
in~\Cref{sec: scale-reduce-cluster}.

This comparison also explains why the argument accommodates perturbations with 
at most quadratic growth in the spatial index, provided that their frequency 
derivatives satisfy corresponding bounds.
Indeed, in the complementary regime, a derivative correction bounded by
$(n_1^2+n_2^2)\sqrt{\varepsilon}$ can still be absorbed into $\|k\|_2$ by 
choosing $\varepsilon$ sufficiently small. 
For perturbations of higher order, the same comparison no longer guarantees that 
$\|k\|_2$ dominates uniformly in the spatial indices.
With these key differences in place, we now state the main theorem, which 
establishes the convergence of the alternating numerical procedure~\eqref{eqn: alternating-numerical-kdv}
and provides a priori error bounds for its iterates.

\begin{theorem}[Convergence and a priori bound]
    \label{thm: main-kdv}
    Fix $\delta_0 \in (0, 1)$ and $\tau > m + 1$, and let 
    $\mathscr{A} = [\delta_0, 1]^{m}$. 
    There exists a threshold $\varepsilon_0 > 0$ such that, for any 
    $0 < \varepsilon \leq \varepsilon_0$, the following statements hold: 
    \begin{itemize}
        \item[1.] \textbf{Measure of resonance}: 
                  There exist constants $\kappa = \kappa(m, \tau) >0$ and 
                  $\delta = \delta(\varepsilon) > 0$, with 
                  $\delta \rightarrow 0$ as $\varepsilon \rightarrow 0$, such 
                  that the excluded set $\mathscr{A}^{\star} \subset \mathscr{A}$ 
                  of initial amplitudes satisfies 
                  $\mathrm{mes}(\mathscr{A}^{\star}) \leq \kappa + \delta$. 

        \item[2.] \textbf{Iterative sequence}: 
                  For any $ a \in \mathscr{A} \setminus \mathscr{A}^{\star}$, the 
                  alternating numerical scheme~\eqref{eqn: alternating-numerical-kdv}, 
                  initialized at $(\omega^{(0)}, \hat{u}_p^{(0)} ) = (\omega, 0)$, 
                  generates a well-defined sequence of iterates 
                  $\{ (\omega^{(r)}, \hat{u}_p^{(r)}) \}_{r=0}^{\infty}$. 
                  Moreover, for any $r \geq 0$, the updated frequency map 
                  $\omega^{r+1}: a \mapsto \omega^{(r+1)}(a)$ is a local 
                  diffeomorphism.

        \item[3.] \textbf{Convergence}: 
                  There exists a Gevrey exponent $s = s(\varepsilon) \in (0, 1)$ 
                  such that this sequence converges to an exact solution pair 
                  $(\omega^\star, \hat{u}_p^\star)$, where the limiting frequency 
                  map $\omega^{\star}: a \mapsto \omega^{\star}(a)$ is also a local 
                  diffeomorphism, and $\hat{u}_p^{\star} \in \mathcal{K}(s)$.

        \item[4.] \textbf{A priori error bounds}: 
                  For any $r \geq 0$, the iterative sequence 
                  $\{ \omega^{(r)}, \hat{u}_p^{(r)} \}_{r=0}^{\infty}$ satisfies 
                  the following error bounds:
                  \begin{subequations}
                      \label{eqn: convergence-vector-frequency-kdv}
                      \begin{empheq}[left=\empheqlbrace]{align} 
                          \|\hat{u}^{(r)}-\hat{u}^{\star}\| &\leq \exp\left\{ -\frac{3}{2}(M^s)^{r+1} \right\},                    \label{eqn: u-hat-converge-kdv}               \\
                          |\omega^{(r)}-\omega^{\star}| &\leq \exp\left\{ -\frac{3}{2}(M^s)^{r+1} \right\}.                   \label{eqn: frequency-converge-kdv}
                      \end{empheq}
                  \end{subequations}
    \end{itemize}
    Furthermore, define the $r$-th approximate solution by 
    \begin{equation}
        \label{eqn: approximate-pde-solution-kdv}
        u^{(r)}(x,t) = \sum_{ (n, k) \in \mathbb{Z}^{1+m} } \hat{u}^{(r)}(n, k) e^{i \left \langle k, \; \omega^{(r)} t + \theta_0 \right \rangle } e^{inx}. 
    \end{equation}
    Then, for any $t \in [0, M^r]$, the following time-domain error bound holds:
    \begin{equation}
        \label{eqn: pde-solution-error-kdv}
        \|u^{(r)}(x,t)-u^{\star}(x,t)\|_{L^2} \leq \exp \left\{ -(M^s)^{r+1} \right\},
    \end{equation}
    where $u=u^{\star}(x, t)$ is the exact quasi-periodic solution of the KdV 
    equation~\eqref{eqn: kdv} with periodic boundary 
    conditions~\eqref{eqn: per-bdy} associated with 
    $(\omega^{\star}, \hat{u}_p^{\star})$. 
    In particular, for any fixed $t\geq0$, $u^{(r)}(x,t)$ converges in $L^2$ to 
    $u^\star(x,t)$ as $r \rightarrow \infty$.
\end{theorem}

\subsection{Numerical demonstration}
\label{subsec: numerical-demo-kdv}

We demonstrate the performance of the alternating numerical 
scheme~\eqref{eqn: alternating-numerical-kdv} for the KdV 
equation~\eqref{eqn: kdv} with periodic boundary conditions~\eqref{eqn: per-bdy}. 
Throughout our numerical tests, we set the initial phase $\theta_0=0$.

\paragraph{Periodic solutions ($m=1$)}
As in the case of the NLAW equation~\eqref{eqn: nlaw}, the linear solution of 
the KdV equation~\eqref{eqn: kdv} is a real-valued traveling wave of the form
\begin{equation}
    \label{eqn: lin-kdv-1}
    u(x, t) = 2a_1 \cos(\omega_1 t + n_1x),
\end{equation}
whose only two nonzero space-time Fourier coefficients satisfy 
$\hat{u}(n_1, 1) = \hat{u}(-n_1, -1)=a_1$. 
The corresponding nonlinear solution has spatial Fourier coefficients of the form
\begin{subequations}
    \label{eqn: kdv-periodic-1}
    \begin{empheq}[left=\empheqlbrace]{align} 
       & \hat{u}_{n_1}(t) = \overline{\hat{u}_{-n_1}(t)} = a_1 e^{i\omega_1't} + \sum_{k \in \mathbb{Z}\setminus \{1\}} \hat{u}(n_1, k) e^{ik\omega_1' t},                  \label{eqn: kdv-periodic-1-1}                \\
       & \hat{u}_{n}(t)   = \sum_{k \in \mathbb{Z}} \hat{u}(n, k) e^{ik\omega_1' t}, \mathrel{\phantom{===============}} n \in \mathbb{Z} \setminus \{-n_1, n_1\}.          \label{eqn: kdv-periodic-1-2}                
    \end{empheq}        
\end{subequations} 
For each nonzero spatial mode, we plot the trajectory in the phase-plane coordinates
\[
    \left(\mathrm{Re} \left(\hat{u}_n(t)\right),~\mathrm{Im} \left(\hat{u}_n(t)\right) \right), 
\]
so that the linear solution traces a circle in each excited mode. 
In the numerical test, we take $n_1=1$ and $a_1=1$, so that $\omega_1 = 1$, and 
set the perturbation parameter to $\varepsilon = 0.5$. 
The results are shown in~\Cref{fig: kdv-p-1torus-traj}.
\begin{figure}[htb!]
    \centering
    \begin{subfigure}[t]{0.325\linewidth}
    \centering
    \includegraphics[scale=0.19]{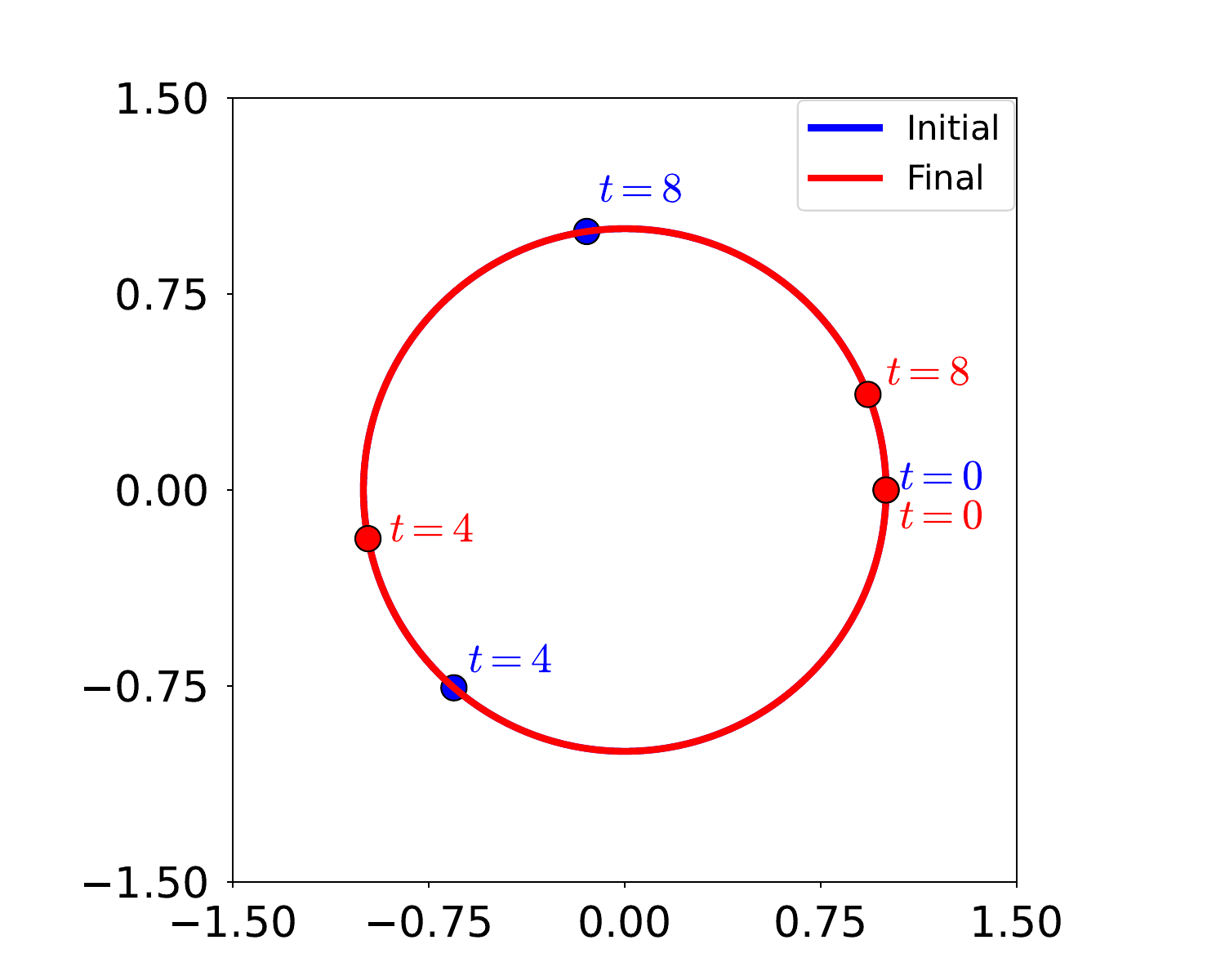}
    \caption{Tangential mode: $\hat{u}_1(t)$}
    \end{subfigure}
    \begin{subfigure}[t]{0.325\linewidth}
    \centering
    \includegraphics[scale=0.19]{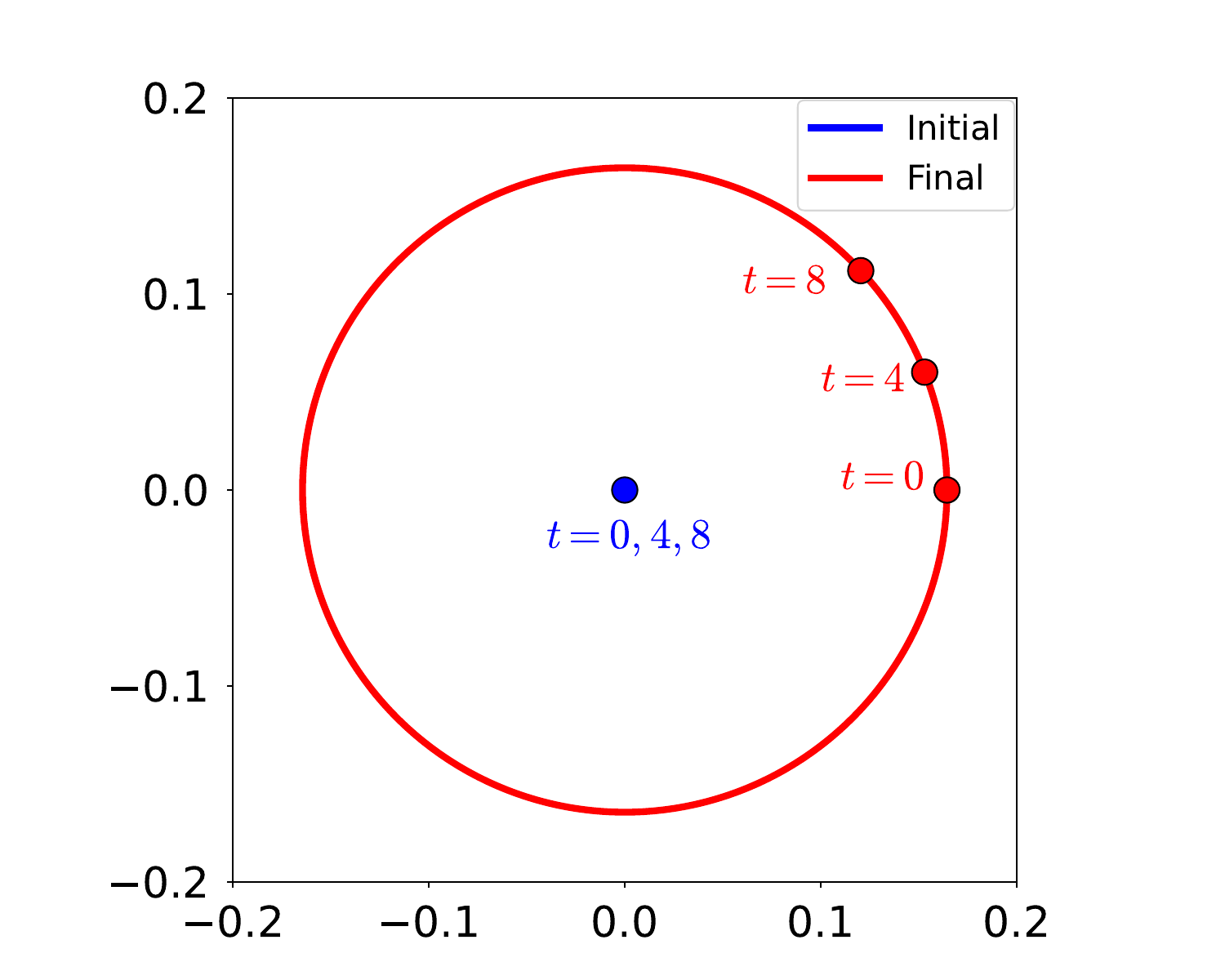}
    \caption{Normal mode: $\hat{u}_2(t)$}
    \end{subfigure}
    \begin{subfigure}[t]{0.325\linewidth}
    \centering
    \includegraphics[scale=0.19]{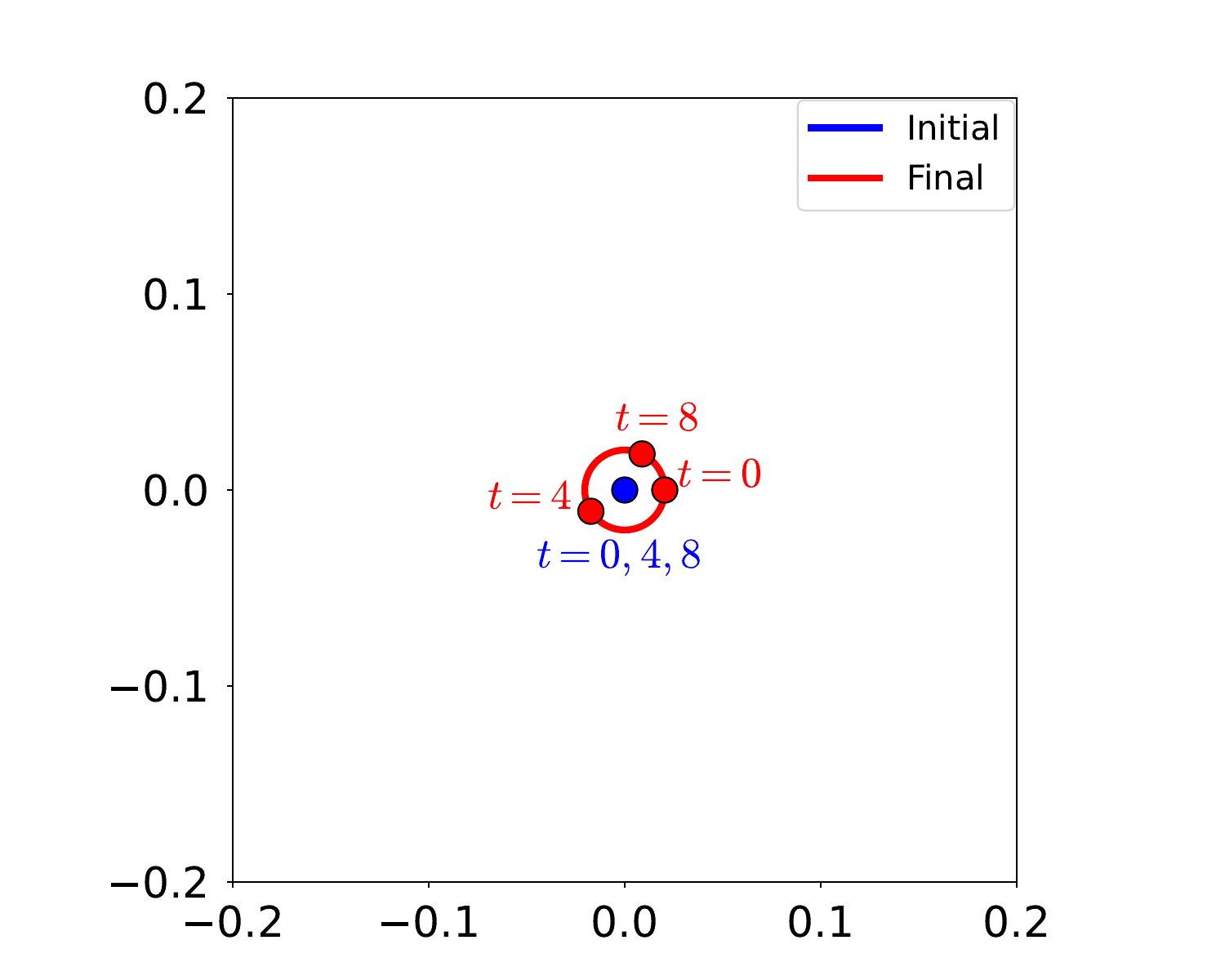}
    \caption{Normal mode: $\hat{u}_3(t)$}
    \end{subfigure}
    \caption{Phase space trajectories of the KdV equation~\eqref{eqn: kdv} with 
             the periodic boundary condition~\eqref{eqn: per-bdy}.  
             ``Initial'' and ``Final'' denote the initial linear solution and 
             the computed nonlinear solution, respectively. 
             Markers indicate specific time points at $t = 0$, $4$, and $8$.} 
    \label{fig: kdv-p-1torus-traj}
\end{figure}
The tangential mode $\hat{u}_1(t)$ maintains an amplitude virtually 
indistinguishable from unity but exhibits a substantial frequency shift 
relative to the linear solution. 
As in the NLAW case, the representative normal modes 
$\hat{u}_2(t)$ and $\hat{u}_3(t)$ develop small-amplitude oscillations through 
nonlinear interactions, with $\hat{u}_2(t)$ excited more strongly than 
$\hat{u}_3(t)$. 
Compared with their counterparts in the NLAW case, however, the 
normal-mode trajectories are more nearly circular.~\Cref{fig: kdv-p-1torus-con} 
illustrates the convergence of the alternating numerical 
scheme~\eqref{eqn: alternating-numerical-kdv}, measured by successive 
differences in the Fourier coefficients, the temporal frequency, and the 
reconstructed solution at $t=4$. 
All three quantities decrease rapidly over the iterations. 
\begin{figure}[htb!]
    \centering
    \begin{subfigure}[t]{0.325\linewidth}
    \centering
    \includegraphics[scale=0.19]{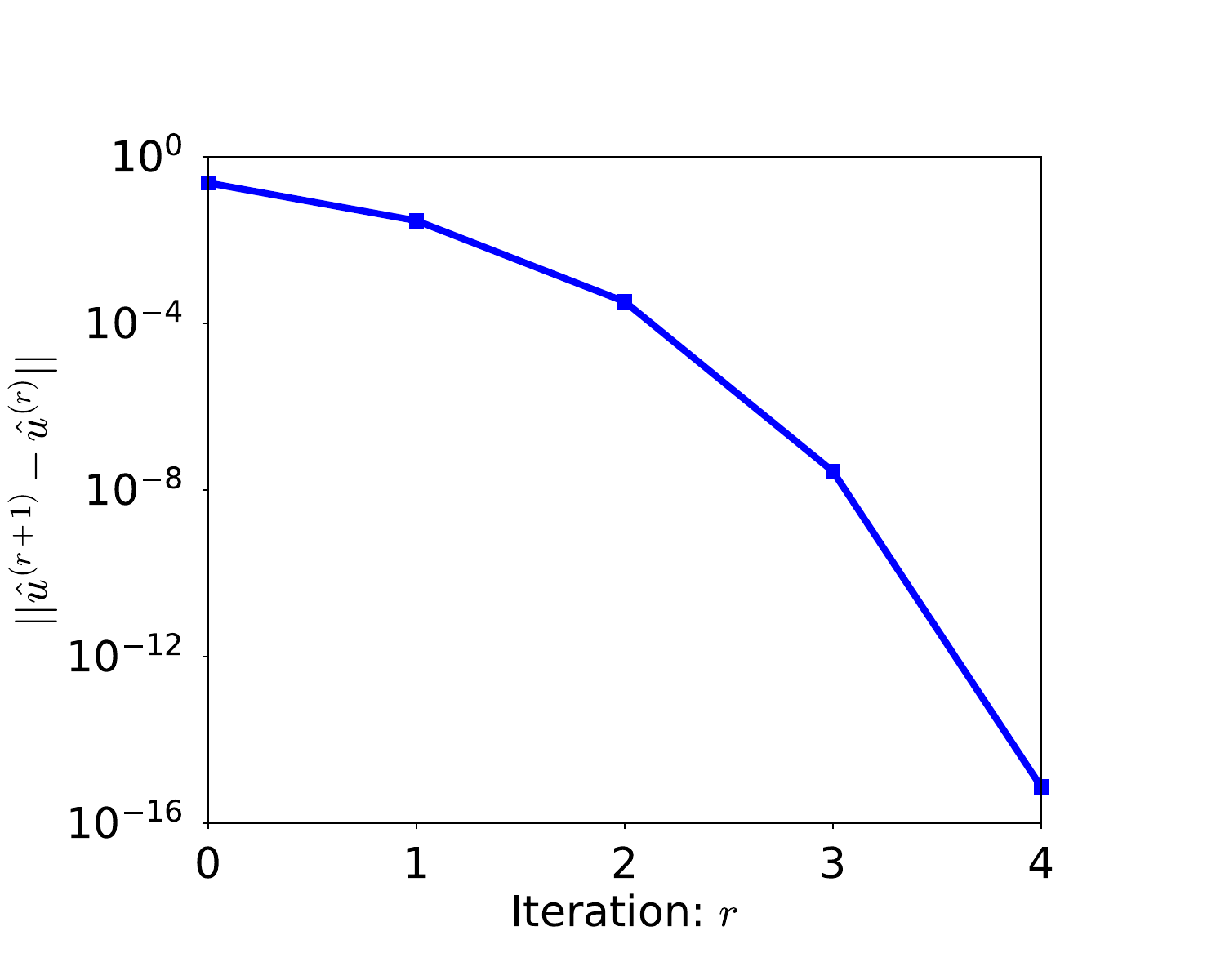}
    \caption{$\| \hat{u}_p^{(r+1)} - \hat{u}_p^{(r)} \|$}
    \end{subfigure}
    \begin{subfigure}[t]{0.325\linewidth}
    \centering
    \includegraphics[scale=0.19]{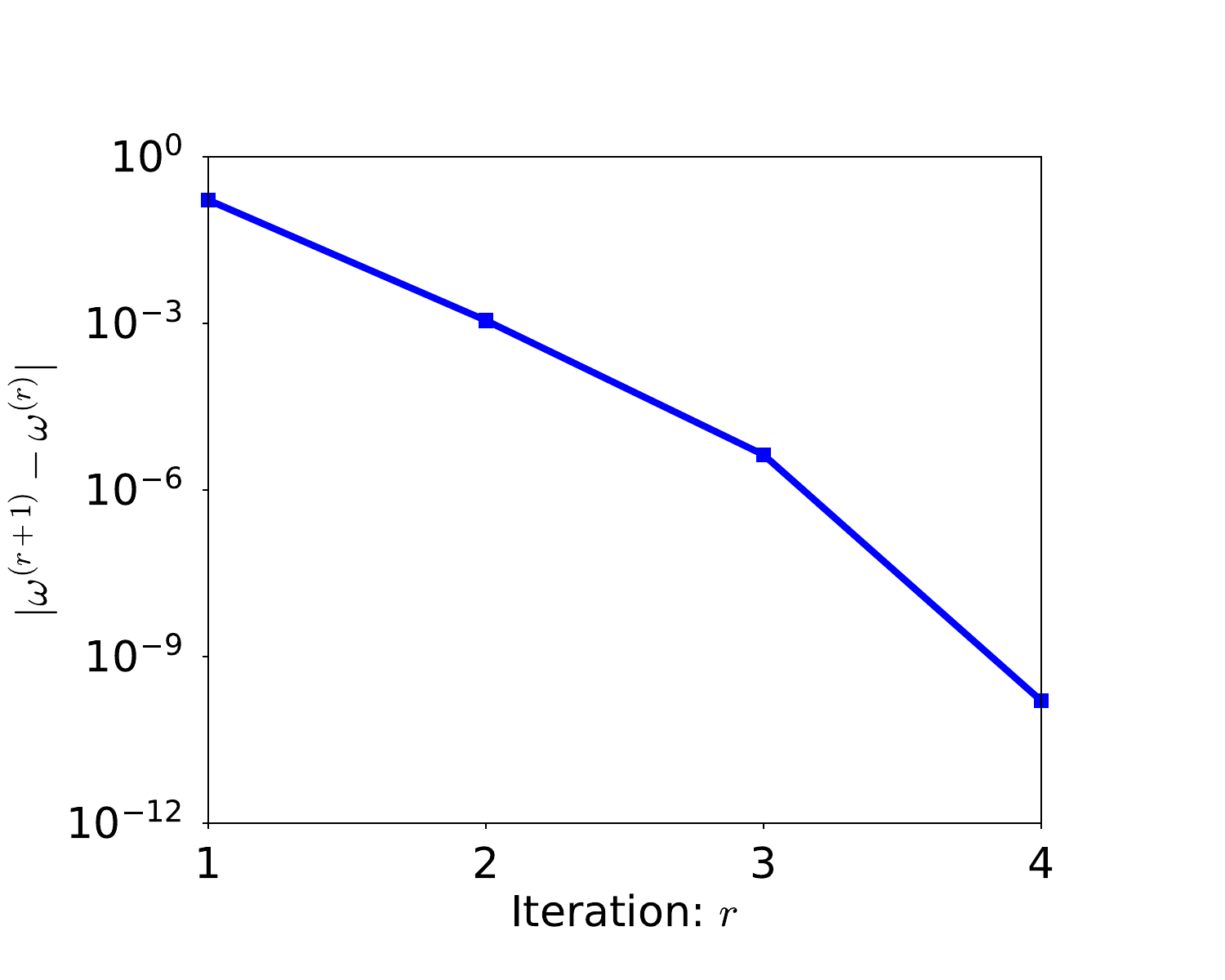}
    \caption{$| \omega^{(r+1)} - \omega^{(r)} |$}
    \end{subfigure}
    \begin{subfigure}[t]{0.325\linewidth}
    \centering
    \includegraphics[scale=0.19]{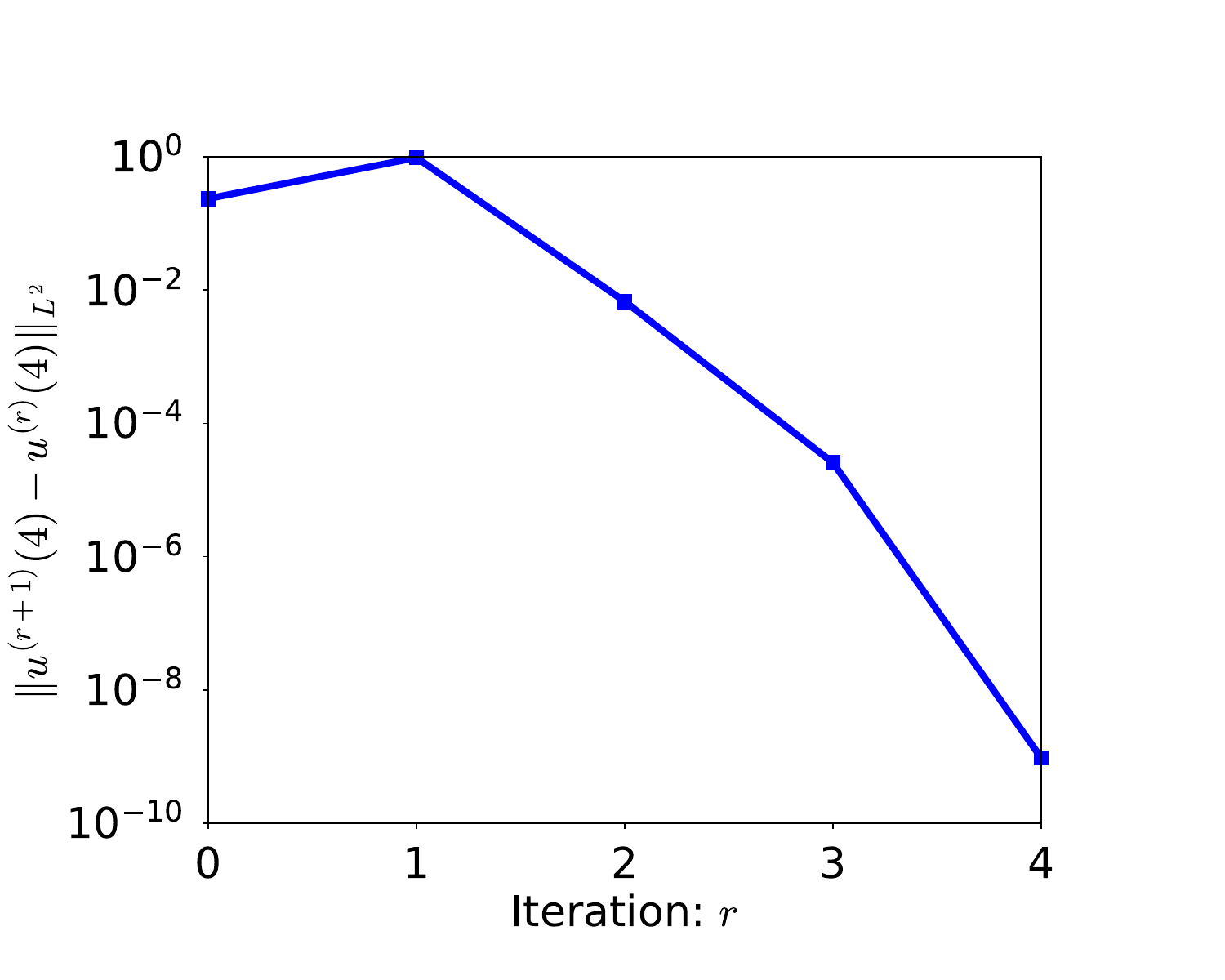}
    \caption{$\| u^{(r+1)}(x, 4) - u^{(r)}(x, 4) \|_{2}$}
    \end{subfigure}
    \caption{Convergence profiles across iterations for the KdV 
             equation~\eqref{eqn: kdv} with the periodic boundary 
             condition~\eqref{eqn: per-bdy}.} 
    \label{fig: kdv-p-1torus-con}
\end{figure}

We next examine the spatiotemporal distribution of $|u(x,t)|^2$. 
For the linear solution~\eqref{eqn: lin-kdv-1}, the squared modulus is given by 
\begin{equation}
    \label{eqn: lin-kdv-1-sq}
    |u|^2 = 4\cos^2\big(x + t\big), 
\end{equation}
which describes a traveling wave propagating in the negative $x$-direction, with 
spatial period $\pi$ and uniform crest height $4$, as shown 
in~\Cref{fig: kdv-p-1torus-heat-slice-lin}. 
Thus, each $2\pi$-period contains two identical crests. 
For the nonlinear solution, shown in~\Cref{fig: kdv-p-1torus-heat-slice-nonlin}, 
the traveling-wave structure and propagation direction are preserved, but the 
intensity profile is deformed.
\begin{figure}[htb!]
    \centering
    \begin{subfigure}[t]{0.46\linewidth}
    \centering
    \includegraphics[scale=0.28]{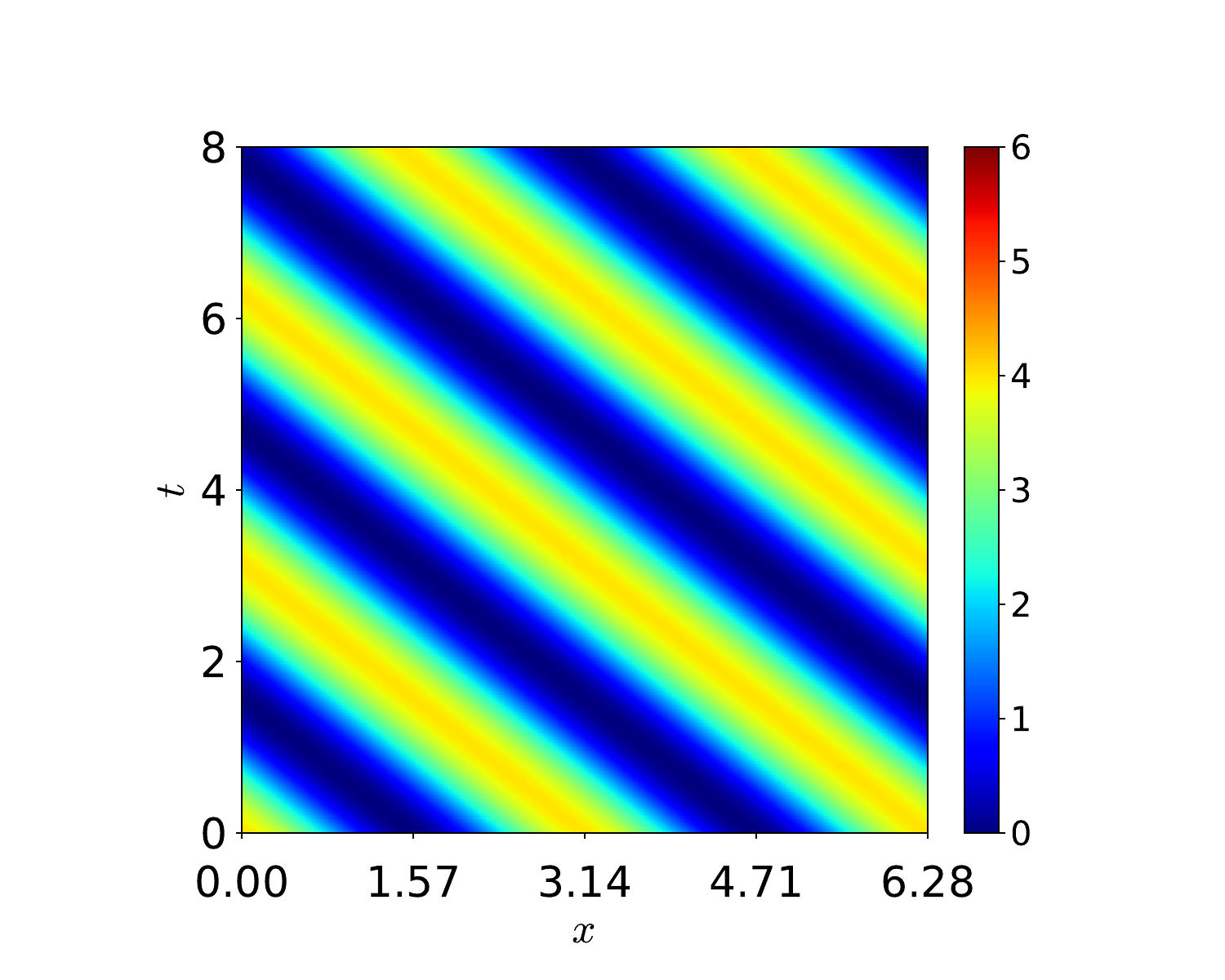}
    \caption{Linear solution}
    \label{fig: kdv-p-1torus-heat-slice-lin}
    \end{subfigure}
    \begin{subfigure}[t]{0.46\linewidth}
    \centering
    \includegraphics[scale=0.28]{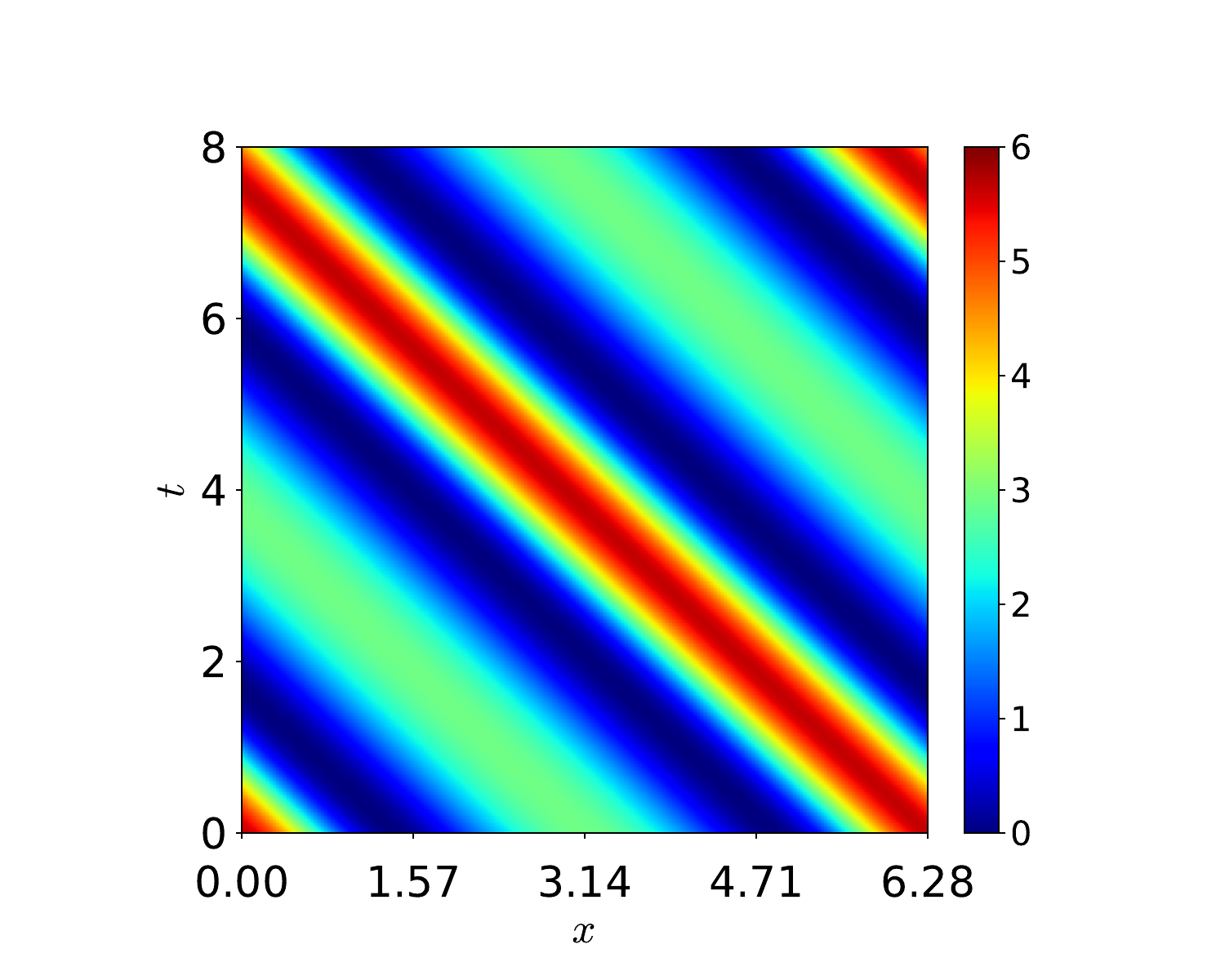}
    \caption{Nonlinear solution}
    \label{fig: kdv-p-1torus-heat-slice-nonlin}
    \end{subfigure} 
    \caption{Spatiotemporal distributions of $|u(x,t)|^2$ for the KdV 
             equation~\eqref{eqn: kdv} with the periodic boundary 
             conditions~\eqref{eqn: per-bdy}. 
             The squared modulus of the linear solution is~\eqref{eqn: lin-kdv-1-sq}, and the 
             nonlinear solution is obtained via the alternating numerical 
             scheme~\eqref{eqn: alternating-numerical-kdv}.}
    \label{fig: kdv-p-1torus-heat-slice}
\end{figure}
The nonlinear solutions of both the NLAW equation~\eqref{eqn: nlaw} and the KdV 
equation~\eqref{eqn: kdv} exhibit a similar deformation: the equally intense 
bright bands of the linear solutions are replaced by alternating stronger and 
weaker bands, with the stronger bands becoming more concentrated. 
The main difference lies in the change in band inclination relative to the 
linear case, which is more pronounced for the KdV equation~\eqref{eqn: kdv} than 
for the NLAW equation~\eqref{eqn: nlaw}. 
This indicates a larger nonlinear change in propagation speed for the nonlinear 
solution of the KdV equation~\eqref{eqn: kdv} in the numerical examples 
considered.

\paragraph{Quasi-periodic solutions ($m=2$)}
For $m=2$, the initial linear solution of the KdV equation~\eqref{eqn: kdv} is a 
superposition of two real traveling waves, which reads 
\begin{equation}
    \label{eqn: lin-kdv-2}
    u(x, t) = 2a_1 \cos(\omega_1 t + n_1x) + 2a_2 \cos(\omega_2 t + n_2x).
\end{equation}
The corresponding nonlinear solution can be written in terms of its spatial 
Fourier coefficients as
\begin{subequations}
    \label{eqn: kdv-periodic-2}
    \begin{empheq}[left=\empheqlbrace]{align} 
       & \hat{u}_{n_1}(t) = \overline{\hat{u}_{-n_1}(t)} = a_1 e^{i\omega_1't} + \sum_{k \in \mathbb{Z}^2\setminus \{(1, 0)\}} \hat{u}(n_1, k) e^{i\langle k, \omega' \rangle t},     \label{eqn: kdv-periodic-2-1}                \\
       & \hat{u}_{n_2}(t) = \overline{\hat{u}_{-n_2}(t)} = a_2 e^{i\omega_2't} + \sum_{k \in \mathbb{Z}^2\setminus \{(0, 1)\}} \hat{u}(n_2, k) e^{i\langle k, \omega' \rangle t},     \label{eqn: kdv-periodic-2-2}                \\
       & \hat{u}_{n}(t)   = \sum_{k \in \mathbb{Z}^2} \hat{u}(n, k) e^{i \langle k, \omega' \rangle t}, \mathrel{\phantom{===============}} n \in \mathbb{Z} \setminus \{\pm n_1, \pm n_2\},          \label{eqn: kdv-periodic-2-3}                
    \end{empheq}
\end{subequations}   
where $\omega' = (\omega_1', \omega_2')$ denotes the shifted frequency vector. 
For the numerical test, we take $n_1 = 1$, $n_2 =2$, and $a_1=a_2 =1$, so that 
$\omega_1=1$ and $\omega_2 = 8$ and set the perturbation quantity 
$\varepsilon = 0.5$.~\Cref{fig: kdv-p-2tori-traj} compares the phase-plane 
trajectories of the initial linear solution and the computed nonlinear solution. 
\begin{figure}[htb!]
    \centering
    \begin{subfigure}[t]{0.325\linewidth}
    \centering
    \includegraphics[scale=0.19]{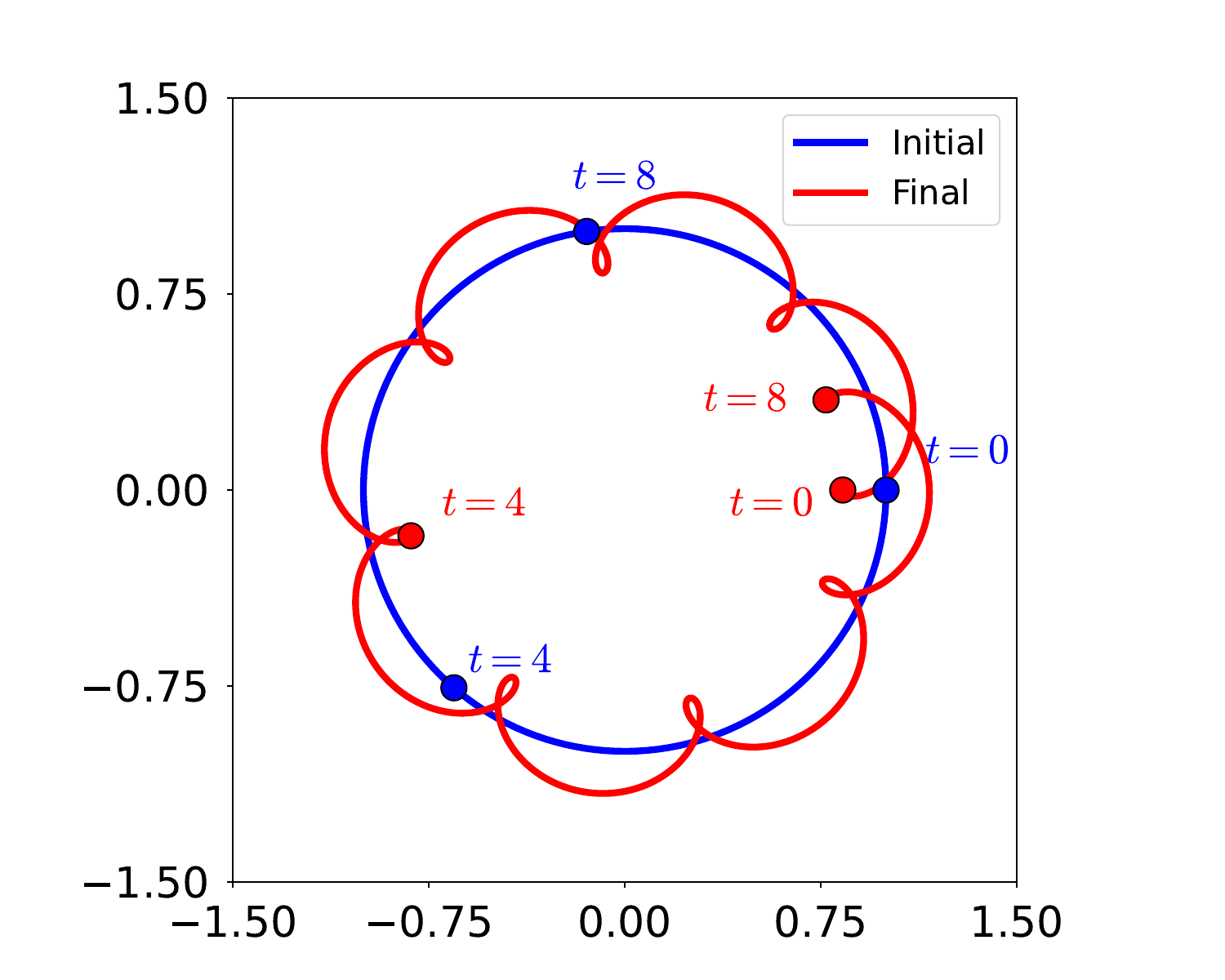}
    \caption{Tangential mode: $\hat{u}_1(t)$}
    \end{subfigure}
    \begin{subfigure}[t]{0.325\linewidth}
    \centering
    \includegraphics[scale=0.19]{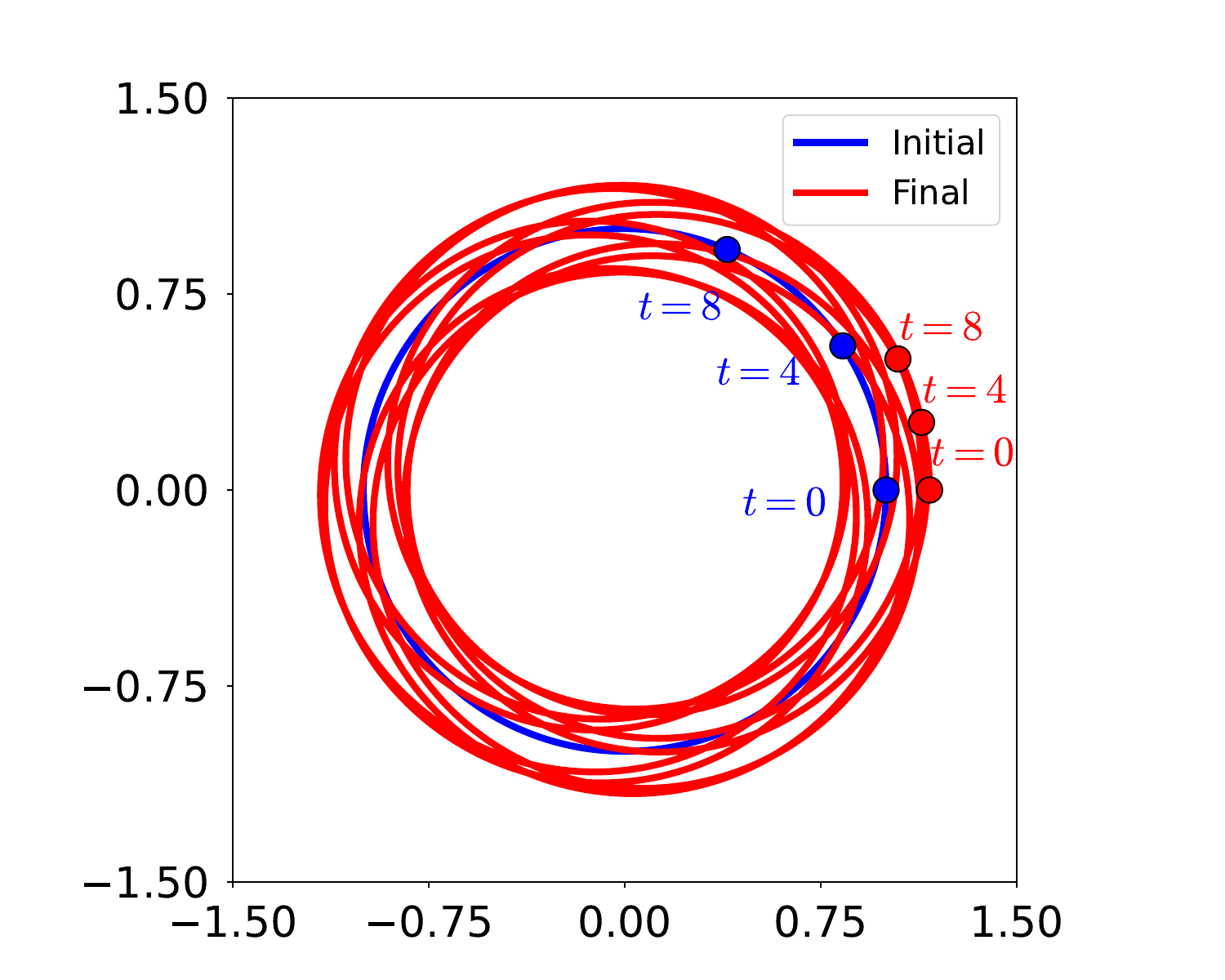}
    \caption{Tangential mode: $\hat{u}_2(t)$}
    \end{subfigure}
    \begin{subfigure}[t]{0.325\linewidth}
    \centering
    \includegraphics[scale=0.19]{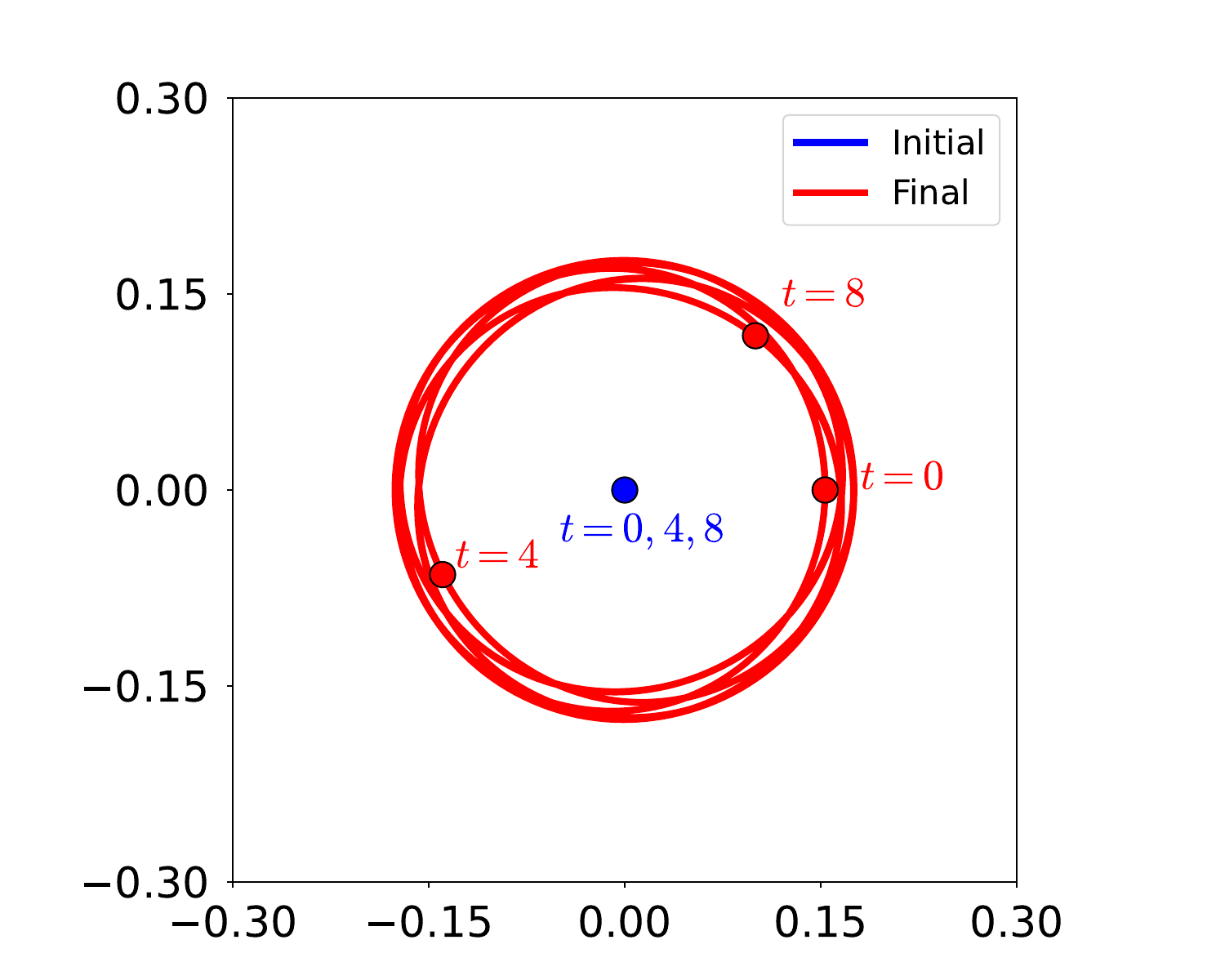}
    \caption{Normal mode: $\hat{u}_3(t)$}
    \end{subfigure}
    \caption{Phase space trajectories of the KdV equation~\eqref{eqn: kdv} with 
             the periodic boundary condition~\eqref{eqn: per-bdy}.  
             ``Initial'' and ``Final'' denote the initial linear solution and 
             the computed nonlinear solution, respectively. 
             Markers indicate specific time points at $t = 0$, $4$, and $8$.} 
    \label{fig: kdv-p-2tori-traj}
\end{figure}
For the tangential mode $\hat{u}_1(t)$, the nonlinear trajectory exhibits a 
flower-like pattern with small loops, departing markedly from the circular 
trajectory of the linear solution. 
In contrast, the trajectory of the tangential mode $\hat{u}_2(t)$ remains 
approximately circular but forms a narrow annular pattern through repeated 
revolutions with amplitude modulation. 
The representative normal mode $\hat{u}_3(t)$, which is absent from the initial 
linear solution, develops small-amplitude oscillations through nonlinear 
interactions, tracing a nearly circular trajectory with slight modulation. 
These modulated trajectories are consistent with the quasi-periodic dynamics of 
the computed nonlinear solution.

Turning to the spatiotemporal distributions, the squared modulus of the linear 
solution~\eqref{eqn: lin-kdv-2}, with the parameters specified above, is
\begin{equation}
    \label{eqn: lin-kdv-2-sq}
    |u|^2 = 4\cos^2 \left( t + x \right) + 4\cos^2 \left( 8t + 2x \right) + 8\cos \left( t + x \right) \cos \left( 8t + 2x \right).
\end{equation}
Expanding the trigonometric products shows that, apart from a constant term, 
$|u(x, t)|^2$ contains four traveling-wave components with distinct phase speeds. 
Their superposition produces the diagonal bands and the intensity variations 
along them shown in~\Cref{fig: kdv-p-2tori-heat-slice-lin}. 
\begin{figure}[htb!]
    \centering
    \begin{subfigure}[t]{0.46\linewidth}
    \centering
    \includegraphics[scale=0.28]{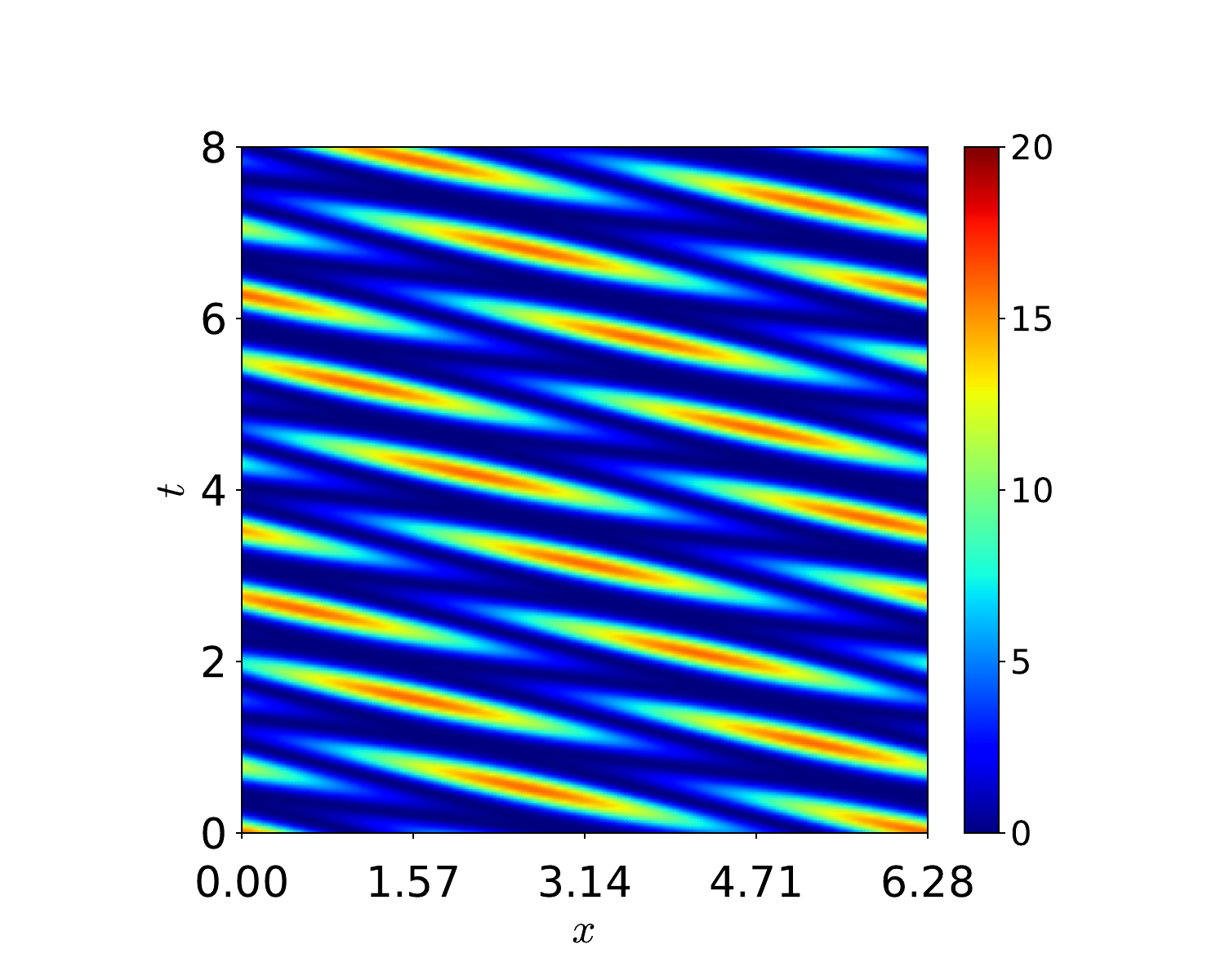}
    \caption{Linear solution}
     \label{fig: kdv-p-2tori-heat-slice-lin}
    \end{subfigure}
    \begin{subfigure}[t]{0.46\linewidth}
    \centering
    \includegraphics[scale=0.28]{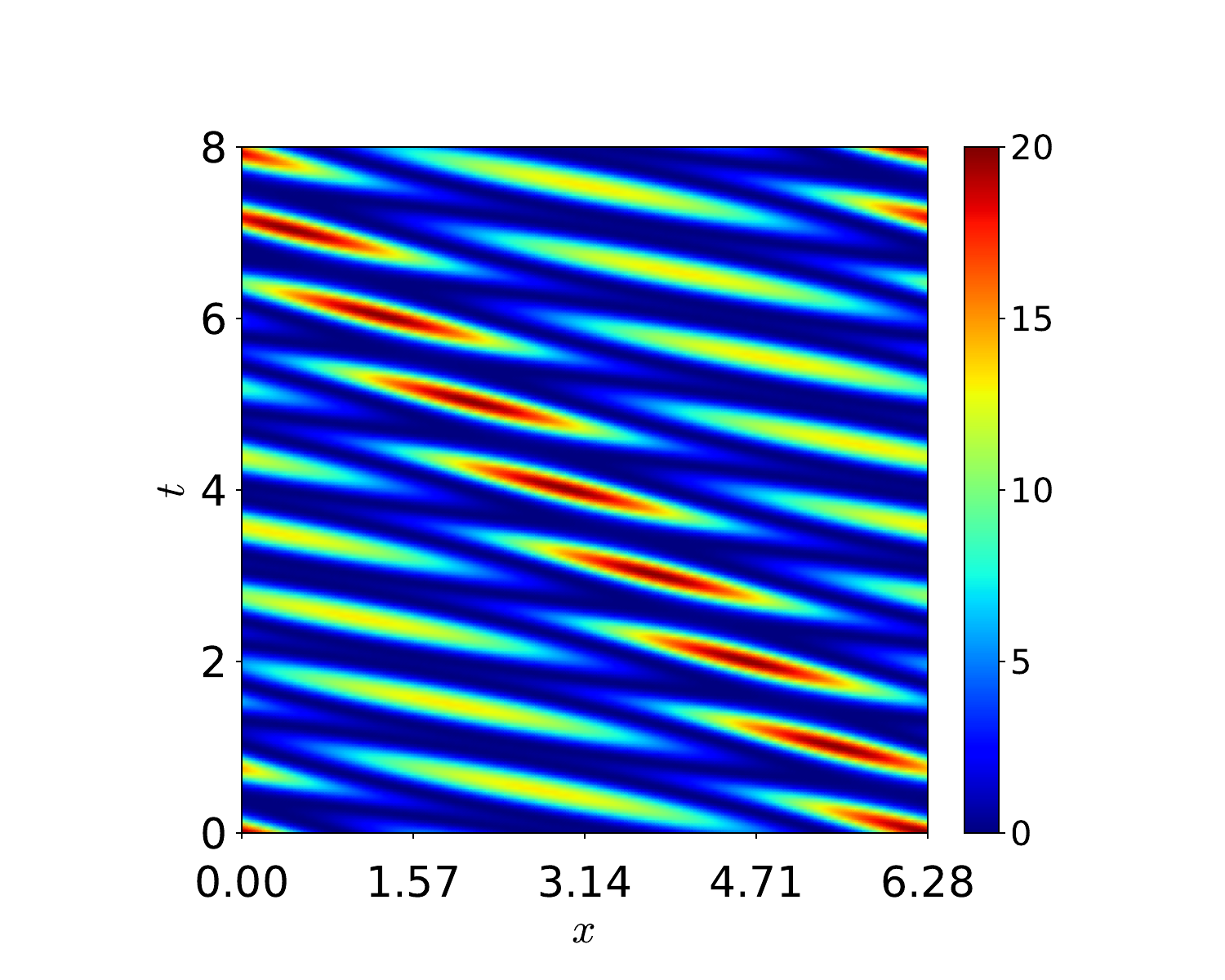}
      \caption{Nonlinear solution}
      \label{fig: kdv-p-2tori-heat-slice-nonlin}
    \end{subfigure} 
    \caption{Spatiotemporal distributions of $|u(x,t)|^2$ for the KdV 
             equation~\eqref{eqn: kdv} with the periodic boundary 
             conditions~\eqref{eqn: per-bdy}. 
             The squared modulus of the linear solution is~\eqref{eqn: lin-kdv-2-sq}, and the 
             nonlinear solution is obtained via the alternating numerical 
             scheme~\eqref{eqn: alternating-numerical-kdv}.}
    \label{fig: kdv-p-2tori-heat-slice}
\end{figure}
In the nonlinear solution, the overall diagonal band structure persists, but the 
locations and relative intensities of the bright regions change. 
As shown in~\Cref{fig: kdv-p-2tori-heat-slice-nonlin}, some bright regions 
become more intense and localized, whereas others are weakened, resulting in a 
more pronounced contrast between successive intensity maxima. 
The bands also exhibit changes in inclination, reflecting the nonlinear 
modification of the underlying temporal frequencies.

\section{Conclusion and further work}
\label{sec: conclu}

In this paper, we propose alternating numerical schemes for constructing 
quasi-periodic solutions to the NLAW equation~\eqref{eqn: nlaw} and the KdV 
equation~\eqref{eqn: kdv} under the periodic boundary 
condition~\eqref{eqn: per-bdy}. 
The solutions constructed here have a traveling-wave rather than a 
standing-wave structure, and thus the present construction does not directly 
extend to Dirichlet boundary conditions. 
For the NLAW equation~\eqref{eqn: nlaw}, replacing the quadratic nonlinearity 
with a cubic one may allow a similar construction under Dirichlet boundary 
conditions, although the physical relevance of such a modification remains to 
be clarified. 
When the zero mode $\hat{u}_0$ is a nonzero constant, corresponding to a trivial 
oscillation, the two equations take the forms~\eqref{eqn: nlaw-non-zero} 
and~\eqref{eqn: kdv-non-zero}, respectively. 
In both cases, the constant mode introduces a correction of lower order in $n$ 
than the original dispersive term. 
The shifted frequencies therefore retain their leading-order behavior, and the 
construction can be adapted to incorporate these modified frequencies. 
The analysis in~\Cref{subsec: main-kdv} requires the unbounded perturbation to 
be at least one order lower than the dispersive frequencies. 
This restriction raises an important question concerning the connection with the 
Fermi-Pasta-Ulam (FPU) model~\citep{fermi1965studies}. 
Although the KdV equation~\eqref{eqn: kdv} provides a long-wave approximation to 
the FPU model, the applicability of our scheme to KdV does not imply that it 
extends to the underlying lattice. 
In particular, for periodic FPU chains, obtaining a construction that is uniform 
as the number of nonlinear oscillators tends to infinity may be obstructed by 
the nonlinear perturbation being of the same order as the linear frequencies. 
Nevertheless, the FPU recurrence phenomenon motivates further investigation 
beyond the scope of the present method. Such work may reveal deeper mathematical 
structures underlying the FPU model and yield a better understanding of its 
recurrence dynamics~\citep{bambusi2006metastability}.

Our results establish quasi-periodic solutions of the NLAW 
equation~\eqref{eqn: nlaw} for a large-measure set of low-spatial modes 
within the invariant class of zero mean velocity. 
The dynamics on the small excluded set, including possible instability, remains 
very interesting. 
A separate question concerns dynamics outside this invariant class. 
Integrating the NLAW equation~\eqref{eqn: nlaw} over the periodic spatial 
domain gives $\hat{u}_0 = b_0 + b_1t$, which rules out quasi-periodicity when 
$b_1 \neq 0$. Writing $u = v + b_0 + b_1t$ with $v$ of zero spatial mean, yields
\[
    v_{tt} - \left[1+2\varepsilon (b_0 + b_1t) \right]v_{xx} + v_{xxxx} - \varepsilon (v^2)_{xx} = 0. 
\]
Linearizing at $v= 0$, we obtain, for each Fourier mode $n \neq 0$, 
\[
    \ddot{\hat{v}}_n = -\left\{n^4 + n^2\left[1+2\varepsilon (b_0 + b_1t) \right] \right\} \hat{v}_n.
\]
The case of interest is $b_1>0$; the opposite case $b_1<0$ admits an 
exponentially growing branch. 
An affine change of time reduces the modal equation to the Airy equation. 
Its asymptotics imply that each nontrivial fixed mode has displacement and 
velocity envelopes of orders $t^{-1/4}$ and $t^{1/4}$ as $t \rightarrow \infty$. 
Consequently, the quadratic modal energy
\[
    E_n(t) = \frac{|\dot{\hat{v}}_n|^2}{2n^2} + \frac{\left[n^2+1+2\varepsilon (b_0 + b_1t) \right] |\hat{v}_n|^2}{2} \sim t^{\frac12}
\]
grows on the order of $t^{1/2}$, despite the decay of the displacement envelope. 
Whether this behavior persists in the nonlinear dynamics is an interesting 
question for future research. 
This mean-driven mechanism for linearized energy growth appears distinct from 
Arnold diffusion~\citep{arnold1964instability}, and does not contradict the 
prevalence of quasi-periodic solutions established here. Indeed, our large-measure result concerns parameters within the invariant zero-mean-velocity class, whereas the growth mechanism above arises from a nonzero mean velocity. Moreover, the three-dimensional NLAW equation 
\[
    u_{tt} - \Delta u + \Delta^2 u + \varepsilon\Delta (u^2) = 0
\]
is a fundamental model for acoustic 
turbulence~\citep{zakharov1992kolmogorov, lvov1997statistical}. 
In forthcoming work, we will build on the multiscale analysis 
of~\citet{bourgain2005green} to establish the existence of quasi-periodic 
solutions for this equation, thereby providing a mathematical framework for 
understanding its recurrent dynamics.

\section*{Acknowledgements}
This work was partially supported by the NSFC (Grant No. 12241105) and by SIMIS (startup fund and cross-disciplinary research projects).

\bibliographystyle{abbrvnat}
\bibliography{sigproc}

\appendix

\section{Proofs of results in~\Cref{sec: basic-setting}}
\label{sec: priori-error-app}

In this appendix, we provide detailed proofs for the results 
in~\Cref{sec: basic-setting}, especially 
for~\Cref{prop: vector-field} ---~\Cref{prop: B-operator} 
(see~\Cref{subsec: vector-field} ---~\Cref{subsec: b-operator}). 
To facilitate these proofs, we first establish an elementary inequality and show 
that the Gevrey decay property is preserved under convolution. 
These preliminary results are rigorously stated in the following lemmas.

\subsection{Preliminary lemmas}
\label{subsec: preliminary}

\begin{lemma}
    \label{lem: app-1}
    Let $a, b \ge 0$ and $0 < s < 1$. Then the following inequality holds:
    \begin{equation}
        \label{eqn: ineq_power}
        a^s + b^s - (a+b)^s \ge (2-2^s) \min\{a,b\}^s. 
    \end{equation}
\end{lemma}

\begin{proof}
    Without loss of generality, we assume that $a \ge b$. 
    The case $b=0$ is immediate. 
    For $b > 0$, set $x = a/b \geq 1$ and define $f(x) = x^s + 1 - (1+x)^s$. 
    Since $0 < s <1$, its derivative satisfies $f'(x) = s(x^{s-1} - (x+1)^{s-1}) > 0$ 
    for all $x \geq 1$. 
    Thus, $f$ is strictly increasing on $[1, \infty)$, and hence 
    $f \geq f(1)= 2 - 2^{s}$. 
    Multiplying this inequality by $b^s$ yields~\eqref{eqn: ineq_power}, so the 
    proof is complete. 
\end{proof}

We next use~\Cref{lem: app-1} to show that convolution preserves the Gevrey 
decay up to a multiplicative constant.

\begin{lemma}
    \label{lem: app-2}
    Let $\mathcal{K}(s)$ be the Gevrey decay set defined 
    in~\eqref{eqn: gevrey-decay-set}, with $0 < s < 1$. 
    There exists a constant $C_2(m, s) > 0$ such that, for any 
    $\hat{a}, \hat{b} \in \mathcal{K}(s)$, their convolution satisfies
    \begin{equation}    
        \label{eqn: convolution}   
        \sup_{(n, k) \in \mathbb{Z}^{1+m}} \left( |(\hat{a}*\hat{b})(n, k)| e^{ (|n| + |k|)^{s} } \right) \leq C_2(m,s).
    \end{equation}
    More generally, for any integer $\ell \geq 2$ and $\hat{a}_j \in \mathcal{K}(s)$ 
    for $j=1, \ldots, \ell$, the $\ell$-fold convolution satisfies: 
    \begin{equation}    
        \label{eqn: p-convolution}   
        \sup_{(n, k) \in \mathbb{Z}^{1+m}} \left( \big| \big(\hat{a}_1 * \hat{a}_2 * \cdots * \hat{a}_{\ell} \big)(n, k) \big| e^{ (|n| + |k|)^{s} } \right) \leq C_{\ell}(m,s),
    \end{equation}   
     where $C_{\ell}(m, s) > 0$ depends only on $\ell$, $m$, and $s$.  
\end{lemma}

\begin{proof}
    Since $\hat{a}, \hat{b} \in \mathcal{K}(s)$, their convolution satisfies
    \begin{align*}
        |(\hat{a}*\hat{b})(n, k)| & \leq \sum_{(n', k') \in \mathbb{Z}^{1+m}} |\hat{a}(n-n', k-k')| |\hat{b}(n', k')|            \\
        & \leq \sum_{(n', k') \in \mathbb{Z}^{1+m}} e^{- (|n-n'| + |k-k'|)^{s}  } e^{- (|n'| + |k'|)^{s} }                       \\
        & \leq \sum_{(n', k') \in \mathbb{Z}^{1+m}} e^{- (|n| + |k|)^{s} } e^{ - ( |n-n'| + |k-k'|)^{s} - (|n'| + |k'|)^{s} + ( |n-n'| + |n'| + |k-k'| + |k'| )^{s} },                                                                                                                                                                   
    \end{align*}
    where the last inequality follows from the triangle inequality. 
    Applying~\Cref{lem: app-1} to bound the sum involving the remaining 
    exponential terms, we obtain
    \[
        |(\hat{a}*\hat{b})(n, k)| \leq e^{ - (| n | + |k|)^{s} } \left( \sum_{(n', k') \in \mathbb{Z}^{1+m}} e^{ - (2 - 2^{s}) \min\left\{ (|n-n'| + |k-k'|)^{s}, \; (|n'| +|k'|)^{s} \right\} } \right).     
    \]
    Since the sum of the Gevrey decay series converges in a $(1+m)$-dimensional 
    lattice $\mathbb{Z}^{1+m}$, we establish the uniform 
    bound~\eqref{eqn: convolution}. 
    By induction, this same argument extends to any finite $\ell$-fold 
    convolutional product, thereby establishing the 
    inequality~\eqref{eqn: p-convolution}.
\end{proof}

\subsection{Proof of~\Cref{prop: vector-field}}
\label{subsec: vector-field}

Since $\hat{u} \in \mathcal{K}(s)$, applying~\Cref{lem: app-2} to the 
convolution representation of the quadratic nonlinearity $u^2$ yields uniform 
Gevrey decay estimates for $\hat{X}(n, k)$. 
By the Leibniz rule, each frequency derivative 
$\partial_{\omega_{\ell}'} \hat{X}(n, k)$, $\ell = 1, \ldots, m$, is a sum of 
convolutions involving the Fourier coefficients $\hat{u}(n, k)$ and their 
frequency derivatives $\partial_{\omega_{\ell}'} \hat{u}(n, k)$. 
A further application of~\Cref{lem: app-2} yields the same Gevrey decay 
estimates for these derivatives, establishing~\eqref{eqn: vector-field-comp}. 
Finally, summing the squares of the coefficientwise bounds over 
$(n, k) \in \mathcal{Z}$ gives the uniform norm 
bound~\eqref{eqn: vector-field-norm}, since the corresponding Gevrey decay 
series converges. 
This completes the proof.

\subsection{Proof of~\Cref{prop: convolution}}
\label{subsec: convolution}

Let $\mathcal{H} = \partial \hat{X} / \partial \hat{u}$ denote the full tangent 
operator. 
For any $(n, k), (n', k') \in \mathcal{Z}$, its entries are given by
\[   
    \mathcal{H} (n - n', k- k') = \frac{2}{(2\pi)^{1+m}} \iint_{\mathbb{T}^{1+m}} u e^{-i\left[ \langle k - k', \theta \rangle + (n - n') x \right]} d\theta dx = 2\hat{u}(n - n', k- k'),
\]
so that $\mathcal{H}$ acts as a convolution operator with kernel $2\hat{u}$. 
Since $\hat{u} \in \mathcal{K}(s)$, the convolution kernel satisfies a uniform 
Gevrey decay estimate. 
Restricting these estimates to the non-resonant indices 
gives~\eqref{eqn: tangent-comp}.

For any $\hat{u} \in \mathcal{K}(s)$, the convolution representation of 
$\mathcal{H}$ and the Cauchy–Schwarz inequality yield
\begin{align}
    \| \mathcal{H} \hat{u} \|^2 &= 4 \sum_{(n,k) \in \mathcal{Z}} \left\| \sum_{(n',k') \in \mathcal{Z}}\hat{u}(n', k') \hat{u}(n - n', k-k') \right\|^2                                                                                \nonumber  \\
    & \leq 16 \sum_{(n,k) \in \mathcal{Z}} \left( \sum_{(n',k') \in \mathcal{Z}} \| \hat{u}(n',k') \|^{\frac12} \cdot \| \hat{u}(n', k') \|^{\frac12} \|\hat{u}(n - n', k-k')\| \right)^2       \nonumber \\     
    & \leq 16 \left( \sum_{(n',k') \in  \mathcal{Z}} \| \hat{u}(n',k') \| \right) \cdot \left( \sum_{(n,k) \in \mathcal{Z}} \sum_{(n',k') \in \mathcal{Z}} \| \hat{u}(n', k') \| \|\hat{u}(n - n', k-k')\|^2 \right)                                                                                                                     \nonumber  \\
    & = 16\bigg[ \sum_{ (n, k) \in \mathcal{Z}} \|\hat{u}(n, k)\| \bigg]^2 \| \hat{u} \|^2,       \label{eqn: h1-bound}
\end{align}
where the last equality follows from interchanging the order of summation and 
using the translation invariance of the lattice sums. 
Since $\hat{u} \in \mathcal{K}(s)$, the inequality~\eqref{eqn: h1-bound} 
provides a uniform bound on the operator norm of $\mathcal{H}$. 
Finally, $\partial \hat{X}_q / \partial \hat{u}_p$, $S$, and 
$\partial \hat{X} / \partial \hat{u}_p$ are obtained from $\mathcal{H}$ by 
restricting its input and/or output indices to the corresponding resonant or 
non-resonant sets. 
Since each such restriction is realized by composing $\mathcal{H}$ with 
orthogonal projections, none of these operators has operator norm exceeding that 
of $\mathcal{H}$. 
This establishes~\eqref{eqn: tangent-bound} and completes the proof.

\subsection{Proof of~\Cref{prop: tensor}}
\label{subsec: tensor}

As in the proofs for the vector field and its tangent operator, we consider the 
second derivative $\mathcal{T} = \partial^2 \hat{X} / \partial \hat{u}^{2}$ as a 
bilinear operator. 
For any $(n, k), (n', k'), (n'', k'') \in \mathcal{Z}$, its entries are given by
\begin{align*}
    \mathcal{T}(n - n' - n'', k - k' - k'') 
    & = \frac{2}{(2\pi)^{1+m}} \iint_{\mathbb{T}^{1+m}} e^{-i \left[\langle k - k' - k'', \theta \rangle + (n - n' - n'')x \right] }  d\theta dx.  \\ 
    & = \left\{
            \begin{aligned}
                & 2,  && (n, k) = (n'+n'', k'+k'')     \\
                & 0,  && (n, k) \neq (n'+n'', k'+k'')
            \end{aligned}
        \right.
\end{align*}
To estimate the associated bilinear operators, let $\hat{u}_1$ and $\hat{u}_2$ 
be supported within $\Lambda_N^{1+m}$. 
The kernel identity above yields
\[
    \mathcal{T}(\hat{u}_1, \hat{u}_2)(n,k) = 2 \sum_{(n',k') \in \Lambda_N^{1+m}}\hat{u}_1(n', k') \hat{u}_2(n - n', k - k').
\]
Applying the Cauchy–Schwarz inequality to the inner sum, we obtain
\begin{align*}
    \|\mathcal{T}(\hat{u}_1, \hat{u}_2)\|^2 
    & = 4 \sum_{(n,k)\in\mathcal{Z}} \left\| \sum_{(n',k') \in \Lambda_N^{1+m}} \hat{u}_1(n', k') \hat{u}_2(n - n', k - k' ) \right\|^2 \\
    & \leq 16 (2N+1)^{1+m} \sum_{(n,k)\in\mathcal{Z}} \sum_{(n',k') \in \Lambda_N^{1+m}} \|\hat{u}_1(n', k')\|^2 \|\hat{u}_2(n - n', k - k' )\|^2 \\
    & = 16 (2N+1)^{1+m} \|\hat{u}_1\|^2 \|\hat{u}_2\|^2, 
\end{align*}
where the last equality follows from the translation invariance of the lattice 
sum. 
Thus, on inputs supported in $\Lambda_N^{1+m}$, the bilinear operator norm 
satisfies
\[
    \left\| \frac{\partial^2 \hat{X}}{ \partial \hat{u}^2} \right\| \leq 4 (2N+1)^{\frac{1+m}{2}} < \infty.  
\]
Finally, these operators, $\partial^2 \hat{X}_q / \partial \hat{u}_p^2$, 
$\partial S/\partial \hat{u}_p$, and 
$\partial^2 \hat{X} / \partial \hat{u}_p^2$, are obtained from 
$\partial^2 \hat{X} / \partial \hat{u}^{2}$ by restricting its input and/or 
output indices to the corresponding resonant or non-resonant sets. 
Since these restrictions are implemented by orthogonal projections, their 
bilinear operator norms satisfy the same bound under the stated support 
condition. 
This establishes~\eqref{eqn: tensor} and completes the proof.

\subsection{Proof of~\Cref{prop: B-operator}}
\label{subsec: b-operator}

Recalling the definition of $B$ in~\eqref{eqn: operator-b-nlaw}, we have, for 
any $(n,k), (n', k') \notin \mathcal{S}$:
\[
    B \left( (n, k), (n', k') \right) = - \frac{\langle k, \omega' \rangle}{n^2} \cdot \frac{\partial \langle k, \omega'^{-1} \odot A^{-1}W_q\hat{X}_q \rangle}{\partial \hat{u}_p(n', k')}\hat{u}_p(n,k),
\]
Using the frequency estimate~\Cref{prop: frequency-drift} and applying the 
Cauchy–Schwarz inequality, we obtain
\[
    \left\| B \left( (n, k), (n', k') \right) \right\| \leq \frac{n_m^2(\omega_m + \sqrt{\varepsilon})}{\delta_0 (\omega_1 - \sqrt{\varepsilon})} \cdot \frac{\| k \|^2}{n^2} \|\hat{u}_p(n, k)\| \cdot \left\| \frac{\partial \hat{X}_q}{\partial \hat{u}_p(n', k') } \right\|.
\] 
The Gevrey decay of $\hat{u} \in \mathcal{K}(s)$, together with the decay 
estimates for the tangent operator established in~\Cref{prop: convolution}, 
therefore gives the Gevrey bound in~\eqref{eqn: b-operator-bound}.

To estimate the operator norm of $B$, we use the factorized entrywise bound 
above and extend the summation to the full product lattice 
$\mathcal{Z} \times \mathcal{Z}$, obtaining
\[
    \| B \| \leq \frac{n_m^2(\omega_m + \sqrt{\varepsilon})}{\delta_0 (\omega_1 - \sqrt{\varepsilon})} \cdot \sum_{(n, k) \in \mathcal{Z}} \frac{\| k \|^2}{n^2} \|\hat{u}_p(n, k)\| \cdot \sum_{(n', k') \in \mathcal{Z}} \left\| \frac{\partial \hat{X}_q}{\partial \hat{u}_p(n', k') } \right\|.
\]
Since $\hat{u} \in \mathcal{K}(s)$, the first sum converges and is uniformly 
bounded. 
The Gevrey decay estimates for the tangent operator established 
in~\Cref{subsec: convolution} imply the same conclusion for the second sum. 
These bounds establish~\eqref{eqn: b-bound} and complete the proof.

\section{Proofs of results in~\Cref{sec: implement}}
\label{sec: implement-app}

In this appendix, we provide the detailed proofs 
of~\Cref{lem: near-measure-initial} and~\Cref{thm: small-size} stated 
in~\Cref{sec: implement}.

\subsection{Proof of~\Cref{lem: near-measure-initial}}
\label{sec: near-measure-initial}

Recall that the frequency set $\Omega \subset \mathbb{R}^m$ is bounded, and set 
$d = \mathrm{diam}(\Omega) = \sup_{\omega'_1, \omega'_2 \in \Omega} \|\omega'_1 - \omega'_2\|_2$ 
and $R = \sup_{\omega' \in \Omega} \|\omega' \|_2$. 
For each non-zero temporal index $k \in \mathbb{Z}^m \setminus \{0\}$, the 
corresponding nearly-resonant strip lies between two parallel hyperplanes 
orthogonal to $k$. 
Estimating the measure of its intersection with $\Omega$ and summing over all 
nonzero $k$, we obtain the following bound for the tangential nearly-resonant 
set defined in~\eqref{eqn: tangential-nearly-resonant}: 
\begin{equation}
    \label{eqn: tangent-measure-estimate}
    \mathrm{mes}\left( \Omega_{0,M}^{\tau} \right) \leq \sum_{k \in \mathbb{Z}^m \setminus \{0\}} \frac{2 d^{m-1}}{\|k\|_2^{\tau+1}}. 
\end{equation}

We next estimate the measure of the single-mode nearly-resonant set 
$\Omega_{1,M}^{\tau}$ defined in~\eqref{eqn: single-nearly-resonant}. 
If the nearly-resonance condition holds for some $\omega' \in \Omega$, then 
$\left|\langle k, \omega' \rangle - \sqrt{n^4 +n^2} \right| < |n|$, which implies 
$n^2 < 2R\|k\|_2$. 
Hence, for each fixed $k$, the number of admissible spatial indices $n$ is 
bounded by $2 \sqrt{ 2R\|k\|_2 } + 1 \leq 3 \sqrt{ 2R\|k\|_2 }$. 
Combining this counting estimate with the strip-measure estimate yields
\begin{equation}
    \label{eqn: single-measure-estimate}
    \mathrm{mes}\left( \Omega_{1,M}^{\tau} \right) 
    \leq \sum_{k \in \mathbb{Z}^m} \frac{6 d^{m-1}  2R \|k\|_2}{(1 + \|k\|_2)^{\tau+1}} 
    \leq \sum_{k \in \mathbb{Z}^m} \frac{12 d^{m-1}R}{(1 + \|k\|_2)^{\tau}} .
\end{equation}

We now turn to the difference nearly-resonant sets, beginning with 
$\Omega_{2,M}^{\tau, +}$, defined in~\eqref{eqn: difference-nearly-resonant-2}.   
The nearly-resonance condition implies 
$\left|\langle k, \omega' \rangle - \sqrt{n_1^4 + n_1^2} - \sqrt{n_2^4 + n_2^2} \right| < |n_1| + |n_2|$, 
and therefore $n_1^2 + n_2^2 < 2R\|k\|_2$. 
For each fixed $k \in \mathbb{Z}^m$, each spatial index therefore has at most 
$2\sqrt{ 2R\|k\|_2 } + 1 \leq 3\sqrt{ 2R\|k\|_2 }$ admissible values. 
Consequently, the number of admissible pairs $(n_1, n_2)$ is at most 
$18R\|k\|_2$. 
Combining this count with the strip-measure estimate gives 
\begin{equation}
    \label{eqn: diff-measure-est-2}
    \mathrm{mes}\left( \Omega_{2,M}^{\tau,+} \right) 
    \leq \sum_{k \in \mathbb{Z}^m} \frac{ 72\sqrt{2}d^{m-1} (R\|k\|_2)^{\frac{3}{2}} }{(2 + \|k\|_2)^{\tau+1}} 
    \leq \sum_{k \in \mathbb{Z}^m} \frac{ 72\sqrt{2}d^{m-1} R^\frac{3}{2} }{(2 + \|k\|_2)^{\tau-\frac{1}{2}}}.
\end{equation}

Finally, consider $\Omega_{2,M}^{\tau, -}$ defined 
in~\eqref{eqn: difference-nearly-resonant-1}. 
The nearly-resonance condition reads 
\[
    \left| \langle k, \omega' \rangle - \sqrt{n_1^4 + n_1^2} + \sqrt{n_2^4 + n_2^2} \right| < \frac{|n_1| + |n_2|}{(2 + |k|)^{\tau}}. 
\]
By the inequality~\eqref{eqn: nlaw-freq-integer-estimate}, this implies
\begin{equation}
    \label{eqn: diff-ctrl}
    \left| \langle k, \omega' \rangle - n_1^2 + n_2^2 \right| < \frac{|n_1| + |n_2|}{(2 + \|k\|_2)^{\tau}} + \frac{1}{8n_1^2} + \frac{1}{8n_2^2}. 
\end{equation}
We use~\eqref{eqn: diff-ctrl} to estimate the number of admissible pairs 
$(n_1, n_2)$ that give rise to nearly resonance in~\eqref{eqn: difference-nearly-resonant-1}. 
Since $n_1, n_2 \neq 0$ and $|n_1| \neq |n_2|$, we have
\begin{equation}
    \label{eqn: square-diff}
    |n_1^2 - n_2^2| = \bigg| |n_1| - |n_2| \bigg| \cdot ( |n_1| + |n_2| ) \geq |n_1| + |n_2|,
\end{equation}
while~\eqref{eqn: diff-ctrl} implies 
\begin{equation}
    \label{eqn: diff-ctrl-1}
    |n_1^2 - n_2^2| \leq R \|k\|_2 + \frac{1}{4} + \frac{|n_1| + |n_2|}{(2+|k|)^\tau}.
\end{equation}
Combining~\eqref{eqn: square-diff} and~\eqref{eqn: diff-ctrl-1}, and using 
$\tau > m + 2 \geq 3$, we obtain
\begin{equation}
    \label{eqn: num-ctrl}
    |n_1| + |n_2| \leq |n_1^2 - n_2^2| \leq 2R\|k\|_2 + 1.
\end{equation}
Consequently, each of the spatial indices $n_1$ and $n_2$ has at most 
$4R\|k\|_2 + 3$ admissible integer values.
Summing over the admissible integer pairs, we bound the measure contribution 
in~\eqref{eqn: difference-nearly-resonant-1} for each fixed $k$ by
\[
    \frac{2d^{m-1}(2R\|k\|_2+1)(4R\|k\|_2+3)^2}{(2 + \|k\|_2)^{\tau+1}} \leq \frac{16d^{m-1}(\max\{2R, 1\})^3}{(2 + \|k\|_2)^{\tau-2}}.  
\]
Therefore, summing over all $k$, we obtain
\begin{equation}
    \label{eqn: diff-measure-est-1}
    \mathrm{mes}\left( \Omega_{2,M}^{\tau,-} \right) \leq \sum_{k \in \mathbb{Z}^m} \frac{16d^{m-1}(\max\{2R, 1\})^3}{(2 + \|k\|_2)^{\tau-2}}.
\end{equation}

Since $\tau > m + 2$, the lattice sums 
in~\eqref{eqn: tangent-measure-estimate}--\eqref{eqn: diff-measure-est-1} converge. 
Combining these estimates with the union 
representation~\eqref{eqn: total-nearly-resonant}, we obtain the desired 
bound~\eqref{eqn: near-measure-initial}. 
This completes the proof.

\subsection{Proof of~\Cref{thm: small-size}}
\label{sec: small-size-original}

Before proving~\Cref{thm: small-size}, we establish an auxiliary lemma showing 
that the Gevrey decay of matrix entries is preserved under operator multiplication. 
For a fixed exponent $0<s<1$, we say that an operator 
$A: \ell_2(\mathbb{Z}^{1+m}) \mapsto \ell_2(\mathbb{Z}^{1+m})$ has 
\textit{proper decay} if its entries satisfy: 
\[
    \sup_{ (n, k), (n', k') \in \mathbb{Z}^{1+m}} \left( \left|A\left( (n, k), (n', k') \right) \right| e^{ (|n-n'| + |k-k'|)^s } \right) \leq 1.
\]

\begin{lemma}
    \label{lem: app-3}
    Let $A_1$ and $A_2$ be two operators with proper decay. 
    Then there exists a constant $C'_2(m, s) > 0$ such that their product 
    $\mathcal{A} = A_1 A_2$ satisfies
    \begin{equation}    
        \label{eqn: convolution-mix}   
        \sup_{(n, k), (n', k') \in \mathbb{Z}^{1+m}} \left( \left|\mathcal{A}\left( (n, k), (n', k') \right) \right| e^{ (|n-n'| + |k-k'|)^s } \right) \leq C'_2(m, s).
    \end{equation}
    More generally, for any finite collection of operators 
    $\{ A_j \}_{j=1}^{\ell}$ with proper decay, there exists a constant 
    $C'_{\ell}(m, s) > 0$ such that the $\ell$-fold product 
    $\mathcal{A} = \prod_{j=1}^{\ell} A_j$ satisfies: 
    \begin{equation}    
        \label{eqn: p-convolution-mix}   
        \sup_{(n, k), (n', k') \in \mathbb{Z}^{1+m}} \left( \left|\mathcal{A}\left( (n, k), (n', k') \right) \right| e^{ (|n-n'| + |k-k'|)^s } \right) \leq C'_{\ell}(m, s).
    \end{equation}   
\end{lemma}

\begin{proof}
    Given that both $A_1$ and $A_2$ have proper decay, the entries of their 
    product satisfy
    \begin{align*}
       \left |\mathcal{A}\left( (n, k), (n', k') \right) \right | 
        & \leq \sum_{(n'', k'') \in \mathbb{Z}^{1+m}} \left| A_1\left( (n,k), (n'',k'') \right) \right| \left|A_2 \left((n'',k''), (n',k') \right) \right|  \\
        & \leq \sum_{(n'', k'') \in \mathbb{Z}^{1+m}} e^{ - (|n-n''| + |k-k''|)^s - (|n''-n'| + |k''-k'|)^s }  \\
        & \leq e^{ - (|n-n'| + |k-k'|)^s } \sum_{(n'', k'') \in \mathbb{Z}^{1+m}} e^{ - (2 - 2^s) \min\left\{ (|n-n''| + |k-k''|)^s, \; (|n''-n'| + |k''-k'|)^s \right\} },
    \end{align*}
    where the last inequality follows from~\Cref{lem: app-1}. 
    The final series converges, and its sum is bounded by a constant depending 
    only on $m$ and $s$; this establishes~\eqref{eqn: convolution-mix}. 
    Applying the same estimate successively yields~\eqref{eqn: p-convolution-mix} 
    by induction on $\ell$.
\end{proof}

We now complete the proof of~\Cref{thm: small-size}. 
The uniform lower bound~\eqref{eqn: uniform-lower-original-small} ensures that 
the diagonal operator $D_{n_0, N}$ is invertible. 
Thus, the restricted operator $L_{n_0, N}$ admits the factorization 
\[
    L_{n_0, N} = D_{n_0, N} + \varepsilon (S_{n_0, N} + B_{n_0, N}) = D_{n_0, N} \left[ I + \varepsilon D_{n_0, N}^{-1}(S_{n_0, N} + B_{n_0, N}) \right].
\]
By~\eqref{eqn: uniform-lower-original-small} together 
with~\Cref{prop: convolution} and~\Cref{prop: B-operator}, we have 
\[
    \varepsilon \| D_{n_0, N}^{-1}(S_{n_0, N} + B_{n_0, N}) \|_2 \leq \frac12.
\]
Consequently, $L_{n_0, N}$ is invertible and its inverse is given by the 
operator-norm convergent Neumann series
\begin{align}
    L_{n_0, N}^{-1} 
    & = \left[ \mathrm{Id} + \varepsilon D_{n_0, N}^{-1}(S_{n_0, N} + B_{n_0, N}) \right]^{-1} D_{n_0, N}^{-1}  \nonumber \\
    & = \left\{ \sum_{\ell=0}^{\infty} \left[ - \varepsilon D_{n_0, N}^{-1}(S_{n_0, N} + B_{n_0, N}) \right]^{\ell} \right\} D_{n_0, N}^{-1}.    \label{eqn: T-nu-expansion}
\end{align}
Since $\sum_{\ell =0}^{\infty} \|\varepsilon D_{n_0, N}^{-1}(S_{n_0, N} + B_{n_0, N})\|_2^{\ell} \leq 2$, 
it follows that $\|L_{n_0, N}^{-1}\|_2 \leq 2\|D_{n_0, N}^{-1} \|_2$. 
Together with~\eqref{eqn: uniform-lower-original-small}, this yields the norm 
estimate~\eqref{eqn: small-size-inverse}.

To establish the off-diagonal decay, we separate the $\ell=0$ term 
in~\eqref{eqn: T-nu-expansion} and rewrite the remaining series as
\[
    L_{n_0, N}^{-1} = D_{n_0, N}^{-1} - \left[ \sum_{\ell=0}^{\infty} \left( - \varepsilon D_{n_0, N}^{-1}(S_{n_0, N} + B_{n_0, N}) \right)^{\ell} \right] \left(\varepsilon D_{n_0, N}^{-1}(S_{n_0, N}+B_{n_0, N}) D_{n_0, N}^{-1}\right).
\]
The Gevrey decay estimates for $S$ and $B$ established 
in~\Cref{prop: convolution} and~\Cref{prop: B-operator}, combined 
with~\eqref{eqn: uniform-lower-original-small} and~\Cref{lem: app-3}, imply 
that, upon further reducing $\varepsilon$ if necessary, each summand satisfies 
\[
    \left| \left( - \varepsilon D_{n_0, N}^{-1} (S_{n_0, N} + B_{n_0, N}) \right)^{\ell} \left( (n, k), (n', k') \right) \right| \leq \left(\frac{1}{2}\right)^{\ell+1} e^{ - (|n-n'|+|k-k'|)^s },
\]
for any $\ell \geq 0$ and any 
$(n, k), (n', k') \in (n_0, 0) + \Lambda_N^{1+m}$. 
Since $D_{n_0, N}^{-1}$ is diagonal with the 
bound~\eqref{eqn: uniform-lower-original-small}, summing these bounds over 
$\ell \geq 0$ yields the off-diagonal decay 
estimate~\eqref{eqn: small-size-localization}. 
This completes the proof of~\Cref{thm: small-size}.

\section{Proofs of results in~\Cref{sec: iterative-scheme}}
\label{sec: completion-app}

In this appendix, we provide the detailed proof 
for~\Cref{lem: inverse-derivative}.

\subsection{Proof of~\Cref{lem: inverse-derivative}}
\label{subsec: derivative-app}

From~\Cref{prop: convolution} and~\Cref{prop: B-operator}, it follows that for 
any $(n, k) \neq (n', k')$, the off-diagonal entries of $\partial L_N$ satisfy
\[
    |\partial L_N \left( (n,k), (n', k') \right)| \leq e^{ - \left( |n-n'| + |k-k'| \right)^{s} }.
\]
In view of the operator identity in~\eqref{eqn: t-n-derivative-inverse}, each 
entry of $\partial L_{N}^{-1}((n, k), (n', k'))$ can be expanded as the double sum 
\begin{multline*}
    \partial L_{N}^{-1}\left( (n,k), (n', k') \right) \\ 
    = - \sum_{(n_1, k_1) \in \Lambda_N^{1+m}} \sum_{(n_2,k_2) \in \Lambda_N^{1+m}}L_{N}^{-1}\left( (n,k), (n_1, k_1) \right) \partial L_N\left( (n_1,k_1), (n_2, k_2) \right) L_{N}^{-1}\left( (n_2,k_2), (n', k') \right).
\end{multline*}
Applying the triangle inequality and inserting the bounds of~\Cref{lem: app-3}, 
we obtain: 
\begin{align*}
    & \left| \partial L_{N}^{-1}\left( (n,k), (n', k') \right) \right| \\ 
    & \leq \sum_{(n_1, k_1) \in \Lambda_N^{1+m}} \sum_{(n_2,k_2) \in \Lambda_N^{1+m}} |L_{N}^{-1}\left( (n,k), (n_1, k_1) \right)| \cdot | \partial L_N\left( (n_1,k_1), (n_2, k_2) \right)| \cdot |L_{N}^{-1}\left( (n_2,k_2), (n', k') \right)| \\
    & \leq \frac{2RmN^2 (2N+1)^{2m+2}}{\varepsilon_N^4} e^{ - \frac{\left( |n-n'| + |k - k'| - 2N^{\frac12} \right)^{s}}{2} } 
\end{align*}
Since the lattice sites are separated by $|n-n'| + |k-k'| \geq N^{\frac34}$, 
the prefactor is absorbed by the exponential, and the off-diagonal 
entries of $\partial L_{N}^{-1}$ satisfy the Gevrey 
bound~\eqref{eqn: inverse-restricted-entries}. 
This completes the proof.

\end{document}